\documentclass[reqno, 12pt]{amsart}
\usepackage[letterpaper,hmargin=1in,vmargin=1in]{geometry}
\usepackage{amsfonts, amsmath, amssymb, amsthm, verbatim, bbm, url, color}
\usepackage{graphicx}
\usepackage{enumerate}
\usepackage[all, 2cell]{xypic}

\usepackage{hmics-macros}

\usepackage[refpage]{nomencl}

\makenomenclature

\title{Harmonic Maps of Conic Surfaces With Cone Angles Less Than $2\pi$}
\author{Jesse Gell-Redman}

\begin{document}

\maketitle


\begin{abstract}
We prove the existence and uniqueness of harmonic maps
  in degree one homotopy classes of closed, orientable surfaces of
  positive genus, where the target has non-positive gauss curvature and
  conic points with cone angles less than $2\pi$.  For a homeomorphism $w$ of such a surface, we prove
  existence and uniqueness of minimizers in the homotopy class of $w$ relative to
  the inverse images of the cone points with cone angles less than or
  equal to $\pi$.   We
  show that such maps are homeomorphisms and that they depend smoothly on the
  target metric.  For fixed geometric data, the space of
  minimizers in relative degree one homotopy classes is a complex
  manifold of (complex) dimension equal to the number of cone points
  with cone angles less than or equal to $\pi$.

When the genus is zero, we prove the same relative minimization
provided there are at least three cone points of cone angle less than
or equal to $\pi$.

\end{abstract}

\section{Introduction} 
\label{cha:introduction}

The study of harmonic maps into singular spaces, initiated in \cite{gs}, has reached a refined state; beyond general existence and uniqueness theorems, there are regularity and compactness results in the presence of minimal regularity assumptions on the spaces involved, see, among many others, \cite{ks}, \cite{mese}, \cite{dm},  \cite{w2}, and \cite{ef}.

Particularly detailed results are available in the case of maps of surfaces.  In \cite{k}, Kuwert studied degree one harmonic maps of closed Riemann surfaces into flat, conic surfaces with cone angles bigger than $2\pi$.  He showed that the minimizing maps can be obtained as limits of diffeomorphisms, and that the inverse image under a degree one harmonic map of each point in the singular set is the union of a finite number of vertical arcs of the Hopf differential.  Away from this inverse image the map is a local diffeomorphism.  In \cite{mese}, Mese proved the same when the target is a metric space with bounded curvature, in particular when it is a conic surface with cone angles bigger than $2\pi$.

In this paper, we study two main topics: energy minimizing maps from Riemann surfaces
into conic surfaces with cone angles less than $2\pi$, and energy
minimizing maps from \textit{punctured} Riemann surfaces into conic
surfaces with cone angles less than or equal to $\pi$. We state our results somewhat sketchily at first, assuming for
simplicity's sake that there is just one cone point.  The energy
functional is conformally invariant with respect to the domain metric (see \S \ref{confinvsect}), so we
state all results in terms of conformal structures, $\mathfrak{c}$, on
the domain.  First, we have
\begin{thm} \label{mthm1}
Let $\Sigma$ be a closed surface of genus $> 0$, equipped with a
conformal structure $\mathfrak{c}$ and a Riemannian metric $G$, with a
conic point $p$ of cone angle less than $2\pi$ and non-positive Gauss
curvature away from $p$.  Let $\phi : \Sigma \lra \Sigma$ be a
homeomorphism.  Then there is a unique map $u: \Sigma \lra \Sigma$
which minimizes energy in the homotopy class of $\phi$.  This map
satisfies
	\begin{align*}
		u^{-1}(p) \mbox{ is a single point}
	\end{align*}
and $u : \Sigma - u^{-1}(p) \lra \Sigma - p$ is a diffeomorphism.
\end{thm}
That $u^{-1}(p)$ is a single point is consistent with the work of Hardt and Lin on maps into round cones in Euclidean space, see \cite{hl}.  Second, we have
\begin{thm}  \label{mthm2}
Let $\Sigma, \mathfrak{c},$ and $G$ be as in the previous theorem.  If
the cone angle at $p$ is less than or equal to $\pi$, then for each $q
\in \Sigma$ and each homeomorphism $\phi$ of $\Sigma$ with $\phi(q) =
p$, there is a unique map $u: \Sigma \lra \Sigma$ with $u(q) = p$
which minimizes energy in the rel$.$ $q$ homotopy class of $\phi$ (see
(\ref{relhtopyclass}) for the definition of relative homotopy class).
This map satisfies
	\begin{align*}
		u^{-1}(p) &= q
	\end{align*}
and $u : \Sigma - q \lra \Sigma - p$ is a diffeomorphism.
\end{thm}

This problem is motivated by the role of harmonic maps in
Teichm\"uller theory and recent work that extends classical
uniformization results to the case of conic metrics on punctured
surfaces.  The uniformization theorems for cone metrics of McOwen and
Troyanov, \cite{mc1}, \cite{mc2}, \cite{tr}, recent work by
Schumacher and Trapani, \cite{st}, \cite{st2}, and unpublished work of
Mazzeo and Weiss has shown that the role the hyperbolic
geometry plays in Teichm\"uller theory is played by conic metrics
in the punctured case.  In the unpunctured case, harmonic maps enter
the picture in the works of \cite{t}, \cite{w1}, \cite{w2}, and
\cite{w}, in the creation of various functionals and in the important
parametrization of Teichm\"uller space by holomorphic quadratic
differentials.  This paper lays the groundwork for the extension
of these results to the punctured case if the uniformizing metrics are
conic, and makes available --- in case the cone angles are less than $\pi$ --- harmonic maps of
punctured Riemann surfaces.

The proofs follow the method of continuity, and accordingly the paper
is divided into a portion with a perturbation result and a portion
with a convergence result.  In both we make frequent use of the fact
that minimizers as in both theorems solve a differential equation (see (\ref{hme})).  To be precise, let $\Mm_{conic}(p, \al)$ denote the space of smooth
metrics on $\Sigma - p$ with a cone point at $p$ of cone angle $2 \pi
\al$ (see \S \ref{metricssect}).  The `tension field' operator $\tau$
is a second order, quasi-linear, elliptic partial differential
operator, arising as the Euler-Lagrange equation for the energy
functional, which takes a triple $(u, g, G)$ of maps and metrics to a
vector field over $u$, denoted by $\tau(u, g, G)$.  For $g \in
\Mm_{conic}(q, \al)$, $ G \in
\Mm_{conic}(p, \al)$, in section \ref{uniquesect} we show that a
diffeomorphism $u : (\Sigma - q, g) \lra (\Sigma - p, G)$ subject to
certain conditions on the behavior near $q$ minimizes energy in its
rel$.$ $q$ homotopy class if $\tau(u, g, G) = 0$.

The perturbation result is an application of the Implicit Function
Theorem to the tension field operator.  Fix $g \in
\Mm_{conic}(q, \al)$, $ G \in
\Mm_{conic}(p, \al)$ and a minimizer $u$ as in Theorem \ref{mthm2}.  The proofs
rest mainly in finding the right space of perturbations of $u$, call
them $\mc{P}(u)$, and the right space of perturbations of $g$, call
them $\mc{M}^{*}(g) \subset \Mm_{conic}(q,\al)$, so that $\tau$ acting on $\mc{P}(u) \times
\mc{M}^{*}(g)$ has non-degenerate differential in the
$\mc{P}(u)$ direction at $(u, g)$.  There are two
not-quite-correct points in this last sentence.  First, to apply the
Implicit Function Theorem, one needs to work with Banach and not
Fr\'echet manifolds like $\Mm_{conic}(p, \al)$; for a precise definition
of the spaces we use see section \ref{setup}.  Second, the map $\tau$
actually takes values in a bundle, specifically the bundle $\EB \lra
\mc{P}(u) \times \mc{M}^{*}(g)$ whose fiber over $(\wt{u},
\wt{g})$ is (some Banach space of) sections of $\wt{u}^*\mbox{T}\Sigma$, so in fact we do not show that $\tau$
has surjective differential but rather that $\tau$ is transversal to
the zero section of $\EB$.


If $z$ denotes conformal coordinates near $q$, the linear operator $L$
can be written $L = \absv{z}^{2\al} \wt{L}$, where $\wt{L}$ falls into
a broad class of linear operators known as elliptic $b$-differential
operators, pioneered and elaborated by R. Melrose, and used
subsequently in countless settings (\S \ref{bcalcsect}).  For detailed
properties of $b$-operators, see \cite{me} and \cite{meme}.  The main
difficulty we encounter is that $L$ is degenerate on a natural space
of perturbations.  Following the example of previous authors,
including \cite{mp}, we supplement the domain of $\tau$ with
`geometric' perturbations; in our case, we add a space $\mc{C}$ of
diffeomorphisms of the domain which look like conformal dilations, rotations, and
translations near the cone point  (\S \ref{setup}).

As we will discuss in section \ref{bcalcsect}, the natural domains for
$L$ are weighted Banach spaces $r^c \Xx^{2, \gamma}_{b}$, which for
the moment should be thought of as vector fields vanishing to order
$r^c$ near $q$.  In particular $L$ acts from $r^{1 - \e} \Xx^{2, \gamma}_{b}$ to $r^{1
  - \e - 2\alpha} \Xx^{2,\gamma}_{b}$, and is Fredholm for sufficiently small $\e > 0$.
Let $\mc{K}$ denote its cokernel (see \eqref{cokerminuse}). The leading order behavior of a
vector field in $\mc{K}$ near the inverse image of the cone point is
characterized by the homogeneous solutions of a related `indicial'
equation, c.f. Lemma \ref{stfsjfs}.  A dichotomy between the behavior
of elements in $\mc{K}$ near cone points of cone angle less than $\pi$
and cone angle greater than $\pi$ arises. 
\begin{lemma} \label{isosjf}
Let $\psi \in\Kk$, and suppose that $G$ has only one cone point, $p$, with $u^{-1}(p) = q$.  If the cone angle $2 \pi \al$ is bigger than $\pi$, then near $q$
	\begin{align} 
		\psi(z) &= w + \frac{ \ov{a} }{1 - \al} \absv{z}^{2(1 - \al)} + \mc{O}(\absv{z}^{2(1 - \al) + \delta})
	\end{align}
for some $w, a \in \C$ and $\delta > 0$.  If the cone angle is less than $\pi$, then near $q$
	\begin{align} 
		\psi(z) &= \mu z + \mc{O}(\absv{z}^{1 + \delta})
	\end{align}
for some $\mu \in \C$ and $\delta > 0$.  
\end{lemma}
\noindent This lemma is central to the proof of the main theorems,
since an accurate appraisal of the cokernel is needed to show
that the geometric perturbations described above are sufficient to
give a surjective problem.

To prove energy minimization, both in relative and absolute homotopy
classes, we use an argument from \cite{ch}, in which the authors prove
uniqueness for harmonic maps of surfaces with genus bigger than 1.
They show that the pullback of the target metric via a harmonic map
can be written as a sum of a two metrics, one conformal to the domain
metric and the other with negative curvature, thereby decomposing the
energy functional into a sum of two functionals which the harmonic
map jointly minimizes.  In section \ref{uniquesect}, we adapt this
argument to prove uniqueness in the conic setting.

For fixed domain and target structures $\cc$ and $G$, if the cone
point of $G$ has cone angle less than $\pi$, the harmonic maps in Theorem \ref{mthm2} from $(\Sigma, \cc)$ to $(\Sigma, G)$ are parametrized by points $q$ and rel$.$ $q$ homotopy classes of diffeomorphisms taking $q$ to $p$.  Denote this space by  $\mc{H}arm_{\cc, G}$. Given any diffeomorphism $\phi$ of $\Sigma$, let $[\phi]_{rel.} = [\phi; \phi^{-1}(p)]$, and denote the corresponding element of $\mc{H}arm_{\cc, G}$ 
by $u_{[\phi]_{rel.}}$.  One of our main results is a formula for the gradient of the energy functional on $\mc{H}arm_{\cc, G}$ in terms of the Hopf differential $\Phi(u_{[\phi]_{rel.}})$ of $u_{[\phi]_{rel.}}$ (see section \ref{polessect} for a definition of the Hopf differential).  It turns out that $\Phi(u_{[\phi]_{rel.}})$ is holomorphic on $\Sigma - q$ with at most a simple pole at $q$.  Given a path $u_t$ in $\mc{H}arm_{\cc, G}$ with $u_0 = u_{[\phi]_{rel.}}$, and writing $J := \ddt u_t$, the gradient is given by 
	\begin{align} \label{grad}
          \ddt E(u_t) &= \Re 2\pi i \Res \rvert_{q} \iota_J \Phi(u_{[\phi]_{rel.}}),
	\end{align}
where $\iota$ is contraction.  The one form $\iota_J
\Phi(u_{[\phi]_{rel.}})$ is not holomorphic, but admits a residue
nonetheless.  By Theorem \ref{mthm1} there is a unique choice
$[\ov{\phi}]_{rel.}$ such that the corresponding solution $\ov{u}$ is
an absolute minimum of energy in the free homotopy class
$[\ov{\phi}]$, and we use (\ref{grad}) to prove that $\ov{u}$ is the
unique element in $\mc{H}arm_{\cc, G}$ for which $\Phi(\ov{u})$
extends smoothly to all of $\Sigma$ (\S \ref{pertsection1}).  In other
words, $\ov{u}$ is the unique critical point of $E : \mc{H}arm_{\cc,
  G} \lra \R$.  We refer to $\ov{u}$ as an \textit{absolute
  minimizer} to distinguish it from the other \textit{relative
  minimizers} in $\mc{H}arm_{\cc, G}$.

As a corollary to (\ref{grad}), we prove the following formula for the Hessian of energy at $\ov{u}$.  Given a path $u_t$ through $u_0 = \ov{u}$ with derivative $J = \ddt u_t$, we can define the Hopf differential of $J$ by $\Phi(J) = \ddt \Phi(u_t)$.  We have
	\begin{align} \label{ihess}
		\ddtt E(u_t) &= \Re 2\pi i \Res \rvert_{\ov{q}} \iota_J \Phi(J)
	\end{align}
We use (\ref{ihess}) to prove the non-degeneracy of the linearized
residue map near $\ov{u}$ (\S \ref{pertsection1}), and use this to show that in the cone angle less than $\pi$
case one can perturb in the direction of absolute minimizers as the
geometric data varies.

In the closedness portion (section \ref{closedsec}), we adapt standard
methods for proving regularity of minimizing maps to the conic
setting.  First, we prove that the sup norm of the energy density of a
minimizing map is controlled by the geometric data (see Proposition
\ref{loclip}).  This uses a standard application of the theorems of
DiGiorgi-Nash-Moser and an extension of a Harnack inequality from
\cite{he}.  It is noteworthy that if a conic metric is chosen on the
domain, with cone point at the inverse image of the cone point of a
the target via a
minimizer from either of the main theorems, then the energy density is
bounded from both above and below if and only if the cone angles are
equal, and we work with such conformal metrics on the domain in what
follows. 

Control of the energy density near the (inverse image of)
the cone points is insufficient, and to obtain stronger estimates
we proceed by contradiction, employing a
rescaling argument, and the elliptic regularity of $b$-differential operators, to produce a minimizing map of the standard round
cone.  We also classify such maps, and the map produced by rescaling
is not among them; thus the desired bounds hold near the cone points
(see Proposition \ref{closed}).

This work was completed under the supervision of Rafe Mazzeo for my
thesis at Stanford university, and I would like to thank him for many helpful
discussions, ideas, and suggestions.  I would also like to thank
Andras Vasy, Leon Simon and Dean Baskin for their help and
encouragement, and Chikako Mese for a helpful email discussion.  I am
also deeply indebted to the NSF for the financial support I have
received through Rafe Mazzeo's research grant.

\section{Setup and an example} \label{example}

In this section, we give the precise statements of the theorems and outline the method of proof. 

Let $\Sigma$ be a closed surface, $\pp = \set{ p_1, \dots, p_k }
\subset \Sigma$ a collection of $k$ distinct points, and set
\begin{equation} 
\Sp :=\Sigma - \pp.\label{eq:Sp}
\end{equation}
Given a smooth metric $G$ on $\Sp$, let $u : \Sigma \lra \Sigma$ be a continuous map so that $u^{-1}(p)$ is a single point for all $p \in \pp$.  Let $$\pp' = u^{-1}(\pp)$$ and assume that 
	\begin{align} \label{firstmap}
	u: \Sigma_{\pp'} \lra \Sp
	\end{align}
Is smooth.  Let $g$ be a smooth metric on $\Sigma_{\pp'}$.  The
differential $du$ is a section of $\mbox{T}^*\Sigma_{\pp'} \otimes
u^*\mbox{T}\Sigma_{\pp}$, which we endow with the metric $g \otimes
u^* G$.  We define the Dirichlet energy by
	\begin{equation} 
	\begin{split}
	\label{energy}
		E(u, g, G) &= \frac{1}{2} \int_{\Spr} \norm[{g \otimes u^*G}]{du}^2 dVol_g 
		= \frac{1}{2}  \int_{\Spr} g^{ij} G_{\al \bb} \frac{\partial u^\al}{\partial x^i} \frac{\partial u^\bb}{\partial x^j} dVol_g. 
	\end{split}
	\end{equation}
As we will see in (\ref{energyinv}) below, the value of the energy functional at such a triple depends only on the conformal class of $g$.  

\nomenclature[E]{$E(u, g, G)$}{The energy of a map $u: (\Sigma, g)
  \lra (\Sigma, G)$}
\nomenclature[S1]{$\Sigma$}{a closed, orientable, smooth surface}
\nomenclature[S2]{$\pp$}{a finite subset of $\Sigma$}
\nomenclature[S3]{$\Sigma_{\mathfrak{p}}$}{$\Sigma - \mathfrak{p}$}

Assume for the moment that $G$ is smooth on all of $\Sigma$.  In this
case critical points of this functional are smooth \cite{si} and
satisfy that the \textit{tension field}
 \begin{equation} \label{hme}
   \tau(u, g, G) : =\Tr \nabla_{(u,g,G)} (du) = \Delta_g u^\g + { ^ G {\Gamma}}^\g_{\al \bb}\la d u^\al, d u^\bb \ra_g
 \end{equation}
is identically zero, where $\nabla_{(u,g,G)}$ is the Levi-Cevita connection on
$\mbox{T}^*\Si \otimes u^{*}\mbox{T}\Si$ with metric $g^{-1}
\otimes u^* G$, so $\Tr \nabla du \in \Gamma(u^* \mbox{T}\Si)$.
Here $\Gamma(B)$ denotes the space of smooth sections of a vector
bundle $B \lra \Si$.  In particular, $\tau(u, g, G)$ is a vector field
over $u$, and it is minus the gradient of the energy
functional in the following sense \cite{el}.
	\begin{lemma}[First Variation of Energy] \label{frstvarbdry}
		Let $(M, g)$ and $(N, \wt{g})$ be smooth Riemannian manifolds, possibly with boundary, and let $u : (M, g) \lra (N, \wt{g})$ be a $C^2$ map.  If $u_t$ is a variation of $C^2$ maps through $u_0 = u$ and $ \eval{\frac{d}{dt}}_{t = 0} u_t = \psi$, then
		\begin{equation} \label{frstvar0}
			\begin{split}
		\eval{\frac{d}{dt}}_{t = 0}E(u_t, g, \wt{g}) &= - \int_M \la \tau(u, g, \wt{g}) , \psi \ra_{u^*\wt{g}} dVol_{g}  + \int_{\p M} \la u_* \p_\nu , \psi \ra_{u^*{\wt{g}}} ds,
			\end{split}
		\end{equation}
where $\p_{\nu}$ is the outward pointing normal to $\p M$ and $ds$ is
the area form.
	\end{lemma}

\nomenclature[T]{$\tau(u, g, G)$}{the tension field of $u:(\Sigma, g)
  \lra (\Sigma, G)$}

\nomenclature[G]{$\Gamma(B)$}{the sections of a bundle $B$.}

Returning to the case of our surface $\Sigma$, if $G$ is instead only smooth on $\Sp$, the map $u$ in (\ref{firstmap}) still has a tension field 
 \begin{equation} \label{pointedtension}
		\tau(u, g, G) \in u^*\mbox{T}\Sp.
 \end{equation}
We will study the energy minimizing problem by finding zeros of $\tau$.

\subsection{Example: Dirichlet problem for standard\\ cones.}\label{examplesection}

\nomenclature[Ca]{$C_{\alpha}$}{the standard flat cone of cone angle
  $2\pi\alpha$}
\nomenclature[Cag]{$g_{\alpha}$}{metric on the standard flat cone}

The standard cone of cone angle $2 \pi \al$ is simply the sector in $\R^2$ of angle $2 \pi \al$ with its boundary rays identified.  Thus in polar coordinates $(\wt{r},  \wt{\theta})$ we can write this as a quotient $ \R^+ \times [0, 2\pi \al] / \lp (\wt{r} , 0) \sim (\wt{r} , 2 \pi \al) \rp$, with the $g_\al := d\wt{r}^2 +  \wt{r}^2  d\wt{\theta}^2 $.
Let $\wt{z} = \wt{r} e^{i  \wt{\theta} }$, and set 
	\begin{align} \label{wedgemap}
		\wt{z} = \frac{1}{\al}z^\al
	\end{align}
Then, in terms of $z$,
	\begin{align} \label{roundmetric}
		g_\al = \absv{z}^{2(\al - 1)} \absv{ dz }^2,
	\end{align}
and  this expression is valid for all $ z \in \C^*$.  We denote this space by
 \begin{equation} \label{stfltcone}\begin{split}
		C_\al &:= (\C, g_\al) \\
		C_\al^* &:= C_\al - \set{0}
\end{split}\end{equation}
Also let $D \subset \C$ be the standard disc of radius one, and set $D^* := D - \set{0} $.
We will discuss the Dirichlet problem for harmonic maps from $D$ to $C_\al$.  In this case, see (\ref{hmec}), the tension field operator for a smooth map $u : D^* \lra C_\al^*$ is
 \begin{equation*}
		\tau(u) = u_{z \ov{z}} + \frac{\al - 1}{u} u_z u_{\ov{z}}.
\end{equation*}
Therefore the identity map 
	\begin{align*}
		id : D^* &\lra C_\al \\
		z &\longmapsto z
	\end{align*}
is harmonic on $D^*$.  Given a map $\phi : \p D \lra C_\al$ near to $id \rvert_{\p D}$, we would like to find a harmonic map
 \begin{equation} \label{Dirprob}\begin{split}
		u : D &\lra C_\al \\
		\tau(u) &= 0  \mbox{ on } D - u^{-1}(0)\\
		u \rvert_{\p D} &=  \phi
\end{split}\end{equation}
Let $z = r e^{i \theta}$.  Initially, we consider $\tau$ acting on maps of the form
 \begin{equation} \label{iiuform}\begin{split}
		u(z) &= 
		z + v(z) \\
		v &\in r^{1 + \e} C^{2, \g}_b(C_\al(1)),
\end{split}\end{equation}		
where, given $v: D \lra \C$,
	\begin{align} \label{ic2gb}
		\begin{array}{c} 
		v \in r^{c} C^{2,\g}_b(D)
		\end{array} 
		&\iff 
		\begin{array}{c} 
		\mbox{ $r^{-c} v$ has uniformly bounded $C^{2,\g}$ norm on balls }\\
		\mbox{ of uniform size with respect to the rescaled} \\
		\mbox{ metric $  g_\al / r^{2\al} = \frac{dr^2}{r^2} +  d\theta^2$. }
		\end{array} 
	\end{align}
We will describe this space more precisely in section \ref{formsect}, but we mention now that $v$ as in (\ref{iiuform}) satisfies $\absv{v(z)} = \mc{O}(r^{1 + \e})$.  We refer to the space of $u$ as in (\ref{iiuform}) by $\Bb^{1 + \e}$.  Let $\mc{B}^{1 + \e}_0 \subset \mc{B}^{1 + \e}$ consist of $u  $ with $u \rvert_{\p D} = id \rvert_{\p D}$, i.e.\ those $u$ whose $v$ in (\ref{iiuform}) satisfy $v \rvert_{\p D}(z) \equiv 0.$  Near $id$ we can write $\Bb^{1 + \e}$ as a product of $\Bb^{1 + \e}_0$ and
 \begin{equation*}
		\mc{E} := \set{ \phi : \p D \lra C_\al : \norm[2,\g]{ \phi - id} < \e }
\end{equation*}
This can be done for example by picking a smooth radial cutoff function $\chi$ which is $1$ near $\p D$ and defining $\mc{E} \lra \mc{B}^{1 + \e}$  be $\phi \longmapsto \phi(\theta) \chi(r)$; the addition map $\mc{E} \times \mc{B}^{1 + \e}_0 \lra \mc{B}^{1 + \e}$ is a local diffeomorphism of Banach manifolds near $(id, id)$.

The strategy now is to study $\tau$ acting on the space $\Bb^{1 + \e} \sim  \Bb^{1 + \e}_0 \times \mc{E}$.  If the derivative of $\tau$ in the $ \Bb^{1 + \e}_0$ direction were non-degenerate at $id$, then the zero set of $\tau$ near $id$ would be a smooth graph over $\mc{E}$, \textit{but this turns out to be false}.  To fix this, we augment the domain of $\tau$ as follows.  Define the spaces of conformal dilations and rotations
	\begin{align*}
		\Vv &= \set{M_\lambda(z) =  \lambda z : \lambda \in \C } \\
		\mc{T} &=  \set{T_w(z) =   z - w : w \in \C } 
	\end{align*}
and define two $2-$dimensional spaces
	\begin{equation} \label{ilocconf}
		\begin{split}
		\Vv_0 &= \set{ \begin{array}{c}
		\mbox{any family $\wt{M}_\lambda$ parametrized by $\lambda$ near $1 \in \C $} \\
		\mbox{ $\wt{M}_\lambda = M_\lambda$ near $0$ and vanishes on $\p D$}
		\end{array}
		} \\
		\mc{T}_0 &= \set{ \begin{array}{c}
		\mbox{any family $\wt{T}_w$ parametrized by $w$ near $0 \in \C $} \\
		\mbox{ $\wt{T}_w = T_w$ near $0$ and vanishes on $\p D$}
		\end{array}
		}
		\end{split}
	\end{equation}
The dichotomy between cone angles less than $\pi$ and between $\pi$ and $2\pi$ now enters.  We will show in \S \ref{pertsection1} that for $\e > 0$ sufficiently small
 \begin{equation} \label{ilins}
		D_{id} \tau: T \lp \Bb^{1 + \e }_0 \circ \Vv_0 \rp \lra r^{- 1 + \e} C^{0, \g}_b(D) \mbox{ is an isomorphism if } 0 < 2 \pi\al < \pi 
 \end{equation}
while
	\begin{align}\label{ilinb}
		D_{id} \tau  : T \lp \Bb^{1 + \e }_0 \circ \Vv_0 \circ \mc{T}_0 \rp \lra r^{-1 + \e } C^{0, \g}_b(D) \mbox{ is an isomorphism if } 2 \pi > 2 \pi \al > \pi
	\end{align}
These two facts have obvious (and distinct) implications about solving
(\ref{Dirprob}).  In the former case, for boundary values $\phi$ near
$id$ in $C^{2, \g}$ we can find solutions of the form $u(z) = \lambda z + v(z)$
while in the latter we must `move the cone point,' so the solutions
are in the form $u(z) = \lambda( z - w) + v(z - w)$.


As we will see in section \ref{confinvsect}, the precomposition of a
harmonic map $u$ with a conformal map is harmonic.  Define $C_\al(1)= C_\al \cap \set{z : \absv{z} \leq 1}$.
Clearly $C_\al(1)$ is conformally equivalent to $D$, so for the moment we think of $id$ not as a map from $D$ to $C_\al$ but from $C_\al(1)$ to $C_\al$, and applying the inverse of (\ref{wedgemap}) to both spaces, think of $id$ as the identity map on the wedge $\set{ r \leq 1, 0 \leq \theta  \leq 2 \pi \al }$. (For simplicity, we have dropped the tildes from the notation.)  We can think of a boundary map, $\phi$ near $id$ as a map from the arc $\set{ r = 1 , 0 \leq \theta  \leq 2 \pi \al }$ into $\C$, satisfying the condition that
	\begin{align} \label{condition}
		\phi(e^{2\pi \al i } ) = e^{2 \pi \al i} \phi(1).
	\end{align}
We decompose $\phi$ in terms of the eigenfunctions of $\p_\theta^2$
which satisfy (\ref{condition}), i.e.\ we write $\phi(e^{i \theta}) = \sum_{j \in \Z} a_j e^{i (1 + \frac{j}{\al})\theta}$.
Assuming convergence, these are the values on the arc of a the sector of the harmonic map
 \begin{equation} \label{powerseries}
		u(z) = \sum_{j \geq 0} a_j z^{1 + \frac{j}{\al}} + \sum_{j < 0} a_j \ov{z}^{- 1 + \frac{j}{\al}}.
 \end{equation}
(In the flat metric $dr^2 + r^2 d\theta^2$ on the sector, the tension field of $u$ is simply $\Delta u$, so by the decomposition $\Delta = 4 \p_z \p_{\ov{z}}$, the sum of a conformal and an anti-conformal function is harmonic.)


Consider the coefficient $a_{-1}$.  If $\al > 1/2$, the term $a_{-1}
\ov{z}^{- 1 + \frac{1}{\al}}$ dominates $a_0 z$ as $z \to 0$, and it
is easy to check that a power series as in (\ref{powerseries})
sufficiently close to $id$ on the boundary passes to a map of the
cone if and only if $a_{-1} = 0$.  It is exactly the deformations
$\Tt$ that eliminate the $a_{-1}$ and give actual mappings of the
cone.  When $\al < 1/2$, any power series as in (\ref{powerseries})
gives a map on the cone, so for any sufficiently regular boundary data
near $id$, there is a harmonic map of $D^*$ with that boundary data.  

To go deeper, as we show in section \ref{uniquesect}, the map $u$ is
minimizing among all maps with its boundary values if and only if the
residue of the Hopf differential of $u$ is zero,  and a simple
computation shows that (in the case under consideration,)
 \begin{equation*}
		\Res \Phi(u) = \ov{a_{-1}} \lp -1  + \frac{1}{\al} \rp.
\end{equation*}
That is, \textit{if $\al < 1/2$, $a_{-1}$ is the obstruction to having a minimizer (w.r.t$.$ the boundary values) on all of $D$, not just $D^*$.}
Thus it is natural, in the case $\al < 1/2$, to ask if we can solve the augmented Dirichlet problem
	\begin{equation} \label{Dirproba}
	\begin{split}
		u : D &\lra C_\al \\
		\tau(u) &= 0  \mbox{ on } D - u^{-1}(0)\\
		a_{-1} &= 0 \\
		u \rvert_{\p D} &=  \phi
	\end{split}.
	\end{equation}
In fact we can.  By (\ref{ilins}) there is a graph of solutions to (\ref{Dirprob}) over $\Tt_0 \times \mc{E}$.  The solutions lying over $\Tt_0 \times \set{id}$ solve (\ref{Dirprob}) with identity boundary value.  Call this space $\mc{S}$.  In section \ref{pertsection2}, we show that the map from $\mc{S}$ to $a_{-1}$ has non-degenerate differential.  Since it is a map of two dimensional vector spaces, it is an isomorphism, and we get a smooth graph of solutions to (\ref{Dirproba}) over $\Tt_0 \times \mc{E}$, see Figure \ref{res2boundary}.

\begin{figure}[htbp]
\begin{center}
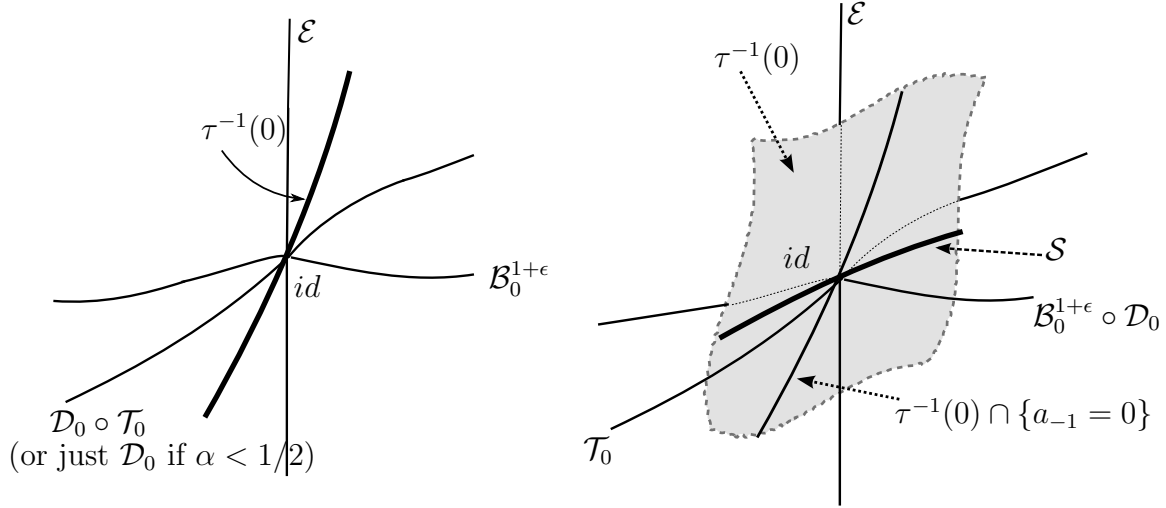
\caption{Solutions to the augmented equation (\ref{Dirproba}) are found by applying the inverse function theorem to the map from $\mc{S}$ to $a_{-1}$.}
\label{res2boundary}
\end{center}
\end{figure}

\subsection{Conic metrics} \label{metricssect}

The simplest example of a conic metric is that of the standard cone $C_{\al}$ in (\ref{roundmetric}).  This motivates the following definition.  Given $ {\al_j} \in \R_+$ and ${\nu_j} > 0$, a Riemannian metric $G$ on $\Sp$ is said to have a conic singularity at $p_{j} \in \pp$ with cone angle $2 \pi {\al_j} $ and type ${\nu_j}$ if there are conformal coordinates $z$ centered at $p_{j}$ such that
 \begin{equation}	 \label{conemet}\begin{split}
		G &= c e^{2 {\lambda_j}} \absv{ z }^{2({\al_j} - 1)} \absv{ dz }^2 \\
		{\lambda_j} = {\lambda_j}(z) &\in r^{\nu_j} C^{k, \g}_b(D(\sigma)) \\
		c &> 0.
\end{split}\end{equation}
Here $D(\sigma) = \set{ z : \absv{z} < \sigma }$.  Obviously, in any
other centered conformal coordinates $w = w(z)$, the metric will take
the same form.  For convenience, we single out those coordinates for which $c = 1$, and we refer to these as \textit{normalized} conformal coordinates.
Given $\Aa = ( \al_1, \dots, \al_k) \subset (0, 1)^k$ and $\nu =
(\nu_{1}, \dots, \nu_{k}) \in \R_{+}^{k}$, we define
	\begin{align} \label{conemetric1}
		\Mm_{k, \g, \nu}(\pp, \Aa)&= 
		\left\{
		\begin{matrix} \mbox{ $C^{k,\g}_{loc}$ metrics on $\Sp$ with cone points $\pp$} \\
		\mbox{ and cone angles $2 \pi \Aa$ such that in (\ref{conemet}) }  \\
		 \mbox{ we have $\lambda_{j} \in r^{\nu_{j}} C^{k, \g}_b(D(\sigma))$ for some $\sigma$.} 
		\end{matrix}
		\right\}
	\end{align}
\nomenclature[D]{$D(R)$}{the disc in $\C$ of radius $R$}
\nomenclature[A]{$\Aa$}{a point in $(0, 1)^{k}$}
\nomenclature[M1]{$\Mm_{k, \g, \nu}(\pp, \Aa)$}{is the space of conic metrics with
  cone points at $\pp$ of cone angles $2\pi\Aa$}

\subsection{The form of the minimizers} \label{formsect}
We will look for harmonic maps which have specified behavior near the
inverse image of the cone points.  By the uniqueness theorem in
section \ref{uniquesect}, all minimizers have this behavior, but we
begin by stating it as an assumption.
\begin{form} \label{uform}
We say that $u: (\Sigma, g) \lra (\Sigma, G)$ is in Form \ref{uform}
(with respect to $g$ and $G$) if 
\begin{enumerate}
	\item $ u $ is a homeomorphism of $\Sigma$ homotopic to the
          identity.  
	
	\item If $\pp' = u^{-1}(\pp)$, $u : \Sigma_{\pp'} \lra \Sp$ is a $C^{2, \gamma}_{loc}$
          diffeomorphism (see \ref{eq:Sp}).
	
	\item For each $p \in \pp$, if $z$ is a conformal coordinate
          around $u^{-1}(p)$ w.r.t$.$ $g$ and by abuse of notation $u$ is a conformal coordinate around $p$ w.r.t$.$ $G$, then
 \begin{equation*}
			u(z)  =  \lambda z  + v(z)
\end{equation*}
	where $\lambda \in \C^*$ and $v \in r^{1 + \e}
        C^{2,\g}_b(D(R))$ for some sufficiently small $\e  > 0$.
\end{enumerate}
\end{form}
\noindent In words, in normal coordinates near $p$ and $u^{-1}(p)$, $u$ is a
dilation composed with a rotation to leading order.

\begin{remark} \label{thm:idremark}
  When working with harmonic diffeomorphisms of surfaces it is often
  convenient to pull back the target metric, i.e.\ given $G \in \Mm_{2,
    \g, \nu}(\pp, \Aa)$ and
  \begin{equation}
    \label{eq:3}
    u : (\Sigma_{\pp'}, g) \lra (\Sp, G)
  \end{equation}
  harmonic, to consider instead
  \begin{equation}
    \label{eq:3}
    id : (\Sigma_{\pp'}, g) \lra (\Sigma_{\pp'}, u^*G).
  \end{equation}
  We will often make this assumption below.  It should be noted that
  it is exactly Form \ref{uform} which guarantees that the pullback
  metric $u^*G$ is a member of $\Mm_{2, \g, \nu}(\pp', \Aa)$.
\end{remark}

The space $C^{k, \g}_b(D(\sigma))$ was defined in the above as the space of functions with $C^{k, \g}$ norm uniformly bounded on balls of uniform size with respect to the rescaled metric $g / (e^{2 \mu} \absv{z}^{2 \al}) = \frac{ dr^2 + r^2 d\theta^2}{r^2}$, but here we give an alternative characterization that is easier to work with.  Given $f: D(R) \lra \C$, let
 \begin{equation} \label{c2gb0}
		\norm[C^{0, \g}_b(D(R))]{f} = \sup_{0 < \absv{z} \leq R} \absv{f} + \sup_{0 < \absv{z}, \absv{z'} \leq R} \frac{\absv{f(z) - f(z')}}{\absv{ \theta - \theta'}^\g +\frac{\absv{ r - r'}^\g}{\absv{ r +r'}^\g} }.
 \end{equation}
(Here, as always, $z = r e^{i \theta}, z' = r' e^{i \theta'}$.)  Then 
 \begin{equation} \label{c2gb}
		f \in C^{k , \g}_b(D(R)) \iff \norm[0 , \g,
                D(R)]{(r\p_r)^i \p^j_\theta f} < \infty \mbox{ for all } i + j \leq k.
 \end{equation}
Finally, define the weighted H\"older spaces by
 \begin{equation} \label{c2gb2}  \begin{split}
	f \in r^c C^{k, \g}_b(D(R)) &\iff r^{-c} f \in C^{k, \g}_b(D(R)) \\
	\norm[r^c C^{k, \g}_b(D(R))]{f} &= \norm[C^{k,
          \g}_b(D(R))]{r^{-c}f} := \sum_{i + j \leq k}\norm[0 , \g,
                D(R)]{r^{-c}(r\p_r)^i \p^j_\theta f}
\end{split}\end{equation}
\nomenclature[C2gb]{$r^{c}C^{2, \gamma}_{b}$}{weighted $b$-H\"older spaces}

\subsection{Scalar curvature} \label{scalarcurvature}

Let $\kappa_G$ denote the gauss curvature of $G \in \Mm_{2, \g, \nu}(\pp, \Aa)$, and near $p \in \pp$, write $G = \rho \absv{dz}^2$ with $\rho = e^{2 \lambda} \absv{ z }^{2(\al - 1)}$.  There is a simple sufficient condition on $\nu$ which guarantees that $\absv{\kappa_G} \leq c < \infty$; one simply computes
	\begin{equation} \label{eq:scalar}
          \begin{split}
		\kappa_G &= - \frac{1}{\rho} \p_z \p_{\ov{z}} \log
                \rho^2 = \frac{- 4}{e^{2 \lambda} \absv{ z }^{2\al}}   \lp z\p_z  \rp \lp \ov{z} \p_{\ov{z}} \rp \mu.
              \end{split}
	\end{equation}
Using $z \p_z = \frac{1}{2} \lp r \p_r - i\p_\theta \rp$
we see that if $\mu \in r^{2\al}C^{2, \g}_b$, i.e.\ if $G$ is type $\nu
= 2\al$, then $\kappa_G$ is a bounded function, and we will always
make this assumption.

\nomenclature[k]{$\kappa_{h}$}{scalar curvature of the metric $h$}

It will be important to have a somewhat stronger assumption about the
metrics.  We need the following 
	\begin{definition}\label{thm:phg}
		A function $f : D(R) \lra \C$ is \textbf{polyhomogeneous} if $f$ admits an asymptotic expansion
	\begin{equation} \label{polylambda}
		f(r, \theta) \sim \sum_{(s, p) \in \wt{\Lambda}} r^s \log^p r a_{s, p}(\theta),
	\end{equation}
where $\wt{\Lambda} \subset \R \times \N$ is a discrete set for which
each subset $\set{ s \leq c} \cap \wt{\Lambda}$ is finite.  In this
setup, $\wt{\Lambda}$ is called the `index set' of $f$, and the symbol
$\sim$ means that
	\begin{equation*}
		f(r, \theta) - \sum_{(s, p) \in \wt{\Lambda}, N > s}
                r^s \log^p r a_{s, p}(\theta) = o(r^{N})
	\end{equation*}
	\end{definition}
Let
\begin{align}\label{eq:phgmetrics}
  \Mm_{k, \g, \nu}^{phg}(\pp, \Aa) = \set{h \in \Mm_{k, \g, \nu}(\pp,
    \Aa)
    \left|
      \begin{array}{c}
      \mbox{ near } p \in \pp, h = e^{2\mu}\absv{z}^{2(\al -
        1)}\absv{dz}^{2} \\
      \mbox{ and } \mu \mbox{ is polyhomogeneous.}
      \end{array}
           \right.
      }
\end{align}
\nomenclature[M2]{$\Mm_{k, \g, \nu}^{phg}(\pp, \Aa)$}{polyhomogeneous
  conic metrics}

\subsection{Restatement of theorems}

Given two maps $u_i : \Sigma \lra \Sigma$, $i = 1, 2$, and a finite subset $\qq \subset \Sigma$, we define the equivalence relation
	\begin{align} \label{relhtopyclass}
		u_1 \sim_\qq u_2 &\iff \begin{array}{c} u_1 \mbox{ is homotopic to } u_2 \mbox{ via } \\
		 F:[0, 1] \times \Sigma \lra \Sigma \\
		 F_t(q) = u_1(q) \mbox{ for all } (t, q) \in [0, 1] \times \qq
		\end{array}.
	\end{align}
If $u_1 \sim_\qq u_2$, we say that the two maps are homotopic relative
to $\qq$ (rel$.$ $\qq$).  For fixed $u_1$, the set of all $u_2$
satisfying (\ref{relhtopyclass}) is referred to as the rel$.$ $\qq$
homotopy class of $u_1$, and is denoted by $[u_1; \qq]$.  By $u_1 \sim
u_2$ we mean that $u_1$ and $u_2$ are homotopic with no restrictions
on the homotopy.  

\nomenclature[p]{$\pl, \pg, \pp_{= \pi}$}{points with various angle specifications}

Define
	\begin{align} \label{inotepi}
		\begin{array}{cc}
		\pp_{< \pi} &= \set{ p_i \in \pp \vert 2 \pi \al_i < \pi} \\
		\pp_{> \pi} &= \set{ p_i \in \pp \vert 2 \pi \al_i > \pi}  \\
		\pp_{= \pi} &= \set{ p_i \in \pp \vert 2 \pi \al_i = \pi} 
		\end{array}
	\end{align}
and assume that
 \begin{equation*}
		\pe = \varnothing.
\end{equation*}
We discuss the case $\pe \neq \varnothing$ in section \ref{coneanglepisect}.

\begin{thm} \label{mainTheorem}
Assume that genus $\Sigma > 0$.  For $\Sigma$ and $\pp$ as above, let $G \in  \Mm_{k, \g,
  \nu}^{phg}(\pp, \Aa)$ have cone angles less than $2 \pi$, and
$\kappa_G \leq 0$.  Let $\cc$ be a conformal class on $\Sigma$.  Fix a
(possibly empty) subset $\qq \subset \pl$, let $w_0 : \Sigma \lra \Sigma$ be
a homeomorphism.

If $\qq \neq \varnothing$, set $\qq' := w_0^{-1}(\qq)$.  Then there exists a
unique energy minimizing map $u$ in $[w_0 ; \qq']$.  Furthermore,
	\begin{align} \label{minthm1}
		u^{-1}(p_i) \mbox{ is a single point for all } i = 1, \dots, k,
	\end{align}
and 
 \begin{equation} \label{minthm2}
		u: \Sigma - u^{-1}(\pp)  \lra \Sigma - \pp
 \end{equation}
is a diffeomorphism.

If $\qq = \varnothing$, there exists an energy minimizing map $u$ in
$[w_0]$, unique up to precomposition with a
conformal automorphism of $(\Sigma, \cc)$.  These also satisfy
(\ref{minthm1}) and (\ref{minthm2}).

If $\Sigma = S^{2}$ then, assumptions as above, there is a unique
energy minimizing map in the rel$.$ $\qq$ homotopy class of $w_{0}$
provided $\absv{\qq} \geq 3$.  Again the map satisfies \eqref{minthm1}---\eqref{minthm2}.

\end{thm}
\nomenclature[p]{$\qq$}{the cone points in whose relative homotopy
  class we minimize}






A simple argument using the isometry invariance of the energy
functional allows us to reduce Theorem \ref{mainTheorem} to the case
in which $\qq' = \qq$ and $w_0 \sim_{\qq} id$.  In fact, if we prove
Theorem \ref{mainTheorem} in this case, then given an arbitrary $w_0
\in \Diff(\Sigma)$ and a metric $G$ with the properties in Theorem
\ref{mainTheorem}, let $ \wt{u} : (\Si, (w_0^{-1})^*\cc) \lra (\Si, G
) $ be the unique minimizer in $[id; \qq]$.  Then since composition
with an isometry leaves the energy unchanged, $ u :=  \wt{u} \circ
w_0^{-1} :  (\Si, \cc) \lra (\Si, G ) $ is the unique minimizer in
$[w_0; \qq']$.  The same argument with $\qq = \varnothing$ shows that
we can reduce to the case $w_0 \sim_{\qq} id$.  Thus we will always assume that $\qq' = \qq$.

Our approach to proving the theorem is to find homeomorphisms whose
tension fields vanish away
from $u^{-1}(\pp)$.  
Among diffeomorphisms in Form \ref{uform} with
vanishing tension field, minimizing energy in the sense of Theorem \ref{mainTheorem} turns out
to be equivalent to a condition on the Hopf differential of $u$, which
we define now.  
Given $u : \Sigma_{\pp'} \lra \Sp$ with $\tau(u, g, G) = 0$ on $u^{-1}(\pp)$, let
${u^*G}^\circ$ denote the $g$-trace-free part of $u^*G$.  Among
$C^2$ maps, away from the cone points
 \begin{equation} \label{divfree}
		\tau(u, g, G) = 0 \implies \delta_g \lp {u^*G}^\circ \rp= 0,
 \end{equation}
where $\delta_g$ is the divergence operator for the metric $g$ acting
on symmetric $(0, 2)-$ tensors.  Trace-free, divergence-free tensors
are equal to the real parts of holomorphic quadratic differentials
w.r.t$.$ the conformal class $[g]$, so in conformal coordinates, we can write

\nomenclature[p]{$\Phi(u)$}{the Hopf differential}

 \begin{equation} \label{idecomp}
		u^*G = \lambda \absv{ dz }^2  + 2 \Re \lp \phi(z) dz^2 \rp
 \end{equation}
where $\phi$ \textbf{is holomorphic}.  Parting slightly with standard notation, e.g. from \cite{w}, we use
the symbol $\Phi$ to refer to the tensor which in conformal
coordinates is expressed $\phi(z) dz^2$; this is called the
\textbf{Hopf differential} of $u$.  In section \ref{polessect} we will
show that, if $\pl' = u^{-1}(\pl),$ then
	\begin{align*}
		\Phi(u) \mbox{ is holomorphic on $\Sigma - \pl'$ with at most simple poles on $\pl'$.}
	\end{align*}
Our proof of Theorem \ref{mainTheorem} then relies on the following lemma, proven in section \ref{uniquesect}.
	\begin{lemma} \label{minlemma}
		Say genus $\Sigma > 0$.  Given a harmonic diffeomorphism $u: (\Sigma_{\pp'}, g) \lra
                (\Sp, G)$, for any $\qq \subset \pl$, if $\qq' =
                u^{-1}(\qq)$, then $u$ is energy minimizing in its rel$.$ $\qq'$
                homotopy class if and only if the Hopf differential
                $\Phi(u)$ extends smoothly to all of $\Sigma_{\qq'}$.

                If $\Sigma = S^{2}$, the same is true so long as $\absv{\qq} \geq 3$.
	\end{lemma}
        \begin{remark}
          If $\qq = \varnothing$, this lemma means that $u$ is
          minimizing in its free homotopy class if and only if
          $\Phi(u)$ is smooth on all of $\Sigma$.  We say that such maps
          are \textbf{absolute} minimizers.
        \end{remark}
Thus a diffeomorphism with vanishing tension field on $\Sigma_{\pp'}$
is minimizing in the sense of Theorem \ref{mainTheorem} if and only if
the residues of its Hopf differential vanish on $\pl' - \qq$, i.e.\ when it solves
the augmented equation
	\begin{customequation}{HME$(\qq)$} \label{hme2}
		\begin{array}{cccc}
			u: \Si &\lra& \Si &\mbox{ a homeomorphism} \\
			 u &\sim_{\qq}& id\\
			\tau(u, g, G) &=& 0 & \mbox{ on } \Sigma_{\pp'} \\
			\Res \rvert_{p} \Phi(u) &=& 0 & \mbox{ for
                          each } p \in \pl' - \qq.
		\end{array}
	\end{customequation}

	\begin{remark} \label{notinv}
The residues of a holomorphic quadratic differential are only defined up to multiplication by an element of $\C^*$.  Whether they are zero or not \textit{is} well-defined, but since we will work with the residues directly, we will have to fix first conformal coordinates near each cone point.
	\end{remark}

\section{The harmonic map operator} \label{setup}

In this section we discuss the global analysis of the map $\tau$.  We begin by discussing the invariance under conformal change of the domain metric.

\subsection{Conformal invariance}   \label{confinvsect}

Let $g, G \in \Mm_{2, \g, \nu}(\pp, \Aa)$, and suppose that $u$ is a
$C^2$ map of $\Sp$.  Suppose we have conformal expressions $G = \rho
\absv{du}^2$ at some $x \in \Sp$ and $ g = \sigma \absv{dz}^2$ and near $u^{-1}(q)$. 
The energy density (the integrand in (\ref{energy})) in conformal coordinates is
 \begin{equation} \label{confenergy}
		e(u, g, G)(z) := \frac{1}{2} \norm[{g \otimes u^*G}]{du(z)}^2 = \frac{\rho(u(z))}{\sigma(z)} \lp \absv{ \p_z u}^2 + \absv{ \p_{\overline{z}} u}^2\rp
 \end{equation}
From this expression it is easy to verify that if
	\begin{equation} \label{uC}
		\xymatrix@!C=3pc@R=0pc{
		(\Sigma_1, g) \ar[r]^C & (\Sigma_1, g) \ar[r]^u & (\Sigma_2 , G)\\
		z \ar@{|->}[r] & w = C(z)  \ar@{|->}[r] & u(w)
		}
	\end{equation}
for arbitrary surfaces $\Sigma_i$, $i = 1, 2$ and $C$ is conformal ($C^*g = e^{2\mu}g$), then
 \begin{equation} \label{energyinv}
		e(u \circ C, g, G)(z) = \absv{\p_z C(z)}^2 e(u,  C^*g, G)(z),
 \end{equation}
and so
	\begin{align} \label{energyinv2}
		E(u \circ C , g, G) = E(u, C^* g, G) = E(u, g, G).
	\end{align}
A simple computation preformed e.g. in \cite{sy} shows that 
	\begin{align} \label{hmec}
		\fbox{$
		\displaystyle{
		\tau(u, g, G) = \frac{4}{ \sigma} \lp u_{z \overline{z}} + \frac{\p \log \rho}{\p u} u_z u_{\overline{z}}  \rp
		}
		$}
	\end{align}
The tension field enjoys a point-wise conformal invariance; in the
situation of (\ref{uC}), if one knows only at some particular
point $z_{0}$ that $C_{\ov{z}}(z_0) = 0 $, one can use this formula to chech that
 \begin{equation} \label{tauconfinv}
		\tau( u \circ C, g, G) = \tau(u, C^*g, G)\circ C
 \end{equation}

\subsection{Global analysis of $\tau$} \label{defHsect}

Given a diffeomorphism $\wt{u}_0: (\Sp, g_{0}) \lra (\Sigma_{\pp'},
\wt{\Gfixed})$ solving, pull back $\wt{\Gfixed}$ as in Remark
\ref{thm:idremark} to obtain that $id = u_0: (\Sp, g) \lra (\Sp,
\Gfixed)$ is a diffeomorphism solving
(\ref{hme2}) in Form \ref{uform}.
To avoid tedious repetition below, we give a name to our main
assumption on the metrics and maps.

\begin{assumption} \label{assump}

\begin{enumerate}
\item  $\Gfixed \in \mc{M}_{2, \g, \nu}(\pp, \Aa)$ with $\al_j < 1$ for
  each $ 2 \pi \al_j \in \Aa$, and $\nu_{j} > 2\al_{j} $ (see section \ref{scalarcurvature})  
\item $g$ is also in $\Mm_{2, \g,
    \nu}(\pp, \Aa)$.
\end{enumerate}

\end{assumption}

By (\ref{def2}), if we let $\tau$ act on a space of varying maps and
metrics, then the range of this operator depends on the domain, and
therefore $\tau$ is most naturally viewed as a section of the bundle.  We
will now define precisely the domain of $\tau$ and the vector bundle
$\EB$ in which it takes values.

Fix coordinates $z_j$ near each $p_j \in \pp$, conformal with respect
to $g$, and let $z_j = r_j e^{i \theta_j}$.  Continuing to abuse
notation, let $u_j$ be conformal coordinates for $G$ near $p_j$.  (We
will often omit $j$ from the notation when it is understood that we
work with a fixed cone point.)  First, given any $u$ in Form
\ref{uform} we define $r^c \Xx_b^{k,
  \g}(u)$ for any $c \in \R^k$ by
		\begin{align} \label{vflds}
		\psi \in r^{ c } \Xx_b^{k, \g}(u) &\iff \left\{ 
		\begin{array}{c}
\psi \in \Gamma(u^{*}T\Sp) \\
 \psi \in C^{k, \g}_{loc} \mbox{ away from } u^{-1}(\pp)  \\
\psi \in r^{c_j }C^{k, \g}_b(D(\sigma)) \mbox{ near } p_{j} \in \pp
	\end{array} \right. ,
	\end{align}
for some $\sigma > 0$.  

\begin{remark}
We will often write $r^{1 + \e} \Xx^{k, \g}_b(u)$ for  a positive
number $\e$, by which we mean $r^{c} \Xx^{k, \g}_b(u)$ where $c_j = 1
+ \e$ for all $j$.  Given $\delta \in \R$, by $c > \delta$ we mean
that $ c_j > \delta$ for all $j$.
\end{remark}
\nomenclature[X]{$r^{1 + \e} \Xx^{k, \g}_b(u)$}{$r^{1 + \epsilon}C^{2,
    \gamma}_{b}$ vector fields over $u$}

\begin{itemize}

	\item Let $ \Bb^{1 + \e}(u_0)$ be the space of perturbations of $u_0$ defined by
 \begin{equation} \label{maps}
		\Bb^{1 + \e}(u_0) = \{ \exp_{u_0} \psi \ | \ \psi \in r^{ 1 + \e } \Xx_b^{2, \g}(u_0), \| \psi \|_{r^{ 1 + \e } \Xx_b^{2, \g}} < \delta \}.
 \end{equation}
	The $\delta > 0$ is picked small enough so that all the maps in this space are diffeomorphisms of $\Sp$ and will be left out of the notation.
	Note that
 \begin{equation} \label{tangentspace}
			\mbox{T}_{u_0}\Bb^{1 + \e}(u_0) \simeq  r^{ {1 + \e} } \Xx_b^{2, \g}(u_0)
 \end{equation}
	
\nomenclature[B]{$\Bb^{1 + \e}(u_0)$}{$r^{1 + \epsilon}C^{2,
    \gamma}_{b}$ perturbations of $u_{0}$}

	\item We also need a space of automorphisms of $\Sigma$, analogous to
          those in \ref{ilocconf}, that are locally
          conformal near $\pp$ with respect to $g$.  Let $\Vv$ be a
          $2\absv{\pp}$-dimensional space of diffeomorphisms of
          $\Sigma$ parametrized by an open set $0 \in U \in
          \C^{\absv{\pp}}$, so that $\lambda \in U$ corresponds to a
          map $M_{\lambda}$ with
          \begin{equation}
            \label{eq:cplxmult}
            M_{\lambda}(z_{j}) = \lambda_{j}z_{j} \mbox{ near } p_{j}.
          \end{equation}
          Similarly, let $\mc{T}$ be a $2\absv{\pp'}$-dimensional
          space of diffeomorphisms of
          $\Sigma$ \\
parametrized by an open set $0 \in V \in
          \C^{\absv{\pp'}}$, so that $w \in V$ corresponds to a map
          $T_{w}$ with
          \begin{equation}
            \label{eq:cplxtrans}
            		T_w(z_j) = z_j - w_j \mbox{ near } p_{j}.
          \end{equation}
To be precise, we take these locally defined maps and extend them to
diffeomorphisms of $\Sigma$ in such a way that they are the identity
outside a compact set, and in particular such that $T_{0} = id$.  For any subset $\wt{\pp} \subset \pp$, define $\Tt_{\wt{\pp}} \subset$ by the
condition that
that $w_{i} = 0$ for $p_{i} \not \in \wt{\pp}$.  This can be done in
such a way that
\begin{equation}
\label{thets}
  \mc{T} = \mc{T}_{>\pi} \circ \Tt_{< \pi} \circ \Tt_{= \pi}.
\end{equation}
Finally, set
\begin{equation}
  \label{eq:totalconf}
  \mc{C} := \Vv \circ \mc{T}
\end{equation}

\nomenclature[T]{$\mc{T}$}{local conformal translations}
\nomenclature[T]{$\mc{T}_{\wt{\pp}}$}{local conformal translations
  near $\wt{\pp}$}

	\item  Finally, given $h_0 \in \Mm_{k, \g, \nu}(\pp, \Aa)$,
          define a subspace $$\Mm^*_{k, \g, \nu}(h_0, \pp, \Aa) \subset
          \Mm_{k, \g, \nu}(\pp, \Aa)$$ as follows.  Let $D_{j} =
          \set{z_{j} \leq 1}$ where again $z_{j}$ are conformal
          coordinates for $h_{0}$ near $p_{j}$.
	\begin{align} \label{metstar}
		\Mm^*_{k, \g, \nu}(h_0, \pp, \Aa) &= \set{
		h \in \Mm_{k, \g, \nu}(\pp, \Aa) \left|
                 \begin{array}{c}
                  id \rvert_{D_{j}} : (D_{j},
                h_{0}) \lra (\Sigma, h)  \\
		\mbox{is conformal for all $j$ }
                \end{array}
		\right.
              }
	\end{align}
In words, this means the conformal coordinates for $h_{0}$ near $\pp$ \textbf{are}
conformal coordinates for $h_{0}$. This definition may seem arbitrary, but as we will see in section
\ref{seccont}, it is motivated by the requirement that $\tau$ be a
continuous map.  

\nomenclature[M2]{$\Mm^*_{k, \g, \nu}(h_0, \pp, \Aa) $}{metrics
  locally uniform to $h_{0}$}

\end{itemize}

To clarify the relationship between $\Mm^*_{k,
  \g, \nu}(h_0, \pp, \Aa)$ and $\Mm_{k, \g, \nu}(\pp, \Aa)$, we can
construct, locally near $h_{0}$, a smooth injection from an open ball
$\mc{U}$ in the latter space into $\Diff_{0}(\Sigma;
\pp)$ times the former by uniformizing locally around $\pp$.  To be
precise, given any $h \in \Mm_{k,
    \g, \nu}(\pp, \Aa)$, let $v_{h}: D_{j} \lra (D_{j}, h)$ be the
  solution to the Riemann mapping problem normalized by the condition
  that $v_{h}(0) = 0$ and $v_{h}(1) = 1$.  Let $\chi : D_{j} \lra \R$ be a smooth cutoff function with $\chi
  \equiv 1$ near $0$ and $\chi \equiv 0$ near $\p D_{j}$.  Consider the map
  \begin{align*}
    \wt{v}_{h}(z) &= \left\{
      \begin{array}{ccc}
        \chi(z_{j}) v_h(z_{j}) + \lp 1 - \chi(z_{j}) \rp z_{j} &\mbox{
          on } D_{j} \\
        id & \mbox{elsewhere}
      \end{array}
      \right.
  \end{align*}
Then $\wt{v_h}$ is well-defined, and is a diffeomorphism if, for
$\norm[C^{\infty}(D)]{v_h - id} < \epsilon.$ We have
\begin{lemma}\label{thm:localuni}
  The map --- defined locally near $h_{0}$ --- given by
  \begin{align*}
    \Mm_{k, \g, \nu}(\pp, \Aa) &\lra \Diff_{0}(\Sigma; \pp) \times \Mm^*_{k, \g, \nu}(h_0, \pp, \Aa)\\
    h &\longmapsto (\wt{v}_{h}, \wt{v}_{h}^{*}h)
  \end{align*}
is an isomorphism onto its image on a ball near $h_{0}$.
\end{lemma}
\begin{proof}
  This follows immediately from the fact that the solution to the
  Riemann mapping problem has $C^{\infty}$ norm controlled by the
  distance from $h$ to $h_{0}$
\end{proof}
Note that, by construction $id: (\Sigma, h_{0}) \lra (\Sigma,
\wt{v}_{h}^{*}h)$ is conformal near $\pp$.

With these definitions, we consider
 \begin{equation} \label{mapquest}
   \begin{split}
		\tau : \lp \Bb_{2, \g}^{1 + \e}(u_0) \circ \mc{C} \rp
                \times \Mm^*_{2, \g, \nu}(g_{0}, u^{-1}(\pp), \Aa)
                \times \Mm^*_{2, \g, \nu}(\Gfixed, \pp, \Aa)
                & \lra \EB \\
                (u, g, G) &\longmapsto \tau(u,g,G),
              \end{split}
 \end{equation} 
where $\pi: \EB \lra \lp \Bb_{2, \g}^{1 + \e}(u_0) \circ \mc{C} \rp
                \times \Mm^*_{2, \g, \nu}(g_{0}, u^{-1}(\pp), \Aa)
                \times \Mm^*_{2, \g, \nu}(\Gfixed, \pp, \Aa)$ is the bundle
whose fibers satisfy
\begin{equation*}
  \pi^{-1}(u \circ C, g, G) = r^{1 + \e - 2\Aa} \Xx^{0, \g}_b(u)
\end{equation*}
In section \ref{seccont}, we will prove
	\begin{proposition}	\label{Hcontinuity}
		Let $(u_0, g_{0}, \Gfixed)$ solve (\ref{hme2}) and satisfy  Assumption \ref{assump}.  Then the map
                (\ref{mapquest}) is $C^1$.
	\end{proposition}

\nomenclature[EB]{$\EB$}{bundle in which $\tau$ takes values.}

\subsection{Proof of Theorem \ref{mainTheorem}} \label{sec:actualproof}

Let $\cc$ be any conformal structure on $\Sigma$ and $G \in
\Mm_{2, \g, \nu}(\pp, \Aa)$ any metric satisfying the hypotheses of
the theorems.  Given $\qq \subset \pp$ with $\qq \neq \varnothing$
(resp$.$ $\qq= \varnothing$), we would like to find a map $u : (\Sigma,
\cc) \lra (\Sigma, G)$ that minimizes energy in the rel$.$ $\qq$ homotopy class of
the identity (resp$.$ the free homotopy class of the identity.)  Let
$\cc_{0} := [G]$ be the conformal class of $G$, and
let $\cc_{t}$, $t \in [0, 1]$ be a smooth path of conformal structures from
$\cc_{0}$ to $\cc_{1} = \cc$.  (That the space of conformal structures is
connected follows immediately from the convexity of the space of
metrics.)  Finally, define
	\begin{align*}
		\mc{H}(\qq) = \set{ t \in [0, 1] \left|
                   \begin{array}{c}
                     \mbox{there is a map } u_t : (\Sigma,
                  \cc_{t}) \lra
                   (\Sigma, G) \\
                   \mbox{ so that } (u_{t},
                     \cc_{t}, G) \mbox{ satisfies (\ref{hme2}).}
                 \end{array}
                \right. }
	\end{align*}
In the remainder of this paper, we will prove
\begin{proposition}
  \label{thm:clopenprop}
If genus $\Sigma > 0$, $\mc{H}(\qq)$ is closed, open, and non-empty.
If $\Sigma = S^{2}$, then the same is true provided $\absv{\qq} \geq 3$.
\end{proposition}
\begin{remark}
  The set is non-empty; it contains $0$, since $\cc_{0} =
  [G]$ and thus the identity is conformal.  Degree one conformal maps are global energy minimizers in
  their homotopy classes, \cite{el}.
\end{remark}

We can now prove that Theorem \ref{mainTheorem} is true, at least up
to the proofs of the preceding proposition and Proposition
\ref{umin}, where uniqueness is shown.
\begin{proof}[Proof of Theorem \ref{mainTheorem}]
The content of Proposition \ref{umin} is that solutions to
\eqref{hme2} are minimizing in their rel$.$ $\qq$
(or free if $\qq = \varnothing$) homotopy classes, and that such minimizers are unique in
the appropriate sense.  Thus it suffices to solve the
equation, but Proposition \ref{thm:clopenprop} implies that a solution
always exists.
\end{proof}


\section{Openness via non-degeneracy}

Our proof that $\mc{H}(\qq)$ is open relies on two non-degeneracy
results.  We describe these now, and then use them to prove openness.

\subsection{Non-degeneracy of $\tau$}

Let $(u_0,
g, \Gfixed)$ solve (\ref{hme2}) with $u_0$ in Form \ref{uform}.  Let $u_t$
be a $C^{1}$ path in $\Bb_{2, \g}^{1 + \e}(u_0) \circ  \Vv \circ \Tt_{> \pi}$ and
write $\psi := \dot{u}_0$. Define
 \begin{equation} \label{L}
		\Ll \psi := \ddt \tau(u_t, g, \Gfixed).
 \end{equation}
By (\ref{tangentspace}) and the fact that $u_0 = id$, the domain of $L$ is
 \begin{equation*}
		 \mbox{T} \lp  \Bb^{1 + \e} \circ \Vv \circ \Tt_{>
                   \pi}  \rp = r^{1 + \e} \Xx^{2, \g}_b +
                 \mbox{T}_{id} \Vv + \mbox{T}_{id}\Tt_{>\pi}.
\end{equation*}
We will show that the linearization  of $\tau$ is transverse to those
conformal Killing fields of $g$ which lie in its natural range.
A conformal Killing field for $g$ is
a vector field $C$ satisfying $\mc{L}_C g = \mu g$ for some function
$\mu$, where $\mc{L}$ denotes the Lie derivative.  This is the
derivative of the conformal map equation $F_t^* g = e^{\mu_t} g$ for
some family $F_t$ with $F_0 = id$.  It is well known that for surfaces
the conformal killing fields are exactly the tangent space to the
identity component of the conformal group,
 \begin{equation} \label{confgrp}
		\mc{C}onf_0 = \set{ C : (\Sigma, g) \lra (\Sigma, g) : C^*g = e^{2\mu} g }
 \end{equation}
This space contains only the identity map if genus $\Sigma > 1$ and is
two or three dimensional if the genus is $1$ or $0$,
respectively.  We have
	\begin{proposition} \label{mainlemma}
	Notation as above,
	\begin{equation} \label{icok}
		\Ll \lp\mbox{T} \lp r^{1 + \e} \Xx^{2, \g}_b
                +\mbox{T}_{id} \Vv +\mbox{T}_{id}\Tt_{>\pi} \rp \rp
                \oplus \lp\CKK\rp= r^{1 + \e - 2\Aa}\Xx^{0, \g}_b.
	\end{equation}
	\end{proposition}
In words, $\tau$ in \eqref{mapquest} is transver to whatever conformal
Killing fields lie in the range.
\ref{pertsection1}.	


\subsection{Non-degeneracy of the residue map}

To describe the second non-degeneracy result, we begin by describing the space of harmonic maps near a given
solution to (\ref{hme2}) with fixed geometric data.  Let $(u_{0}, g,
\Gfixed)$ solve \eqref{hme2} and satisfy Assumption \ref{assump}.  By the previous
section, there is an open set $\mc{U} \subset \Tt_{< \pi} \times
\Mm^*_{2, \g}(\Gfixed, \pp, \Aa)$ and a map graphing zeros of the tension
field operator,
 \begin{equation} \label{harmfmap}\begin{split}
	\mc{S}: \mc{U}&\lra \Bb^{1 + \e}_{2, \g}(u_0) \circ \Vv \circ \Tt_{>\pi}  \\
	(T, G) &\longmapsto u \mbox{ with } \tau(u, g, G) = 0,\\
\end{split}\end{equation}
so that $u = \wt{u} \circ D \circ T' \circ T$, where $T' \in \Tt_{>
  \pi}, D \in \Vv, \wt{u} \in \Bb^{1 + \e}_{2, \g}(u_0) $.

Given $\wt{\qq} \subset \pl$, set
 \begin{equation*}
			\Tt_{\wt{\qq}} := \set{ T_w : w_i = 0 \mbox{ for } p_i \not \in \wt{\qq}},	
\end{equation*}
c.f. (\ref{thets}).  We take $\wt{\qq} = \pl - \qq$.  Given the
identification of $\Tt_{ \pl - \qq}$ with a ball $U \subset \C^{\absv{
    \pl - \qq}}$ around the origin, if $U$ is a sufficiently small we
can define (locally near $u_{0}$ the $2\absv{\pl - \qq})$-dimensional
manifold of harmonic maps fixing $\qq$,
 \begin{equation} \
		\mc{H}arm_{\qq} = \set{\mc{S}(T_w, \Gfixed) : w \in U},
 \end{equation}
Also, let 
	\begin{align*}
		u_w := \mc{S}(T_w, \Gfixed).
	\end{align*}
We denote the tangent space to $\mc{H}arm_{\qq} $ by
	\begin{align} \label{jacspace}
		\fbox{$
		\mc{J}_{\qq} := T_{u_0} \mc{H}arm_{\qq}.
		$}
	\end{align}
We identify this space with $\C^{\absv{ \pl - \qq}}$ by setting
 \begin{equation}\label{eq:jw}
		J_w := \ddt u_{tw}.
\end{equation}
Clearly $\Ll J_w = 0$ on $\Sp$.
Consider the residue map,
	\begin{align*}
		\Res_{u_w^{-1}(\pl - \qq)} : \mc{H}arm_{\qq} &\lra \C^{\absv{ \pl - \qq}} \\
		u_w &\longmapsto \eval{\Res}_{u_w^{-1}(\pl - \qq)} \Phi(u_w),
	\end{align*}
where $\Phi(u_w)$ is the Hopf differential defined in (\ref{divfree})-(\ref{idecomp}).  (The $J_w$ also have Hopf differentials, defined by $\Phi(J_w) = \ddt \Phi(u_{tw})$.)  Differentiating $\Res$ at $u_0$ gives
 \begin{equation} \label{h2map}\begin{split}
		D\Res_{\pl -\qq} :  \mc{J}_{\qq} &\lra \C^{\absv{ \pl - \qq}} \\
		J_w &\longmapsto \Res(\Phi(J_w)).
\end{split}\end{equation}
We will prove the following 
\begin{proposition} \label{h2prop}
	If genus $\Sigma > 1$, the map (\ref{h2map}) is an isomorphism.  If genus $\Sigma = 1$, then the space $\mc{H}arm_{\qq}$ decomposes near $id$ as
 \begin{equation} \label{confdecomp}
			\mc{H}arm_{\qq} = \mc{H}arm_{\qq}' \circ \mc{C}onf_0,
 \end{equation}
and 
	\begin{align} \label{smallinj}
		D \Res_{\pl - \qq} : T\mc{H}arm_{\qq}' \lra \C^{\absv{ \pl - \qq}} \mbox{ is injective.}
	\end{align}
\end{proposition}

\nomenclature[H]{$\mc{H}arm_{\qq}$}{harmonic maps with fixed geometric
  data fixing $\qq$}


\subsection{$\mc{H}(\qq)$ is open}

We will now use the two non-degeneracy results just discussed to prove openness.

\begin{proposition}
Propositions \ref{mainlemma} and \ref{h2prop} imply that $\mc{H}(\qq)$ is open.
\end{proposition}
\begin{proof}

Given $t_{0} \in \mc{H}(\qq)$, let $g_{0} \in \cc_{t_{0}}$ be a
metric satisfying Assumption \ref{assump}.  (See section
\ref{sec:actualproof} for definitions.)  We now use Lemma
\ref{thm:localuni}; there is a small $\delta > 0$
such that for $t \in (t_{0} - \delta, t_{0} + \delta)$ there is a path
of diffeomorphisms $\wt{v}_{t}$, all isotopic to the identity, such that the pullback conformal
structures $\wt{v}_{t}^{*}\cc_{t}$ have the property that
$id: (\Sigma, \cc_{t_{0}}) \lra (\Sigma, \wt{v}_{t}^{*}\cc_{t})$
is conformal near $\pp$.  Let $\wt{g}_{t} \in \wt{v}_{t}^{*}\cc_{t}$ be any family of metrics
that equal $g_{\alpha_{j}}$ on the conformal ball $D_{j}$ near
$p_{j}$.  The point is that this can be done uniformly for $t$ near
$t_{0}$ since all the conformal structures $\wt{v}_{t}^{*}\cc_{t}$ are
equal there.  Thus $\wt{g}_{t} \in \Mm^*_{2, \g, \nu}(g_0, \pp,
\Aa)$.  We claim that there is a unique solution $\wt{u}_{t} :
(\Sigma, \wt{g}_{t}) \lra (\Sigma, G)$ to \eqref{hme2}.  Assuming this
for the moment, the proof is finished, since $u_{t}:= \wt{u}_{t} \circ \wt{v}_{t}^{-1} : (\Sigma, \lp
\wt{v}_{t}^{-1} \rp^{*}\wt{g}_{t}) \lra (\Sigma, G)$
is a relative minimizer and $\lp\wt{v}_{t}^{-1} \rp^{*}\wt{g}_{t}
\in \cc_{t}$. 

The proposition is proven modulo the existence of $\wt{u}_{t}$, which
we now prove using Propositions \ref{mainlemma} and \ref{h2prop}.

If genus $\Sigma > 1$, then Proposition \ref{mainlemma} states that
the differential of the map \eqref{mapquest}
at a solution $(u_0, g_{0}, \Gfixed)$ to (\ref{hme2}) in the direction of $
\Bb^{1 + \e} \circ \Vv \circ \Tt_{> \pi} $ is an isomorphism.  This
together with the fact that $\tau$ is $C^1$ (Proposition
\ref{Hcontinuity}) and the Implicit Function Theorem shows that the
zero set of $\tau$ near $(u_0, g_{0}, \Gfixed)$ is a smooth graph over $\mc{T}_{<
  \pi} \times \Mm^*_{2, \g, \nu}(g_0, \pp, \Aa) $.  The map $\mc{S}$
in \eqref{harmfmap} realizes the
manifold of relative minimizers as a graph over an open set $\mc{U}
\subset  \Tt_{< \pi} \times
\Mm^*_{2, \g, \nu}(\Gfixed, \pp, \Aa)$.  We now use Proposition
\ref{h2prop}. The map $D (\Res_{\pl - \qq}
\circ \mc{S})$ is an isomorphism from $\mbox{T}_{id}
\mc{T}_{\pl - \qq}$ to $\C^{\absv{\pl - \qq}}$.  Thus the set $\mc{U} \cap
\set{\Res_{\pl - \qq}^{-1}(0)}$, which by Lemma \ref{minlemma}
consists of rel$.$ $\qq$ minimizers, is a graph over an open set $\mc{V} \subset
\mc{T}_{\pl - \qq} \times
  \Mm^*_{2, \g, \nu}(\Gfixed, \pp, \Aa)$ sufficiently close to
  $\set{id} \times \Gfixed$.  This proves the existence of the $\wt{u}_{t}$.

Now suppose genus $\Sigma = 1$.  Let $\set{C_i}$ be a basis for
$\CKK$.  Note that 
\begin{equation}
  \label{eq:ckk}
  \CKK = \left\{ \begin{array}{ccc}
      \CK &\mbox{ if }& \pl = \varnothing \\
      \set{0} &\mbox{ if }& \pl \neq \varnothing
      \end{array} \right.
\end{equation}
This follows immediately from the fact that conformal Killing fields
are nowhere vanishing.  

Thus if $\pl \neq \varnothing$ we again have
surjective differential, and as in the genus
$> 1$ case, rel$.$ $\qq$ minimizers forms a graph over $\mc{T}_{\pl -
  \qq} \times \Mm^*_{2, \g, \nu}(g_{0}, \pp, \Aa)$.  We lift to
the universal cover of $\Sigma$, which we can take to be $\C$, and let
$z$ denote the coordinate there.  By integrating
around a fundamental domain, we claim that if $\phi_w dz^2 = \Phi(J_w)$ then
 \begin{equation} \label{sumzero}
		\sum_{p_i \in \pl } \Res \rvert_{p_i} \phi_w= 0.
 \end{equation}
This follows immediately from the fact that the deck transformations
are $z \mapsto z + z_{0}$ for some $z_{0}$, so $\phi_{w}$ is actually
a periodic function with respect to the deck group.  
Define a subset $V \subset \C^{\absv{\pl - \qq}}$ by $V =
\mbox{span}\la (1, \dots, 1), (i, \dots, i) \ra$.
Thus by
\eqref{sumzero}, $V^\perp =  D\Res(T \mc{H}arm_{\qq})$,
where the orthocomplement is taken with respect to the standard
hermitian inner product on $\C^{\absv{\pl - \qq}}$.  Now by Proposition
\ref{h2prop} and the fact that $\CK$ has one complex dimension, $$ D
\Res: T \mc{H}arm_{\qq}' \lra V^\perp$$ is an isomorphism, and again
the existence of the $\wt{u}_{t}$ follows from the Implicit Function Theorem.

We now tackle the case: $\pl = \varnothing$ and genus $= 1$.
In this case $\Tt_{> \pi} = \Tt$.  Let $\CKg$ denote the identity
component of the conformal group of $g$.  Consider the quotient bundle $\wt{\EB}$ whose fiber over $(u, g, G)$ is given by $E_u / V_u$, where
	\begin{align*}
		V_u := u_* \TCKg \subset r^{1 - \e - 2\Aa}\Xx^{2, \g}_b(u)
	\end{align*}
Let $\pi: \EB \lra \wt{\EB}$ be the projection.  Proposition \ref{mainlemma}
immediately implies that the differential of the composition $\pi
\circ \tau$ in the direction of $ \Bb^{1 + \e} \circ \Vv \circ \Tt$ is an isomorphism.  By the Implicit Function Theorem, the
zero set of $\pi \circ \tau$ near $(u_0, g, \Gfixed)$ is a smooth graph
over $\Mm^*_{2, \g, \nu}(g_0, \pp, \Aa) $, so for each $g \in \Mm^*_{2, \g,
  \nu}(g_0, \pp, \Aa)$ there is a map $u = \wt{u} \circ D \circ T$
with $\Tt $ such that
 \begin{equation} \label{Vu}
	\tau(u, g, \Gfixed) \in V_u.
 \end{equation}
We will show that (\ref{Vu}) implies that $\tau(u, g, \Gfixed) = 0$.
Suppose that $\tau(u, g, \Gfixed) = u_*C$ for some $C \in \TCKg$.  Let
$f_t \subset \mc{C}onf_0$ be a family with $\eval{ \frac{d}{dt}}_{t =
  0} f_t = C$.  By the conformal invariance of energy,
(\ref{energyinv2}), we have.
 \begin{equation} \label{intoetilde0}
		\eval{ \frac{d}{dt}}_{t = 0} E(u \circ f_t, g, \Gfixed) = 0 
 \end{equation}
On the other hand we will show using (\ref{energyinv2}) that 
 \begin{equation} \label{intoetilde}
   \begin{split}
		\eval{ \frac{d}{dt}}_{t = 0} E(u \circ f_t, g, \Gfixed) &=
                \int_{\Sigma} \la \tau(u, g, \Gfixed) , u_* C\ra \sqrt{g} dx \\
                &= \norm[L^{2}]{u_{*}C}^{2},
              \end{split}
 \end{equation}
with $L^{2}$ norm as in \eqref{pairingdef}.  Some care is needed in
the proof since in general the boundary term in (\ref{frstvarbdry}) can
be singular.  We postpone the rigorous computation to section
\ref{confkillsect}, where several similar computations are done at
once.  Using the fact that the solutions are a graph over $\Mm^*_{2, \g,
  \nu}(g_0, \pp, \Aa)$ to conclude that the $\wt{u}_{t}$ exist as in
the previous cases.

Finally, suppose $\Sigma = S^{2}$.  Then again $\CKK = \varnothing$
since the only conformal Killing field vanishing at three points is
identically zero.  The proof then proceeds as in the previous cases.

\end{proof}

\section{Uniqueness and convexity} \label{uniquesect}

%

%


%

%

The main result of this section is Proposition \ref{umin}, which states that if $(u, g, G)$ is a solution to
(\ref{hme2}) in Form \ref{uform} satisfying Assumption \ref{assump}, then, up to conformal
automorphism, u is uniquely energy minimizing in its rel$.$ $\qq$
homotopy class.

\subsection{The Hopf differential} \label{polessect} Let
$(u, g, G)$ solve (\ref{hme2}) and satisfy Assumption \ref{assump}.  In conformal coordinates, a trivial computation using (\ref{confenergy}) yields
 \begin{equation} \label{pbm}\begin{split}
		u^* G 
		&= e(u) \sigma \absv{dz}^2 + 2 \Re \rho(u) u_z \ov{u}_z dz^2.
\end{split}\end{equation}
where $g = \sigma  \absv{dz}^2$. If we let $u^*G^\circ$ denote the
trace free part of $u^*G$ w.r.t$.$ $g$, i.e.\ $u^*G^\circ = u^*G -
\frac{1}{2} \lp \tr_g u^*G \rp g$, then $u^*G^\circ  = 2 \Re \Phi(u)$
where $\Phi(u)$ is the Hopf differential
	\begin{align} \label{hopfcoeff}
		\Phi(u) := \phi(z) dz^2 = \rho(u) u_z \ov{u}_z dz^2.
	\end{align}
It follows directly (see section 9 of \cite{s}) that for $z_0 \in \Sp$
	\begin{equation} \label{iff}
          \begin{split}
            \tau(u, g, G)(z_0) = 0 &\implies \p_{\ov{z}}\phi(z_0) = 0  \\
\p_{\ov{z}}\phi(z_0) = 0 \quad \& \quad J(u)(z_{0}) \neq 0 &\implies \tau(u, g, G)(z_0) = 0,
          \end{split}
	\end{equation}
where $J(u)(z_{0})$ is the Jacobian determinant of $u$.  This means that, among local diffeomorphisms, the vanishing of the
tension field is equivalent to the holomorphicity of the (locally
defined) function $\phi$.
By Form \ref{uform}, near $p \in \pp$ we have $u(z) = \lambda z + v(z)$
with $\lambda \in \C^*$ and $v(z) \in r^{1 + \e}C^{2, \g}_b$ for some $\e > 0$.  
By the definition of $C^{2, \g}_b$ from the previous section, and the fact that $\p_z = \frac{1}{2 z} \lp r \p_{r} - i \p_{\theta} \rp$, we see that
 \begin{equation} \label{hopfbound}
		\phi(z)= \lp \absv{\lambda z}^{2 (\al - 1)} + o( \absv{ z}^{2 (\al - 1)})\rp \lp \lambda + o(1) \rp \mc{O}(\absv{z}^{\e}) = \mc{O}(\absv{z}^{-2 + 2\al + \e}).
\end{equation}	
Since $ - 2 + 2 \al + \e > -2$, the function $\phi$ has at worst a
simple pole at $z = 0$.  If $\al \geq 1/2$ then $- 2 + 2 \al + \e >
-1$, so $\phi$ extends to a holomorphic function over $p$.  Thus we have proven
	\begin{lemma} \label{hopfpoles}
		Let $\Phi(u)$ be the Hopf differential of a solution $(u, g, G)$ to (\ref{hme2}) in Form \ref{uform}, with $G \in \Mm_{2, \g, \nu}(\pp, \Aa)$.
		\begin{align*}
			\Phi(u) \mbox{ is holomorphic on $\Sigma - \pp_{< \pi}$ with at most simple poles on $\pp_{< \pi}$.}
		\end{align*}
	\end{lemma}

\subsection{Uniqueness} \label{realuniquesect}


The main result of this section is the following.  
\begin{proposition}\label{umin}
Let $(u, g, G)$ solve (\ref{hme2}) with $u$ in Form \ref{uform}, and
assume that $\Phi(u)$ has non-trivial poles exactly at $\qq \subset \pl.$ 

Assume $\qq \neq \varnothing$.  Then for any $w : \Sigma \to \Sigma$ with $w \sim_{\qq} u$ (see (\ref{relhtopyclass}); in particular $w \rvert_\qq = u  \rvert_\qq $) we have
 \begin{equation*}
		E(u, g, G) \leq  E(w, g, G)
\end{equation*}
with equality if and only if $u = w$.

If $\qq = \varnothing$ (i.e.\ $(u, g, G)$ solves (\ref{hme2})) and $w : \Sigma \lra \Sigma$ satisfies the weaker condition $w \sim u$, then 
 \begin{equation*}
		E(u, g, G) \leq  E(w, g, G)
\end{equation*}
with equality if and only if $u = w \circ C$
for $C \in \mc{C}onf_0$ (see (\ref{confgrp})).  In particular, if genus $\Sigma > 1$, equality holds if and only if $u = w$.
\end{proposition}
\begin{proof}

Assume $\qq = \varnothing$.  We use a trick from \cite{ch} to reduce to the smooth case.

First assume that the genus of $\Sigma > 1$.  Let $ \omega(z) \absv{
  dz }^2$ be the unique constant curvature $-1$ metric in $[g]$.  By
section \ref{polessect}, for any $\e > 0$ in local coordinates we can write
 \begin{equation} \label{chdecomp}\begin{split}
		u^*G &=  e(u) \sigma dzd\ov{z} + 2 \Re \phi(z) dz^2 \\
		&= \underbrace{ \lp  e(u) \sigma  - \lp \e \omega^2 + \absv{\phi}^2 \rp^{1/2} \rp \absv{ dz }^2}_{ :=  H_1} + \underbrace{\lp \e \omega^2 + \absv{\phi}^2 \rp^{1/2}  \absv{ dz }^2  +  2  \Re \phi(z) dz^2 }_{ :=  H_2}
\end{split}\end{equation}
For $\e$ sufficiently small $H_1$ is a metric on $\Sp$; this follows
from the fact that for $u^*G$ to be positive definite we must have
that $ e(u) \sigma(z) > \absv{\phi(z)}$.  As for $H_2$, if $\qq =
\varnothing$, $\phi$ extends smoothly $p \in \pp$, so
$H_{2}$ is a smooth metric.  It is slightly more involved but also trivial to verify that the Gauss curvature of $H_2$ satisfies
 \begin{equation}
		\kappa_{H_2} < 0.
 \end{equation}
See Appendix B of \cite{ch} for the computation.  Since the above
characterization of $\tau(\wt{u}, g, \wt{G}) = 0$ in (\ref{iff}) is
necessary and sufficient, we see that 
	\begin{align*}
		id : (\Si, g) \lra (\Si, H_1) & \mbox{ is conformal} \\
		id : (\Si, g) \lra (\Si, H_2) & \mbox{ is harmonic}
	\end{align*}
From equation (\ref{energy}), for any $w: \Sigma \to \Sigma$ we have
 \begin{equation} \label{split}
		E(w, g, G) = E(w, g, H_1)  + E(w, g, H_2),
 \end{equation}
Unique minimization now follows from (\ref{split}), the fact that
degree one conformal maps are energy minimizing, and fact that a harmonic diffeomorphism into a negatively curved surface is uniquely energy minimizing in its homotopy class (see e.g. \cite{t}, \cite{h}).

If the genus of $\Sigma$ is $1$, then lifting to the universal cover
$\C$, we obtain a harmonic map $\wt{u} : (\C, \pi^* g) \lra (\C, \wt{\pi}^*
G)$,
where $\pi$ and $\wt{\pi}$ are conformal covering maps with respect to the standard
conformal structure on $\C$.  The metric $\absv{dz}^2$ descends to
$\Sp$ and is in the unique ray of flat metrics in the conformal class
of $g$.  Here $\Phi(\wt{u}) = \wt{\phi}(\wt{z}) d\wt{z}^2$ is defined
globally on $\C$ and $\wt{\phi}$ is entire and periodic with respect
to the deck transformations, hence bounded, hence constant.  Write
$\Phi(\wt{u}) = a d\wt{\wt{z}}^2$. For sufficiently small $\e > 0$, we
decompose
	\begin{align} \label{flatdecomp}
		\wt{u}^*(\pi^*G) = \underbrace{ \lp  e(\wt{u}) \sigma  - \e \rp \absv{ d\wt{z} }^2}_{ :=  \wt{K}_1} + \underbrace{ \e \absv{ d\wt{z} }^2  +  2  \Re \lp a d\wt{z}^2 \rp }_{ :=  \wt{K}_2}
	\end{align}
Since $e(\wt{u}) \sigma  > 2 \absv{a}$ and $\lp e(\wt{u}) \sigma
\rp(\wt{z})$ is periodic, there is an $\e$ so that the $\wt{K}_i$ are
positive definite.  If $K_i = \pi_* \wt{K}_i$, then $id: (\Sigma, g)
\lra (\Sigma, K_{i})$ is harmonic for both $i = 1, 2$.  We now argue as above, invoking both the minimality of conformal maps, and the minimality up to conformal automorphisms for maps of flat surfaces.  This completes the $\qq = \varnothing$ case.

If genus $= 0$, it is standard that $\Phi(u) \equiv 0$, and thus $u$
is conformal away from $\qq$, hence globally conformal.  Since
conformal maps are energy minimizing this case is complete.

To relate the $\qq = \varnothing$ case to the $\qq \neq \varnothing$ case, we will use the following
	\begin{lemma} \label{covering}
		Given a closed Riemann surface $R = (\Sigma, \cc)$
                (here $\cc$ is the conformal structure) with genus $>
                0$ and any finite subset $\mathfrak{q} \subset
                \Sigma$, there is a finite sheeted conformal covering
                space $ \pi: S \lra R $ which has non-trivial branch
                points exactly on $f^{-1}(\mathfrak{q})$.
                Furthermore, genus $S > $ genus $\Sigma$. 

                If $R = S^{2}$, the above is true provided $\absv{\qq}
                \geq 3$.
	\end{lemma}
This is a well-known fact from the theory of Riemann surfaces.  We
include a sketch of a proof here for the convenience of the reader.
First assume genus $\Sigma > 0$.
For each $q \in \qq$, there is a branched cover $Y_q \lra R$ ramified
above $q$, constructed as follows.  The fundamental group of $R -
\set{q}$ is a free group with $2g$ generators, where $g$ is the genus.
Let $Z_q$ be any normal covering space of $R - q$ that is not a
covering of $R$, i.e.\ let $Z_q$ correspond to a normal subgroup of
$\pi_1(R - \set{q})$ that is not the pullback of a normal subgroup in
$\pi_1(R)$ via the induced map.  Let $Y_q$ be the unique closure of
$Z_q$.  Finally, let $\wt{S}$ be the composite of the $Z_q$, i.e.\ the
covering of $R - \qq$ corresponding to the intersection of the groups
corresponding to the $Z_q$.  Then $\wt{S}$ is normal and has a unique
closure $S$ with a branched covering of $R$ that factors through each
of the $Y_q \lra R$.  The deck transformations of $\wt{S}$ extend to
$S$ and by normality act transitively on the fibers, and therefore $S$
is ramified exactly above $\qq$. \cite{c}.  If $\Sigma = S^{2}$ and
$\absv{\qq} \geq 3$, let $\qq_{0} \subset \qq$ have exactly three
points. There is a branched cover of $\Sigma$ by a torus branched over
$\qq_{0}$.  Applying the previous argument to the torus and branching
over all the lifts of cone points to the torus gives the result. \cite{ya}

Assuming $\qq \neq \varnothing$, let $S$ a covering map branched
exactly over $\qq$. (By our assumption that $\absv{\qq} \geq 3$ in
case $\Sigma = S^{2}$, such a cover exists.) and lift $u$ to a map $\wt{u}$ so that
	\begin{align}\label{eq:lift}
		\xymatrix{
			(S, \pi^*g) \ar[r]^{\wt{u}} \ar[d]^\pi & (S, \pi^* G) \ar[d]^\pi \\
			(\Sp, g) \ar[r]^u & (\Sp, G)
		}.
	\end{align}
In particular, $\wt{u}^* \pi^* G  = \pi^* u^* G$,
so the Hopf differential of $\wt{u}$ is the pullback via the conformal map $\pi$ of the Hopf differential of $u$.  Pick any $p \in \qq$, and let $ \phi(z) dz^2 $ be a local expression of the Hopf differential of $u$ in a conformal neighborhood centered at $p$.  By assumption, $\phi$ has at most a simple pole at $0$.  Given $ q \in \pi^{-1} (p) $, since $q$ is a non-trivial branch point we can choose conformal coordinates $\wt{z}$ near $q$ so that the map $\pi$ is given by $\wt{z}^k = z$ for some $ k \in \N$, $ k > 1$. If
 \begin{equation*}
		\phi(z) dz^2 = \lp \frac{a}{z} + h(z)  \rp dz^2
\end{equation*}
where $h$ is holomorphic, we have
	\begin{align*}
		 \lp \frac{a}{z} + h(z)  \rp dz^2 
		 &= k^2 \lp a \wt{z}^{k - 2} + \wt{z}^{ 2k - 2} h(\wt{z}^k)  \rp d\wt{z}^2,
	\end{align*}
so $\Phi(\wt{u})$ is holomorphic on all of $S$, and therefore $\wt{u}$ solves (\ref{hme2}) with respect to the pullback metrics.

Note that by Lemma \ref{covering}, the genus of $S$ is at least 2, so we are in the right situation to apply the preceding argument.  Specifically, suppose that $ w \sim_{\qq} u$.  Then the lifts $\wt{u}$ and $\wt{w}$ are homotopic and we have
	\begin{align*}
		\lp \mbox{ \# of sheets } \rp E(u, g, G) &= E(\wt{u}, \pi^*g, \pi^*G) \\
		&\leq E(\wt{w},\pi^*g, \pi^*G) \\
		&= \lp \mbox{ \# of sheets } \rp E(w, g, G) 
	\end{align*}
with equality if and only if $\wt{u} = \wt{w}$, i.e.\ $ u = w$.

\end{proof}


\section{Linear Analysis}\label{pertsection1}

\subsection{The linearization}\label{sec:linearization}

We now explicitly compute the derivative of $\tau$ at $u_0$ in the
$\Bb^{1 + \e}(u_0)$ direction.  That is, given a solution $(u_0, g_0,
\Gfixed)$ to (\ref{hme2}) with $u_0$ in Form \ref{uform}, and a path
$u_t \in \Bb^{1 + \e}(u_0)$ through $u_0$ with $\ddt u_{t} = \psi$, we compute
	\begin{align*}
		\Ll_{u_0, g_0, \Gfixed}\psi &:= \left. \frac{d}{dt} \right|_{t=0} \tau( u_{t} , g_0, \Gfixed ).
	\end{align*}
For any $u \in \Bb^{1 + \e}(u_0)$, near $p \in \pp$, write
	\begin{align*}
		u_0 &= \lambda z + v \\
		u &= u_0 + \wt{v}, \\
		v, \wt{v} &\in r^{1 + \e} C^{2, \g}_b.
	\end{align*}
From \eqref{hmec}, we have
	\begin{align*} 
\lp  \frac{\sigma}{4}\rp\tau(u_0 + \wt{v}, g_0, \Gfixed) &= \p_z \p_{\ov{z}} \lp u_0 + \wt{v} \rp+ \frac{\p \log \rho_0(u)}{\p u} \p_z \lp u_0 + \wt{v} \rp \p_{\ov{z}} \lp u_0 + \wt{v} \rp \\
		&=  \tau(u_0, g_0, \Gfixed)  \\
		&\qquad + \p_z \p_{\ov{z}}   \wt{v} + \lp  \frac{\p \log \rho_0(u)}{\p u}   -  \frac{\p \log \rho_0(u_0)}{\p u}  \rp  \p_z  u_0 \p_{\ov{z}} u_0 \\
		&\qquad +  \frac{\p \log \rho_0(u)}{\p u} \lp \p_z  u_0 \p_{\ov{z}} \wt{v} + \p_z  \wt{v} \p_{\ov{z}} u_0 + \p_z  \wt{v} \p_{\ov{z}} \wt{v} \rp.
	\end{align*}
Since $\tau(u_0, g_0, \Gfixed) = 0$, 
	\begin{equation}
	\begin{split}\label{label}
	\lp  \frac{\sigma}{4}\rp\tau(u_0 + \wt{v}, g, \Gfixed)&=
        \p_z \p_{\ov{z}}   \wt{v} + \lp  \frac{\p \log \rho_0(u)}{\p
          u}   -  \frac{\p \log \rho_0(u_0)}{\p u}  \rp  \p_z  u_0
        \p_{\ov{z}} u_0 \\
		&\qquad +  \frac{\p \log \rho_0(u)}{\p u} \lp \p_z  u_0
                \p_{\ov{z}} \wt{v} + \p_z  \wt{v} \p_{\ov{z}} u_0 +
                \p_z  \wt{v} \p_{\ov{z}} \wt{v} \rp.
	\end{split}	\end{equation}
If we have $u_t$ with $\ddt u_t = \psi$, then
	\begin{equation} \label{Lll}  \begin{split}
\frac{\sigma}{4} \Ll \psi &=\p_z \p_{\ov{z}}  \psi + \frac{\p \log
  \rho_0(u)}{\p u} \lp \p_z  u_0 \p_{\ov{z}} \psi + \p_{\ov{z}} u_0
\p_z  \psi  \rp  + A \psi 
	\end{split} \end{equation}
where
	\begin{align} \nonumber
	\frac{\sigma}{4} A \psi &:=  \p_z  u_0 \p_{\ov{z}} u_0 \cdot
        \ddt  \frac{\p \log \rho_0(u_t)}{\p u} \\ \label{AAA}
	&= 2  \p_z  u_0 \p_{\ov{z}} u_0 \lp \frac{\p^2 \mu}{\p u \p
          \ov{u}} \ov{\psi} + \frac{\p^2 \mu}{\p u^2 } \psi\rp +  \p_z
        u_0 \p_{\ov{z}} u_0 \frac{\al - 1}{u_0^2} \psi.
	\end{align}
As a preliminary to the analysis below, we can now see that
$\Ll$ acts on weighted H\"older spaces  (see \eqref{c2gb2}):
\begin{equation}
  \label{eq:holderbounded}
  \Ll :  r^{c} \Xx^{2, \g}_b \lra
r^{c - 2\al} \Xx^{0, \g}_b
\end{equation}
In fact, if we define the operator $\wt{L} :=  \frac{\sigma \absv{z}^2}{ 4} \Ll$
then
\begin{equation}
  \label{il0}
  	\wt{L}\psi = 	I(\wt{\Ll})\psi  + E(\psi)
\end{equation}
where
	\begin{align}\label{il}
		I(\wt{\Ll}) \psi &= \lp z\p_z \rp \lp
                \ov{z}\p_{\ov{z}} \rp  \psi + \lp \al - 1 \rp \ov{z}
                \p_{\ov{z}} \psi 
	\end{align}
and $E$ is defined (locally near $\pp$) by this equation.  It follows immediately that
\begin{equation}
  \label{eq:holderbounded}
  I(\wt{\Ll}) :  r^{c} C^{2, \g}_b \lra
r^{c} C^{0, \g}_b
\end{equation}
Furthermore, using $\absv{z}\absv{\frac{ \p \mu}{ \p u}} +  \absv{\p_z
  v} + \absv{\p_{\ov{z}} v}  = \mc{O}\lp \absv{z}^{\delta} \rp$
 we see that if
$\psi \in r^{c}C^{2, \gamma}_{b}
$ near $p \in \pp$ then
\begin{equation}
  \label{eq:6}
  E(\psi) \in r^{c + \epsilon}C^{2, \gamma}_{b}.
\end{equation}
From these two equations, \eqref{eq:holderbounded} follows immediately.

For the arguments in section \ref{closedsec} we must also compute the mapping properties of the locally defined operator
	\begin{align*}
		Q \wt{v} &= \tau(u_0 + \wt{v}, g_0, \Gfixed) - L(\wt{v}) 
	\end{align*}
From (\ref{label})-(\ref{AAA}), we have
	\begin{align*}
		\frac{\sigma}{4} \lp \tau(u_0 + \wt{v}, g_0, \Gfixed) - L(\wt{v}) \rp &= \lp \frac{\p \log \rho_0(u)}{\p u}  - \frac{\p \log \rho_0(u_0)}{\p u}  \rp \p_z  u_0  \p_{\ov{z}}  u_0   \\
		&\qquad +  \frac{\p \log \rho_0(u)}{\p u} \p_z \wt{v} \p_{\ov{z}} \wt{v} - A \wt{v}\\
		 &= \lp \frac{\p \log \rho_0(u)}{\p u}  - \frac{\p
                   \log \rho_0(u_0)}{\p u}\rp \p_z  u_0  \p_{\ov{z}}
                 u_0 \\
                 &\qquad - \ddt  \frac{\p \log \rho_0(u_t)}{\p u} \p_z  u_0  \p_{\ov{z}}  u_0   + \lp \frac{\p \mu(u)}{\p u}  + \frac{\al - 1}{ u} \rp \p_z \wt{v} \p_{\ov{z}} \wt{v} 
	\end{align*}
From this formula and the fact that $\wt{v} \in r^{1 + \e}C^{2,
  \g}_b$,  $\mu \in r^{\e}C^{1, \g}_b$, and $\p_{\ov{z}} u \in r^{
  \e}C^{2, \g}_b$ we see that
	\begin{equation} \label{qbound}
		\norm[r^{1 + 2\e - 2\al} C^{2, \g}_b ]{ Q(\wt{v})} < C \norm[r^{1 + \e} C^{1, \g}_b ]{\wt{v}}
	\end{equation}
holds for $\e > 0$ small, $k \in \N$ and $\g \in (0, 1)$ arbitrary.

\subsection{The $b$-calculus package for $\Ll$} \label{bcalcsect}

This section links the study of $L$ to a large body of work on
$b$-differential operators.  For more detailed definitions and proofs
of what follows we refer the reader to \cite{me}.  Fixing conformal
coordinates $z_i$ near each $p_i \in \pp$, we make a smooth function
$r : \Sigma \lra \C$ that is equal to $\absv{z_i}$ in a neighborhood
of each $p_i$.  There is a smooth manifold with boundary $[\Sigma;
\pp]$ and a smooth map $\beta: [\Sigma; \pp] \lra \Sigma$ which is a
diffeomorphism from the interior of $[\Sigma; \pp]$ onto $\Sp$, with
$\bb^{-1}(\pp) = \p [\Sigma; \pp] \simeq \cup_{i = 1}^k S^1$, on which
the map $r \circ \beta$ is smooth up to the boundary.  The space
$[\Sigma; \pp]$ is constructed by radial blow up.  Finally, let  
 \begin{equation} \label{nub}
	\mc{V}_b = \mbox{smooth vector fields on $[\Sigma; \pp]$ which are tangent to the boundary}
 \end{equation}
The Lie algebra $\mc{V}_{b}$ generates a filtered algebra of differential operators called $b$-differential operators.

For a more concrete definition, let $B_{i}$, $i = 1, 2$ be any two vector bundles over $[\Sigma; \pp]$.  Then $P$ is a differential $b$-operator of order $N$ on sections of $B$ if it admits a local expression
	\begin{align} \label{localexpresh}
		P = \sum_{i + j \leq N} a_{i , j} \lp r\p_r \rp^i \p_\theta^j  &\mbox{ where }  
		a_I \in C^\infty \lp [\Sigma ; \pp] ; End(B_{1};B_{2}) \rp
	\end{align}
near the boundary of $[\Sigma; \pp]$.  We will call $P$ \textbf{$b$-elliptic} if for all $(r, \theta)$ and $(\xi, \eta) \in \R^2 - \set{ (0, 0) }$,
 \begin{equation*}
		\sum_{i + j \leq N} a_{i , j}(r, \theta) \xi^i \eta^j
                \mbox{ is invertible.}
\end{equation*}
	
Let $(u, g, G)$ solve (\ref{hme2}) and satisfy Assumption
\ref{assump}.  In \eqref{eq:5}, the operator $\wt{L}$ can be defined
globally using our extension of $r$ and the local conformal
factor $\sigma$ to positive global functions. We immediately see that $\wt{L}$ is
an elliptic $b$-operator, and that
$E$ is a $b$-operator which, in local coordinates as in
(\ref{localexpresh}), has coefficients $a_I$ tending to zero at a polynomial rate.

We will now describe the relevant properties of an
elliptic $b$-differential operator in our context. To begin we define
the set $\Lambda \subset \C$ of indicial roots of $\wt{L}$.
Given $p \in \pp$, we define a set, $\Lambda_{p}$ consisting of all $\zeta \in \C$ such that for some function $a = a(\theta) : S^1 \lra \C$ and some $p \in \N$ 
 \begin{equation} \label{indicialdef}
		\wt{L} r^\zeta a(\theta) = o(r^\zeta).
 \end{equation}
The total set of indicial roots is
\begin{equation}
  \label{eq:indicialroots}
  \Lambda = \set{z \in \C^{n} : z_{i} \in \Lambda_{p_{i}} \mbox{ for
      some } i}.
\end{equation}
We will show in section \ref{indroots} that $\Lambda$ is a discrete
subset of $\R$.  For all $c \not \in \Re \Lambda$ there is a parametrix
$\wt{\mc{P}}_c$ for $\wt{L}$, i.e.\ 
 \begin{equation} \label{parametrix}\begin{split}
					\wt{\mc{P}}_c \circ \wt{\Ll} &= I - R_1 \\
					 \wt{\Ll} \circ \wt{\mc{P}}_c &= I - R_2
\end{split}\end{equation}
for compact operators $R_1$ and $R_2$ on $r^{c} \Xx^{k, \g}_b$ and
$r^{c} \Xx^{k - 2, \g}_b$, respectively.  This holds for all $k \geq
2$.  

The spaces $r^c \Xx^{2, \g}_b$ can be replaced, in the sense that
everything above is still true, by weighted $b$-Sobolev spaces,
defined as follows.  Let $d \mu$ be a smooth, nowhere-vanishing density on $\Sigma_\pp$ so that near $\pp$
 \begin{equation*}
		d\mu = \frac{drd\theta}{r} 
\end{equation*}
Let $\Xx_{L^2_b(d\mu)}$ be the completion of $r^\infty \Xx^{\infty}_b$ (smooth vector fields vanishing to infinite order at $\pp$) with respect to the norm 
 \begin{equation*}
		\norm[L^2_b(d\mu)]{\psi}^{2} = \int_{\Sigma} \norm[u^*G]{\psi}^2 d\mu
\end{equation*} 
As in the H\"older case, $\psi \in r^c \Xx_{L^2_b(d\mu)}$ if and only
if $r^{-c} \psi \in \Xx_{L^2_b(d\mu)}$, and $r^c \Xx_{L^2_b(d\mu)} $
is a Hilbert space with inner product
\begin{equation*}
  \la \psi, \psi' \ra_{r^c \Xx_{L^2_b(d\mu)}} = \la r^{-c}\psi, r^{-c}\psi' \ra_{ \Xx_{L^2_b(d\mu)}}
\end{equation*}
We define the weighted $b$-Sobolev spaces by
\begin{equation}\label{eq:weightedbholder}
	\begin{split}
		\psi \in r^c \Xx_{H^k_b(d\mu)}	&\iff \begin{array}{c}
		\mbox{for all $k$-tuples } V_1, \dots, V_k \in \mc{V}_b,  \\
		\nabla_{V_1} \dots \nabla_{V_k} \psi \in r^c\Xx_{L^2_b(d\mu)}
		 \end{array}
	\end{split}
      \end{equation}

The operators in (\ref{parametrix}) also acting between
weighted $b$-Sobolev spaces, and it follows by standard functional
analysis that $\wt{L}$ is Fredholm, i.e.\ that it has closed range
and finite dimensional kernel and cokernel.  This immediately implies
that 
\begin{equation}\label{eq:sobbounded}
  L : r^{c}\Xx_{H^k_b(d\mu)} \lra  r^{c - 2\Aa}\Xx_{H^{k-2}_b(d\mu)} 
\end{equation}
is Fredholm.  Thinking of $L$ as a
map from the $r^{c}L^2_b(d\mu)$ orthocomplement of its kernel onto its
range, we can define the generalized inverse

 \begin{equation} \label{parametrixr}\begin{split}
					\Gg_c \circ \Ll &= I - \pi_{ker} \\
					 \Ll \circ \Gg_c &= I - \pi_{coker},
\end{split}\end{equation}
where $\pi_{ker}$ (resp$.$ $\pi_{coker}$) is the $r^{c}L^2_b(d\mu)$
(resp$.$ $r^{c - 2\Aa}L^2_b(d\mu)$) orthogonal projections onto the
kernel (resp$.$ cokernel) of $L$.  In particular
\begin{equation} \label{eq:otherbounded}
  \begin{split}
    \Gg_c: r^{c - 2\Aa}\Xx_{H^{k-2}_b(d\mu)} &\lra
    r^{c}\Xx_{H^k_b(d\mu)} \\
    \pi_{ker}: r^{c }\Xx_{H^{k}_b(d\mu)} &\lra
    r^{c}\Xx_{H^k_b(d\mu)}\\
    \pi_{coker}: r^{c - 2\Aa}\Xx_{H^{k-2}_b(d\mu)} &\lra r^{c -
      2\Aa}\Xx_{H^{k - 2}_b(d\mu)}
  \end{split}
\end{equation}
In fact this holds on the weighted $b$-H\"older spaces, so everywhere
$r^{c'} \Xx^{k', \g}_b$ can replace $r^{c'}\Xx_{H^{k'}_b(d\mu)}$ (see
\cite{m}, section 3 for more on the relationship between $b$-H\"older
and $b$-Sobolev spaces).  The fact that \eqref{eq:otherbounded} holds
for both $b$-Sobolev and $b$-H\"older spaces immediately implies the
following lemma, which will be very useful below.
\begin{lemma}\label{thm:cokerker}
For any $c \not \in \Lambda$,
  \begin{equation}
    \label{eq:9}
    \mbox{Ker}\lp L :  r^{c}\Xx_{H^2_b(d\mu)} \lra
    r^{c - 2\Aa}\Xx_{H^{0}_b(d\mu)} \rp = \mbox{Ker}\lp L :
    r^{c}\Xx^{2, \gamma}_{b} \lra
    r^{c - 2\Aa}Xx^{0, \gamma}_{b}  \rp,
  \end{equation}
and $L(r^{c}\Xx^{2, \g}_{b}) \oplus W = r^{c -2\Aa}\Xx^{0, \g}_{b}$, where
  \begin{equation}
    \label{eq:8}
    W = \mbox{Coker}\lp L :  r^{c}\Xx_{H^2_b(d\mu)} \lra
    r^{c - 2\Aa}\Xx_{H^{0}_b(d\mu)} \rp := \lp  L \lp r^{c}\Xx_{H^2_b(d\mu)} \rp\rp^{\perp}
  \end{equation}
where the orthocomplement is computed with respect to the $\la,\ra_{
  r^{c - 2\Aa}\Xx_{H^2_b(d\mu)} }$ inner product.  This last equation
says that both the image of \eqref{mapquest2} and the image of
\eqref{eq:sobbounded} are complimented by $\mbox{Im}\lp\pi_{coker}\rp$
in \eqref{eq:otherbounded}.
\end{lemma}


Finally, we will need to use the fact that approximate solutions admit partial expansions.  To be precise, from Corollary 4.19 in \cite{m}, if $\psi \in r^{c } \Xx^{k, \g}_b $ and $ \Ll \psi \in r^{c + \delta - 2\Aa} \Xx^{k - 2, \g}_b $ for $\delta > 0$, i.e.\ if $\Ll \psi$ vanishes faster than the generically expected rate of $r^{c - 2\Aa}$, then $\psi$ decomposes as
			\begin{align} \label{expdecomp1}
				\psi = \psi_1 + \psi_2
			\end{align}
Where $\psi_1 \in r^{c + \delta } \Xx_b$ and $\psi_2$ admits an
asymptotic expansion, meaning for some discreet $\mc{E} \subset \C
\times \N$ which intersects $\set{ \R z < C}$ at a finite number of
points,
			\begin{align} \label{expdecomp2}
				\psi_2(z) &=  \sum_{(s, p) \in
                                  \mc{E}, c + \delta > s > c}  a_{s, p}(\theta) r^{s} \log^p r,
			\end{align}
where $a_{s, p}: S^1 \lra \C.$
In particular, we have
\begin{lemma}
 Solutions to $\Ll \psi = 0$ have complete asymptotic expansions,
 i.e.\ they are polyhomogeneous (see \ref{thm:phg}).
\end{lemma}
This follows from the improvement of \eqref{eq:otherbounded} in
\cite{m}; if $\mc{A}_{phg}^{\mc{E}}$ denotes the polyhomogeneous
functions with index set $\mc{E}$, then, given $c' > c$, 
\begin{align}
  \label{eq:Gcpolyhomo}
  \Gg_c : r^{c' - 2\Aa}\Xx^{0, \g}_b &\lra   r^{c}\Xx^{2, \g}_b +
 \lp\mc{A}_{phg}^{\mc{E}} \cap  r^{c'}\Xx^{2, \g}_b\rp
\end{align}

We can also use \eqref{eq:Gcpolyhomo} to prove polyhomogeneity for the
solution $u$ 
by applying the parametrix for the linearization of
$\tau$ to the harmonic map equation.   Recall from \eqref{qbound} that
$\tau(u, g, G) = \Ll(u) + Q(u)$.  Since $u$ is harmonic, we have $\Ll(v) = - Q(v)$.
We now want to let our parametrix $\mathcal{G}_{c}$, with $c_{i} = 1 +
\epsilon$ for all $i$, act on both sides of this equation and to use the formula $\mathcal{G}_{c} \Ll = I $. 
This yields 
\begin{equation}\label{expansionrelation}
v =  - \mathcal{G} Q(v)
\end{equation}
We know that, locally $v \in r^{1 + \e} \Xx^{k, \g}_b$. From
(\ref{qbound}) we have
\ben \label{quad}
Q: r^{1 + \e} \Xx^{k, \g}_b &\lra& r^{1 + 2\e - 2\al} \Xx^{k-1, \g}_b
\een
so by \eqref{eq:Gcpolyhomo}, $- \Gg Q(v)  = v_1+ v_2$
where $v_{1} \in \Xx^{1 + 2 \e
}_{k +1, \g} $ and $v_{2} \in \mc{A}_{phg}^{\mc{E}}$.  The full
expansion follows from induction; assuming that $v = v_{1, k} + v_{2,
  k}$, where $v_{1, k} \in  \Xx^{1 + k \e
}_{k +1, \g} $ and $v_{2,  k} \in \mc{A}_{phg}^{\mc{E}} $ we apply the same
reasoning to this decomposition yields to increase the vanishing order
of the term $v_{1, k}$.

\subsection{Cokernel of $L :  r^{1 - \e} \Xx^{2, \g}_b \lra r^{ 1 - \e - 2\Aa } \Xx_b^{0, \g} $}

We continue to let $(u_{0}, g_{0}, G)$ denote a solution to
\eqref{hme2} in Form \ref{uform} satisfying Assumption \ref{assump},
and to let $L$ denote the linearization of $\tau$ at $u_{0}$.

An important step in the proof of Proposition \ref{mainlemma} is an
accurate identification of `the' cokernel of the map 
 \begin{equation} \label{mapquest2}
		\Ll : r^{1 - \e} \Xx^{2, \g}_b \lra r^{ 1 - \e - 2\Aa } \Xx_b^{0, \g}
 \end{equation}
(note the `$-\e$.')  In this section we will prove the following lemma
	\begin{lemma} \label{cokerminuse}
	\begin{align*}
		r^{ 1 - \e - 2\Aa } \Xx^{0, \g}_b = \Ll \lp r^{1 - \e }\Xx^{2, \g}_b \rp \oplus \Kk
	\end{align*}
where
	\begin{align} \label{Kk}
		\fbox{$
		\displaystyle{
		\Kk := \krn{1 + \e - 2\Aa}
		}
		$,}
	\end{align}
and this decomposition is $L^2$ orthogonal with respect to the inner
product in \eqref{pairingdef}.
	\end{lemma}


Given two vector fields $\psi, \psi' \in \Gamma(u_0^* T\Sp)$ which are
smooth and vanish to infinite order near $\pp$, we define the
geometric $L^{2}$ pairing by
	\begin{equation} \label{pairingdef}
          \begin{split}
            \la \psi, \psi' \ra_{L^2} &:= \int_{\Sigma } \left<
              \psi, \psi' \right>_{u_0^*(G)} d\mu_g       = \Re \int_{M} \psi \overline{\phi} \rho(u_{0}) \sigma(z)
            \absv{dz}^2
          \end{split}
	\end{equation}
It is straightforward to check using  $g, G \in
\Mm_{k, \g, \nu}(\pp, \Aa)$ and $u_{0} \sim \lambda z$ that if $\psi \in r^{c}
                \Xx^{0, \g}_b(u_{0})$ and $\psi' \in r^{c'} \Xx^{0,
                  \g}_b(u_{0})$, then
\begin{equation}
  \label{ipfinite}
  c_i + c_i' > 2 - 4 \alpha_{i} \forall i \implies \la \psi, \psi'
  \ra_{L^{2}} < \infty
\end{equation}

Since they are necessary for subsequent functional analytic arguments,
we define weighted $L^2$ spaces of vector fields:
	\begin{definition}
		Let $\Xx_{L^2}$ denote the space with
			\be
	\psi \in \Xx_{L^2} &\iff& \la \psi, \psi \ra_{L^2} = \norm[L^2]{\psi}^2 < \infty
			\ee
		As in section \ref{bcalcsect}, let $r$ denote a
                positive function on $\Sp$ with the property that,
                near $p_{j} \in \pp$, $r =r_{j}$ where $z_{j} =
                r_{j}e^{i\theta_{j}}$ for a centered conformal
                  coordinate $z_{j}$.  Given $a = (a_1, \dots, a_k)$, let $r^a$
                denote a smooth positive function equal to $r_i^{a_i}$ near $p_i$, and
                let $ r^{a} \Xx_{L^2}$ be the space with
			\be
	\psi \in r^{a} \Xx_{L^2}&\iff& \la r^{-a} \psi, r^{-a} \psi \ra_{L^2} := \norm[r^a L^2]{\psi}^2 < \infty
			\ee
	Let $ r^{a} \Xx_{H^1_b} $ denote the space of sections $\psi$ of $\Gamma(u_0^{*} T\Si)$ such that $\| \psi \|_{u_0^*(T\Si)}$ and also $\|\nabla \psi \|_{u_0^*(T\Si)}$ are in $L^2_{loc}$, and near any cone point $r \nabla_{\p_r} \psi, \nabla_{\p_\theta} \psi \in r^{a} \Xx_{L^2}$.
	\end{definition}

It is standard (see e.g. \cite{el}), that the linearization of $\tau$
in (\ref{L}) is symmetric with respect to this inner product and
appears in the formula of the Hessian of the energy functional near a
harmonic map.  For reference, we state this here as 
	\begin{lemma}[Second Variation of Energy] \label{scndvarbdry}
		Let $(M, h)$ and $(N, \wt{h})$ be smooth Riemannian
                manifolds, possibly with boundary, and let $u_{0} :
                (M, h) \lra (N, \wt{h})$ be a $C^2$ map satisfying
                $\tau(u_{0}, h, \wt{h}) = 0$.  If $u_t$ is a variation of
                $C^2$ maps through $u_0 = 0$ with $
                \eval{\frac{d}{dt}}_{t = 0} u_t = \psi$, and $L$ is
                the linearization of $\tau$ at $u_{0}$ (see (\ref{L})), then
		\begin{align*}
			\fbox{ $ \Ll\psi =  \nabla^* \nabla \psi  + \tr_h  R^{\wt{h}}(du, \psi) du $ },
		\end{align*}
where ${\nabla }$ is the natural connection on $u_{0}^{*}(\mbox{T}N)$ induced by $\wt{h}$, and $R^{\wt{h}}$ is its curvature tensor.  If $\p_\nu$ denotes the outward pointing normal to $\p M$, then
 \begin{align}  \label{boundaryterm2}
   \left. \frac{d^2}{dt^2}E(u_t) \right|_{t = 0} 	&= \int_M  \lp
   \left< \nabla \psi ,  \nabla \psi \right>_{u^* \wt{h}}  +  \tr_h
   R^{\wt{h}}(du, \psi, \psi, du) \rp   dVol_h \\ \nonumber
				&\quad + \int_{\p M }  \la \nabla_{u_* \p_\nu} \psi, \psi \ra_{u^* \wt{h}}  ds \\ \label{hess}
                                &= -  \left< \Ll \psi ,  \psi \right>_{L^2} + \int_{\p M }  \lp \la \nabla_{\psi} \psi, u_* \p_\nu \ra_{u^* \wt{h}} + \la \nabla_{u_* \p_\nu} \psi, \psi \ra_{u^* \wt{h}} \rp ds
\end{align}
	\end{lemma}
\noindent Note that the boundary term in the last line can be expressed in terms of the Lie derivative
 \begin{equation} \label{hessbterm}
		\int_{\p M }  \lp \la \nabla_{\psi} \psi, u_* \p_\nu \ra_{u^* \wt{h}} + \la \nabla_{u_* \p_\nu} \psi, \psi \ra_{u^* \wt{h}} \rp ds = \int_{\p M }  \mc{L}_{u_*\psi}\wt{h}(u_*\psi, u_*\p_\nu) ds.
 \end{equation}
We can also write down a necessary condition on the weights of
variations in the $b$-H\"older spaces so that the Hessian of energy is
given by the quadratic form corresponding to $L$.  	If $ \psi \in r^{1 - \Aa + \e}\Xx^{2, \g}_b$ and $u_t \in \Bb^{1 - \Aa + \e}$ has $\eval{\frac{d}{dt}}_{t = 0} u_t = \psi$, then
 \begin{equation*}
				\ddtt E(u_t, g, G) = - \la \Ll \psi , \psi \ra_{L^{2}}
\end{equation*}
As for the symmetry of $L$ with respect to the $L^{2}$ inner product, if $\psi \in  r^{c}\Xx^{2, \g}_b$ and  $\psi' \in  r^{c'}\Xx^{2, \g}_b$
and $c_{i} + c_{i}' > 2\al_{i} - 2$, then
\begin{equation}
  \label{thm:intbyparts}
  \la \Ll \psi , \psi' \ra_{L^{2}} = \la  \psi , \Ll\psi' \ra_{L^{2}}
\end{equation}
The proofs are a simple excercise in counting orders of decay, using
our assumptions on the $u_{0}, g_{0}$ and $G$ (and $\Gamma \sim r^{-1}$.)

We will use the fact that $L$ is symmetric with respect to the $L^{2}$
pairing to prove a relationship between the dimension of the cokernel
and the dimension of the kernel of $\Ll$ on these weighted spaces.
Given a constant $\delta \in \R$ and $c = (c_1, \dots, c_k) \in \R^k$,
let $c + \delta = (c_1 + \delta, \dots, c_k + \delta)$.  
	\begin{lemma}  \label{pairLemma}
		Again, let $(u_0, g_0, \Gfixed)$ solve (\ref{hme2}) and satisfy Assumption \ref{assump}.  If $c \in \R^k$ and if $\psi \in \Xx^{c + \Aa}_{H^1_b},  \psi' \in \Xx^{-c + \Aa}_{H^1_b}$, we have
			\begin{equation} \label{p2}
				\left< \Ll \psi, \psi' \right> = \left<  \psi, \Ll \psi' \right>
			\end{equation}
and the bilinear form 
			\begin{align*}
				r^{c + \Aa} \Xx_{H^1_b} \times r^{- c + \Aa} \Xx_{H^1_b} &\lra \R \\
				\psi, \psi' &\mapsto \left< \Ll \psi, \psi' \right> 
			\end{align*}
is continuous and non-degenerate.


			\end{lemma}
	
\begin{proof}

The fact that the equation $\left< \Ll \psi,
  \psi' \right> =  \left<  \psi,  \Ll \psi' \right> $ holds for $\psi,
\psi' \in r^{\infty} \Xx^{2, \g}_b $ follows from Lemma
\ref{scndvarbdry}.  All of our work up to this point then implies that
both sides of the equation are continuous with respect to the stated
norms, so part one is proven.

\end{proof}

We can now prove Lemma \ref{cokerminuse}.

\begin{proof}[Proof of Lemma \ref{cokerminuse}]

Lemma \ref{pairLemma} immediately implies that
	\begin{equation} \label{eq:cokerker1}
		\Ker  \Ll \rvert_{r^{c + \Aa}\Xx_{H^1_b}} =\lp \Ll \lp r^{- c + 
      \Aa}\Xx_{H^1_b}  \rp\rp^{\perp}
	\end{equation}
where the right hand side is simply defined to be those $\psi \in r^{c
  + \Aa}\Xx_{H^1_b}$ for which $\la L\psi', \psi\ra = 0$ for all
$\psi' \in r^{-c + \Aa}\Xx_{H^1_b}$.  We need to relate this to the
weighted $b$-Sobolev spaces in \eqref{eq:weightedbholder}.  Note that
	\begin{equation*}
	r^c 	\Xx_{H^1_b\lp d\mu \rp} = r^{c + 2\Aa
          - 1} \Xx_{H^1_b}.
	\end{equation*}
It follows from this that
\begin{equation}\label{eq:cokerker2}
  \begin{split}
      \Ker  \Ll \rvert_{r^{c + \Aa }\Xx_{H^1_b}} &=\mbox{Ker}\lp L :
      r^{c -\Aa + 1}\Xx_{H^2_b(d\mu)} \lra
    r^{c - 2\Aa}\Xx_{H^{0}_b(d\mu)} \rp \\
     \lp \Ll \lp r^{- c + \Aa}\Xx_{H^1_b}  \rp\rp^{\perp} &=
\mbox{Coker}\lp L :  r^{- c - \Aa + 1}\Xx_{H^2_b(d\mu)} \lra
    r^{c - 2\Aa}\Xx_{H^{0}_b(d\mu)} \rp 
  \end{split}
  \end{equation}
By \ref{eq:cokerker1} the two left hand sides are equal, thus are
right hand sides, so by Lemma \ref{thm:cokerker} we have that
$L(r^{1 - \epsilon}\Xx^{2, \g}_{b}) + \Kk = r^{1 -
  \epsilon-2\Aa}\Xx^{0, \g}_{b}$, where
  \begin{equation}
    \label{eq:8}
    \Kk= \mbox{Ker}\lp L :  r^{1 + \epsilon - 2\Aa}\Xx_{H^2_b(d\mu)}
    \lra r^{1 + \epsilon - 4\Aa}\Xx_{H^{0}_b(d\mu)} \rp.
\end{equation}
The proof will be finished if we can show that the sum is $L^{2}$
orthogonal.   To do this, we use the fact from the last paragraph of section \ref{bcalcsect}
that each $\psi \in \Ker \Ll \rvert_{r^{1 + \epsilon - 2\Aa} \Xx^{2, \g}_b}$
in fact satisfies $\psi \in r^{c + 1 - \Aa + \delta} \Xx^{2, \g}_b$
for some small $\delta > 0$ (proof:  such $\psi$ admits a
polyhomogeneous expansion in the indicial roots.)  By
\eqref{ipfinite}, we have that $\la L \psi, \psi'\ra$ is finite and
Lemma \ref{thm:intbyparts}, $\la L \psi, \psi'\ra = \la  \psi, L
\psi'\ra = 0$.  This completes the proof.

\end{proof}

\subsection{The indicial roots of $L$} \label{indroots}

We now compute the indicial roots of $L$, defined in (\ref{indicialdef}).  It is sufficient to find all $r$-homogeneous solutions to
 \begin{equation*}
		I(\wt{\Ll}) \wt{\psi} = 0,
\end{equation*}
where $I(\wt{\Ll})$ defined as in (\ref{il}).  Fixing $p \in \pp$, we
separate variables; let $\wt{\psi} = r^s a(\theta)$ be a solution, and
write $a(\theta) = \sum a_j e^{ij\theta}.$  We have
	\begin{align*}
		I(\wt{\Ll}) r^s e^{ij\theta} &= 
		e^{i j \theta} r^{s} \lp s^2 + 2 \lp \al - 1 \rp  s - j^2 - 2 \lp \al - 1\rp j \rp,
	\end{align*}
which equals to zero if and only if $s \in \set{ j, 2 - 2\al - j}$.
Setting
 \begin{equation} \label{indrootsp}
		\Lambda_{p} = \bigcup_{j \in \Z}  \set{ j, 2 - 2\al - j}
 \end{equation}
the indicial roots of $\wt{L}$ is the union of these sets, \eqref{eq:indicialroots}.

We can now say something quite precise about the leading asymptotics
of solutions
\begin{lemma}[Leading order asymptotics]\label{thm:loa}
  Notation as above, near a cone point $p$ of cone angle $2 \pi \al >
  \pi$, we have $ u = \lambda z + v$ where $v \in r^{3  -
    2\alpha}C^{2, \gamma}_{b} \cap \mc{A}_{phg}^{\mc{E}}$.
\end{lemma}
\begin{proof}
  We know that $u$ is polyhomogeneous by the end of the
  section \ref{bcalcsect}.  By \eqref{label}, the leading order term of $v$, call it
  $v_{0}$, satisfies $I(\wt{L})v_{0} = 0$.
  Since $v_{0} \in r^{1 + \e}C^{2, \gamma}_{b}$, this implies that it
  in fact vanishes to the first order indicial root bigger than $1 + \epsilon$.
  By \eqref{indrootsp}, this is at least as big as $3 - 2\alpha$.
\end{proof}

\subsection{Properties of Jacobi fields} 

Solutions to $L\psi = 0$ are called \emph{Jacobi fields}.  Using the
previous section's analysis, we can now describe the behavior of
Jacobi fields in $\mc{K}$ (see \eqref{Kk}) near the conic set.

\begin{lemma}[Structure of Singular Jacobi Fields] \label{stfsjfs}
Let $\psi \in\Kk$ (see (\ref{Kk})), let $p_i \in \pp$, and let $z$ be
conformal coordinates near $p_{i}$.  If $2\pi \al_i < \pi$, then
 \begin{equation} \label{expsmall}
		\psi(z) = \mu_i z + \mc{O}(\absv{z}^{1 + \delta})
 \end{equation}
for some $\mu_{i} \in C$, $\delta > 0$.

If $2 \pi \al_i > \pi$, and the solution $u_{0}(z) = z + v$, then
 \begin{equation} \label{expbig}
		\psi(z) = w_i + \frac{ \ov{a_i} }{1 - \al}
                \absv{z}^{2(1 - \al)} + \mc{O}\lp \absv{z}^{2(1 - \al)
                  + \delta} \rp
 \end{equation}
for some $w_{i}, a_{i} \in \C$, $\delta > 0$.  

\end{lemma}
\noindent The choice of coefficient of $\absv{z}^{2(1 - \al)}$ simplifies subsequent formulae.

\begin{proof}
This follows essentially from the analysis in sections \ref{bcalcsect}
and \ref{indroots}.  An arbitrary solution $\Ll \psi = 0$ has an
expansion near each cone point.  Fix $p \in \pp$.  By (\ref{il}), the
lowest order homogeneous term, which we can write as $f(r, \theta) = r^\delta
a(\theta)$ for $a : S^1 \lra \C$, must solve $I(\wt{\Ll}) f = 0$,
and so $\delta$ must be an indicial root, i.e.\ $\delta \in
\Lambda(p)$.  If $p \in \pl$, then the smallest such indicial root is
$1$.  The eigenvector is $\lambda z$, so the lemma is proven in this case.
If $p \in \pg$, the smallest such indicial root is $0$, and the eigenvector
for this indicial root is a constant complex number $w_i$.  

We could continue the proof by analyzing \eqref{Lll}, but we prefer to
kill two birds with one stone by analyzing the Hopf differential of $\psi$.
  Solutions to $\Ll \psi = 0$ also have holomorphic quadratic differentials: their coefficients are the linearization of the coefficients for harmonic maps in (\ref{hopfcoeff}).  Precisely, suppose $\eval{\frac{d}{dt}}_{t = 0}u_t = \psi,$ and define
 \begin{equation} \label{jachopfdiff}\begin{split}
		\Phi(\psi) &:= \ddt \rho(u_t) \p_z u_t \p_z \ov{u}_t dz^2 \\
		&= \lp \lp \frac{ \p \rho}{ \p u} \psi + \frac{ \p \rho}{ \p \ov{u}} \ov{\psi} \rp  \p_{z}u_0 \p_z \ov{u}_0 + \rho(u_0) \lp \p_z \psi \p_z \ov{u}_0+  \p_z \ov{\psi} \p_z u_0  \rp  \rp dz^2
\end{split}\end{equation}
Using (\ref{Lll}) one can easily check that $\Phi(\psi)$ is
holomorphic on $\Sigma_{\pp}$.  Near $p \in \pl$, it is easy to show
that $\Phi(\psi)$ extends holomorphically over $p$.  Near $p \in
\pg$, since $\psi = w + \psi_{0}$ for
$\psi_{0} \in r^{c}C^{2,
  \g}_{b}$ for some $c > 0$ and $v \in r^{3 - 2\al}C^{2,
  \g}_{b}$ (Lemma \ref{thm:loa}), $\Phi(\psi)$ has at most a simple pole. Suppose that
	\begin{align}  \label{resis}
		\Res \rvert_{p} \Phi(\psi) = a
	\end{align}
We claim that the only term in \eqref{jachopfdiff} that contributes to
this residue is $\rho(u_0) \p_z \ov{\psi} \p_z u_0$.  To see this,
note that
 \begin{equation*}
		\frac{\p \rho}{\p u} = \mc{O}\lp \absv{z}^{2\al - 3} \rp,
\end{equation*}
so by Form \ref{uform} and the assumption $\al > 1/2$
	\begin{align*}
		\frac{\p \rho}{\p u} \psi \p_{z} u_0 \p_z u_0 &= \mc{O}\lp \absv{z}^{2\al - 3 + 1 + \e} \rp = \mc{O}\lp \absv{z}^{-1 + \delta} \rp
	\end{align*}
for some $\delta > 0$, so this term does not contribute to the residue.
Similarly, $\frac{ \p \rho}{ \p \ov{u}} \ov{\psi}  \p_{z} u_0 \p_z \ov{u}_0$
does not contribute.  Using \ref{thm:loa}, the other term is
term is
\begin{align*}
  \rho(u_0) \p_z \psi \p_z \ov{u}_0 &= \mc{O}\lp r^{2\al - 2 }\rp
  \mc{O}\lp r^{\delta - 1 }\rp  \mc{O}\lp r^{2 - 2\al }\rp =  \mc{O}\lp r^{\delta - 1}\rp
\end{align*}

Finally, let $\absv{z}^{\delta}b(\theta)$ be the lowest
order homogeneous term in $\psi_{0} = \psi - w$.  Then the leading
order part of $\rho(u_0) \p_z \ov{\psi} \p_z u_0$ is $ \absv{z}^{2(1 -
  \alpha)} \p_z \ov{\absv{z}^{\delta}b(\theta)} $, which equals to
$\ov{\absv{z}^{\delta}b(\theta)}$.  This implies that 
\begin{align*}
  \absv{z}^{\delta}b(\theta) =\frac{\ov{a}}{1 - \alpha}
  \absv{z}^{2(\alpha - 1)},
\end{align*}
and the proof is complete.

\end{proof}

\begin{remark}
  Note that, as a consequence of the proof, if we (easily) choose conformal
  coordinates so that $u_0$ in Form \ref{uform} has $\lambda = 1$, then
  \begin{align} \label{resis} \fbox{$ \displaystyle{ \Res \rvert_{p}
        \Phi(\psi) = a } $}
  \end{align}
\end{remark}

Writing $\pg = \set{ q_1, \dots, q_n}$, the map
	\begin{align*}
		\Res : \Kk &\lra \C^{\absv{\pg}} \\
		\psi &\longmapsto \Res \Phi(\psi) = (\Res \rvert_{q_1}\Phi(\psi), \dots, \Res \rvert_{q_n}\Phi(\psi) )
	\end{align*}
is obviously linear.  We define a basis of $\Kk$, $J_{a^1}, \dots, J_{a^{m_{1}}}, C_1, \dots, C_{{m_{2}}}$ with $a^j \in \C^n$, so that $C_1, \dots, C_{{m_{2}}}$ is a basis of $\ker (\Res: \Kk \lra  \C^{\absv{\pg}})$, i.e.\
 \begin{equation} \label{Ci}
				\Res C_j = 0 \in \C^n
 \end{equation}
and
 \begin{equation} \label{Jai}
				\Res J_{a^j} = a^j \in \C^n.
 \end{equation}
It follows that the $a^j$ are linearly independent.

%

%


%

\subsection{The $C_j$ are conformal Killing fields} \label{confkillsect}

\subsubsection{The Hessian of energy on $\mc{K}$} \label{scndvarintro}

We will prove that the $C_j$ are conformal Killing fields.  The most
important step in the proof is that they are zeros of the Hessian of
the energy functional, and we begin by proving this.
	\begin{lemma} \label{cokhess}
		Let $(u_0, g_0, \Gfixed)$ solve (\ref{hme2}) and satisfy
                Assumption \ref{assump}.  Let
                $C \in \Kk$ have $\Res C = 0 \in \C^n$.  By Lemma \ref{stfsjfs}, we can find $u_t \in \Bb^{c} \circ \Vv \circ \Tt_{> \pi} $, $\e > 0$
		 so that $\eval{\frac{d}{dt}}_{t = 0} u_t = C$, where
			\begin{equation} \label{ci1} 
                          \begin{split}
                            c_i > 1 &\mbox{ for } p_i \in \pl \\
                            c_i > 2 - 2 \al_i &\mbox{ for } p_i \in
                            \pg.
                          \end{split}
			\end{equation}
Then
			\begin{align} \label{scndvarjac}
				\eval{ \frac{d^2}{dt^2}}_{ t = 0} E(u_t, g_0, \Gfixed) = 0
			\end{align}
	\end{lemma}

The proof hinges on a nice cancellation of boundary terms related to
the conformal invariance of energy.  We begin by proving equation (\ref{intoetilde}) which illustrates this phenomenon in a simpler setting.  

\begin{proof}[Proof of (\ref{intoetilde}):]
We are given a $u \in \Bb^{1 + \e}(u_0) \circ \Vv \circ \Tt_{\pg}$,
and we assume that $\pl = \varnothing$ and genus $\Sigma = 1$.  In particular, near each $q \in u^{-1}(\pp)$ we can choose conformal coordinates so that $u \sim \lambda z$.  Let $C \in \CK$.  Choose $f_t \in \mc{C}onf_0$ with $\eval{\frac{d}{dt}}_{t = 0} f_t = C$ and consider
	\begin{align*}
		\eval{\frac{d}{dt}}_{t = 0} E(u \circ f_t, g, G),
	\end{align*}
which is zero by (\ref{intoetilde0}). By lifting to the universal cover we can assume that 
 \begin{equation*}
		f_t(z) = z - tw
\end{equation*}
for some fixed $w$.  Let
 \begin{equation} \label{confdiscw}
	D_w(r) = \set{ z : \absv{ z - w } < r },
 \end{equation}
be the conformal disc (not necessarily a geodesic ball), so $f_t \lp D_{tw}(r) \rp = D_0 (r)$.  In particular, for
$\delta > 0$ sufficiently small, $f_t \lp  \Sigma- D_{tw}(\delta) \rp
= \Sigma - D_{0}(\delta) $.  For the moment we drop the geometric data from the notation and let
 \begin{equation*}
		E(w, A) = \int_A e(w, g, G) \sqrt{g} dx
\end{equation*}
for any $A \subset \Sigma$.
By the conformal invariance of the energy functional, we have
	\begin{align*}
		E(u \circ f_t, \Sigma - D_{tw}(\delta)) &= E(u,  \Sigma - D_{0}(\delta) ) \\
		E(u \circ f_t, D_{tw}(\delta)) &= E(u, D_{0}(\delta) )
	\end{align*}
for all $t$.  Thus the functions $ E(u \circ f_t, \Sigma )$,  $E(u \circ f_t, \Sigma - D_{tw}(\delta))$, and $ E(u \circ f_t, D_{tw}(\delta))$ are all constant in $t$.  In particular
	\begin{align*}
		0 &= \eval{\frac{d}{dt}}_{t = 0} E(u \circ f_t, \Sigma) \\
		&= \eval{\frac{d}{dt}}_{t = 0} E(u \circ f_t, \Sigma - D_{tw}(\delta)) + \eval{\frac{d}{dt}}_{t = 0} E(u \circ f_t, D_{tw}(\delta)) \\
		&= \eval{\frac{d}{dt}}_{t = 0} E(u \circ f_t, \Sigma - D_{tw}(\delta)) 
	\end{align*}
We can also evaluate this last expression using the first variation
formula \eqref{frstvarbdry} and the chain rule to get the expression
	\begin{equation} \label{crappy0}
          \begin{split}
            \eval{\frac{d}{dt}}_{t = 0} E(u \circ f_t, \Sigma -
            D_{tw}(\delta)) &= - \int_{\Sigma - D_0(\delta)} \la \tau(u, g, G) ,
            u_*C \ra_{u^*G} dVol_g \\ 
            &\quad+ \int_{\p
              D_0(\delta)} \la u_* \p_\nu , u_*C \ra_{u^*G} ds
            \\ 
            &\quad+ \eval{\frac{d}{dt}}_{t = 0} \lp
            \int_{\Sigma - D_{tw}(\delta)} e(u, g, G) \sqrt{g}dx \rp
          \end{split}
	\end{equation}
Here $\p_\nu$ is the outward pointing normal from $\Sigma - D_0(\delta)$.
The last integral satisfies
 \begin{equation} \label{jesus}
		\eval{\frac{d}{dt}}_{t = 0} \lp \int_{\Sigma - D_{tw}(\delta)}  e(u, g, G) \sqrt{g}dx \rp =  - \int_{\p D_{tw}(\delta)}  e(u, g, G)  \la \p_\nu, u_*C \ra_g ds
 \end{equation}
To compare this integral to (\ref{crappy0}), we use the decomposition (\ref{pbm}), $u^*G = e(u)g + u^*G^\circ$.  
 \begin{equation*}
		\int_{\p D_0(\delta)} \la u_* \p_\nu , u_*C \ra_{u^*G} ds =
                \int_{\p D_0(\delta)}   e(u, g, G) \la  \p_\nu , u_*C \ra_{g} ds
                + \int_{\p D_0(\delta)} \la \p_\nu , u_*C \ra_{u^*G^\circ} ds
\end{equation*}
By (\ref{jesus}), the first term on the right exactly cancels the last
term in (\ref{crappy0}).  For the second term, note that although
$u^*G^\circ$ is not holomorphic, the bound from (\ref{hopfbound})
still holds, so
 \begin{equation*}
		 \int_{\p D_0(\delta)} \la \p_\nu , u_*C \ra_{u^*G^\circ}  ds = \int_{0}^{2\pi} \mc{O}\lp r^{2\al - 2 + \e} \rp r d\theta,
\end{equation*}
which goes to zero since $\al > 1/2$.  Looking back at (\ref{crappy0}), we have proven that
 \begin{gather*}
		 \lp \eval{\frac{d}{dt}}_{t = 0} E(u \circ f_t, \Sigma -
                 D_{tw}(\delta)) \rp + \int_{\Sigma - D_0(\delta)} \la \tau(u, g, G)
                 , u_*C \ra_{u^*G} dVol_g \to 0 \\
                 \mbox{as } \delta \to 0
\end{gather*}
This proves \eqref{intoetilde}

\end{proof}
\begin{remark}
  \label{rm:coneanglepi}
  Note that this last proof works if $\al > 1/2$ is replaced by $\al
  \geq 1/2$.  This will be important when we deal with that case in
  section \ref{coneanglepisect}.
\end{remark}

A similar cancellation of boundary terms will lead to the proof of
Lemma (\ref{cokhess}).  Here we take two derivatives, so the relevant
boundary terms look slightly different.  To illustrate this, let $g = \sigma \absv{dz}^2$
be a conformal metric on $\C$ with finite area, and let $T_t = z - tw$, and $\eval{\frac{d}{dt}}_{t = 0} T_t = C (= w)$.  Notation as above
 \begin{equation*}
		\eval{ \frac{d^2}{dt^2} }_{ t = 0 } E(T_t, \C - D_{tw}(\delta)) = 0
\end{equation*}
As above, a direct computation of the second derivative using \eqref{scndvarbdry} and the chain rule will produce boundary terms which must cancel one another.  If we let $e_t = e(T_t, \absv{dz}^2, g)$, and let $\p_{\nu_t}$ be the outward pointing normal to $\C - D_{tw}(\delta)$, then a simple computation using (\ref{hess})--(\ref{hessbterm}) and the product rule shows that 
	\begin{equation} \label{cancelling2}
          \begin{split}
            \eval{ \frac{d^2}{dt^2} }_{ t = 0 } E(T_t, \C - D_{tw}(\delta)) &= 		\eval{ \frac{d^2}{dt^2} }_{ t = 0 } E(T_t, \C - D_{tw}(\delta))  \\
            &= \int_{\p D_0(\delta)} \mathcal{L}_{\dot{T}_0} g \lp \dot{T}_0, \p_\nu \rp ds - 2 \int_{\p D_0(\delta)}  \dot{e}_0 \la \dot{T}_0, \p_\nu \ra_g ds \\
            &\quad - \int_{\p D_0(\delta)} e_0 \eval{ \frac{d}{dt} }_{t = 0}
            \la\dot{T}_{-t}, \p_{\nu_{-t}} \ra_{g(T_{-t})} ds.
          \end{split}
	\end{equation}
Thus the expression on the right \emph{must be equal to zero}. 
%
%

%

%


%

%
\begin{proof}[Proof of Lemma \ref{cokhess}]
Assume that $u_0 = id$.  By (\ref{confenergy}) and conformal invariance, we may replace $g$ by $g / e(u_0)$ and assume that
 \begin{equation} \label{energyis1}
		e(u_0) \equiv 1.
 \end{equation}

We arrange it so that
	\begin{align*}
		u_t &= \wt{u}_t \circ T_t \\
		\wt{u}_t &\in r^{2(1 - \al)}C^{2, \g}_{b} \mbox{ near
                } p \in \pp \\
		T_t &\in \mc{T}_{> \pi}, 
	\end{align*}
with $c$ as in (\ref{ci1}), so $T_t (z_i) = z_i - t w_i$ near $p_i \in \pp_{ > \pi}$.  Define
 \begin{equation*}
		D_{i, t}(\delta) = D_{t w_i}(\delta) \mbox{ in conformal coordinates $z_i$ near $p_i$ },
\end{equation*}
where $D_{t w_i}(\delta)$ is the conformal disc defined in \eqref{confdiscw}.  We can then write
 \begin{align*}
			\eval{ \frac{d^2}{dt^2}}_{ t = 0} E\lp u_{t},
                        \Sigma \rp &= 	\underbrace{\eval{
                            \frac{d^2}{dt^2}}_{ t = 0} E\lp u_{t},
                          \Sigma  - \bigcup_{p_i \in \pp_{ > \pi}}
                          D_{i, t}(\delta)  \rp}_{ : = A(\delta)} +
                        \underbrace{\eval{ \frac{d^2}{dt^2}}_{ t = 0}
                          E\lp u_{t}, \bigcup_{p_i \in \pp_{ > \pi}}
                          D_{i, t}(\delta) \rp}_{ : = B(\delta)}.
\end{align*}

For the term $B(\delta)$ we can use conformal invariance
	\begin{align*}
		E\lp u_{t}, \bigcup_{p_i \in \pp_> \pi } D_{t}(\delta) \rp &= E\lp \wt{u}_{t}, \bigcup_{p_i \in \pp_> \pi } T_{t} \lp D_{t}(\delta) \rp \rp = E\lp \wt{u}_{t}, \bigcup_{p_i \in \pp_> \pi } D_{0}(\delta) \rp.
	\end{align*}
The fact that $ \wt{u}_{t} \in r^{2 - 2 \alpha_{i} + \epsilon}C^{2, \gamma}_{b}$
near $p_{i} \in \pp_{> \pi}$ and the boundary term computation from
section \ref{scndvarintro} show that $B(\delta) \to 0 \mbox{ as } \delta \to 0$.
Note that for fixed $t$ and $\delta$ the integrals $E\lp u_{t}, \Sigma  - \bigcup_{p_i \in \pp_{ > \pi}} D_{i, t}(\delta)  \rp$ are improper since we did not delete balls around the points $ p \in \pl$, but again the boundary contributions from deleted discs here do not contribute to $A(\delta)$.  Let
	\begin{align} \label{sigmaeps}
		\Sigma(\delta) := \Sigma - \cup D_{i, t} (\delta).
	\end{align}
We use the same reasoning as in \eqref{cancelling2} and the fact that
$T_{-1}$ parametrizes the boundary of $\Sigma$ to deduce that
	\begin{align*}
		A(\delta) &= \int_{\Sigma(\delta)} \mc{L}_{C}G( C, \p_\nu )ds  - 2 \int_{\Sigma(\delta)}   \dot{e}_0 \la \dot{T}_0, \p_\nu \ra_g ds \\
		&\quad- \int_{\Sigma(\delta)}  e_0 \eval{ \frac{d}{dt} }_{t = 0} \la \dot{u}_{-t}, \p_{\nu_{-t}} \ra_{g(T_{-t})} ds,
	\end{align*}
where we have once again set $e_t = e(u_t)$.  At this point we use that $C$ is a Jacobi field, so in conformal coordinates $g = \sigma \absv{dz^2}$, 
 \begin{equation*}
		 \mc{L}_{C}G = \dot{e}_0 g + 2 \underbrace{\Re \phi(z) dz^2}_{ \Phi(C)}.
\end{equation*}
We also have that 
	\begin{align*}
		\Res C = 0 \implies \phi(z) dz^2 \mbox{ is smooth on all of } \Sigma.
	\end{align*}
Plugging in and separating everything from the Hopf differential bit, we have (using (\ref{energyis1}))
\begin{equation}  \label{cancelling}
	\begin{split}
		A(\delta) &=\sum_{p \in \pg} \int_0^{2\pi} 2 \lp \Re
                \phi(z) dz^2 \rp \lp C, \p_r \rp \delta d\theta \\
		& \quad + \sum_{p \in \pp_{> \pi}} \int_0^{2\pi}  \lp \dot{e}_0 \la C, \p_r \ra_g - 2  \dot{e}_0 \la \dot{T}_0 , \p_r \ra_g - \eval{ \frac{d}{dt} }_{t = 0} \la \dot{u}_{-t}, \p_{r_{-t}} \ra_{g(T_{-t})} \rp \delta d\theta
	\end{split}
      \end{equation}
Since $\phi$ is bounded, the first term on the right hand side goes to
zero as $\delta \to 0$. 

The point now is that up to terms vanishing
with $\delta$ the bottom line consists exactly of the cancelling terms from
above (\ref{cancelling2}), proving Lemma \ref{cokhess}.  To show this,
we start with the rightmost term. Since $\ddot{T} \equiv 0$ near $ p \in \pp_{ > \pi}$ and $\ddot{\wt{u}}_0 \in r^{2 - 2\Aa + \e} \Xx^{2, \g}_b$, we see that
	\begin{align*}
		\int_{\p \Sigma(\delta)}  e_0 \eval{ \frac{d}{dt} }_{t = 0} \la \dot{u}_{-t}, \p_{\nu_{-t}} \ra_{g(T_{-t})} ds &= 
		\sum_{p \in \pg} \int_0^{2\pi} \eval{ \frac{d}{dt}
                }_{t = 0} \lp \la \dot{T}_t + \dot{\wt{u}}_t ,
                \p_{r_{-t}} \ra_{g(T_{-t})} \rp \delta d\theta \\
                &=\sum_{p \in \pg} \int_0^{2\pi} \eval{ \frac{d}{dt}
                }_{t = 0}  \la \dot{T}_t ,
                \p_{r_{-t}} \ra_{g(T_{-t})}  \delta d\theta \\
		&\quad+\sum_{p \in \pg}\int_0^{2\pi} \eval{ \frac{d}{dt}
                }_{t = 0} \lp \la \dot{\wt{u}}_t ,
                \p_{r_{-t}} \ra_{euc} \sigma(T_{-t})  \rp \delta
                d\theta \\
                &=\sum_{p \in \pg} \int_0^{2\pi} \eval{ \frac{d}{dt}
                }_{t = 0}  \la \dot{T}_t ,
                \p_{r_{-t}} \ra_{g(T_{-t})}  \delta d\theta \\
		&\quad+\sum_{p \in \pg} \int_0^{2\pi}  \mc{O}(\delta^{ 2 - 2 \al + \e })  \sigma(z) \delta d\theta
	\end{align*}
but in the last term, since $\sigma = \mc{O}(\delta^{2(\al - 1) })$,
the integrand is $\mc{O}(\delta^{1 + \epsilon})  d\theta$ so 
\begin{equation} \label{almostdone}
	\begin{split} 
	\int_{\p \Sigma(\delta)}  e_0 \eval{ \frac{d}{dt} }_{t = 0} \la \dot{u}_{-t}, \p_{\nu_{-t}} \ra_{g(T_{-t})} ds
                &=\sum_{p \in \pg} \int_0^{2\pi} \eval{ \frac{d}{dt}
                }_{t = 0}  \la \dot{T}_t ,
                \p_{r_{-t}} \ra_{g(T_{-t})}  \delta d\theta + \mc{O}(\delta^{1 + \epsilon})
	\end{split}
      \end{equation}
Thus as $\delta \to 0$ this approaches the final term in
\eqref{cancelling2}.  Note that in the end $\ddot{\wt{u}}_{0}$
disappears completely.  Next comes the middle term in \eqref{cancelling}.
Using that $T_{0} = id$, $\wt{u}_{0} = id$, we have
 \begin{align*}
		\dot{e}_0 &= \ddt e(T_t) + \ddt e(\wt{u}_t),
\end{align*}
and by \eqref{confenergy}, $\ddt e(\wt{u}_t) = \mc{O}(\delta^{1 +
  \epsilon - 2\alpha})$.  Using this and $\ddot{T}_{0} \equiv 0$ near
$\pp$, we see that
	\begin{align*}
 \sum_{p \in \pp_{> \pi}}\int_0^{2\pi}  - 2  \dot{e}_0 \la \dot{T}_0 ,
 \p_r \ra_g \delta d\theta	&=	  \sum_{p \in \pp_{> \pi}}
 \int_0^{2\pi} \ddt e(T_t)   \la \dot{T}_0, \p_\nu \ra_g  +
 \mc{O}(\delta^{\e}),
	\end{align*}
And this indeed approaches the middle term of \eqref{cancelling2}.
That the first term in \eqref{cancelling} limits to the first term in
\eqref{cancelling2} follows in the same way. This completes the proof
\end{proof}

\subsubsection{Zeros of the Hessian}

We can now prove

%


%

\begin{corollary} \label{positiveindex}
The $C_j$ are conformal Killing fields.
\end{corollary}
\begin{proof}
Assume that $u_0 = id$.  For
genus $\Sigma > 1$, $\mc{C}onf_0 = \set{id}$.  From the decomposition
(\ref{split}) and our work above, it follows that
\begin{equation*}
  	\eval{ \frac{d^2}{dt^2} }_{t = 0} E(u_t, g, \Gfixed) =
        \eval{ \frac{d^2}{dt^2} }_{t = 0} E(u_t, g, H_{1}) + 	\eval{ \frac{d^2}{dt^2} }_{t = 0} E(u_t, g, H_{2}).
\end{equation*}
Note that this is a non-trivial statement, since it is by no means
clear that, for example, the function $t \longmapsto E(u_t, g, H_{1})$
is twice differentiable; but it indeed is, by exactly the same
computations we just preformed.  The fact that conformal maps are global minimizers of
energy now implies that
\begin{equation*}
  \eval{ \frac{d^2}{dt^2} }_{t = 0} E(u_t, g, H_{1}) \geq 0,
\end{equation*}
so since the left hand side of our first equation is zero by Lemma
\ref{cokhess}, we arrive at
\begin{equation*}
  \eval{ \frac{d^2}{dt^2} }_{t = 0} E(u_t, g, H_{2}) = 0.
\end{equation*}
But in fact $E( \cdot, g, H_2)$ if positive definite, as we now shoe.
Cutting out conformal discs $D_{j}(\epsilon)$ as above, near $p \in
\pp_{> \pi}$ the boundary
term in \eqref{boundaryterm2} is 
\begin{align*}
   \lim_{\epsilon \to 0}\int_{\p D_{j} (\epsilon)}  \la \nabla_{u_* \p_\nu} C, C \ra_{u^* H_2}
   ds  &=  \lim_{\epsilon \to 0} \int_{0}^{2\pi}  \p_{\epsilon} \la C, C \ra_{u^* H_2}
   \epsilon d\theta \\
   &= \lim_{\epsilon \to 0} \int_{0}^{2\pi}  \p_{\epsilon} \absv{ w +
     \mc{O}^{2 - 2\al_{j}} } \epsilon d\theta \\
   &= 0.
\end{align*}
Near $p \in \pp_{< \pi}$ the same computation shows that the
contribution is also zero.  We now see from \eqref{boundaryterm2} that
\begin{equation*}
 \int_M  \lp \left< \nabla C ,  \nabla C \right>_{u^* {H_2}}  +  \tr_g
 R^{H_2}(du, C, C, du) \rp   dVol_g = 0,
\end{equation*}
which immediately implies that $C = 0$ by negative curvature. (In
particular $C$ is conformal Killing.)

Now assume the genus of $\Sigma$ is
$1$.  As above we lift to the universal
cover and use (\ref{flatdecomp}).  As in the previous paragraph we
conclude that
\begin{equation*}
   \int_M  \lp \left< \nabla C ,  \nabla C \right>_{u^* {K_2}}  +  \tr_g
 R^{K_2}(du, C, C, du) \rp   dVol_g = 0,
\end{equation*}
but now we can only have flatness and thus can only conclude that
	\begin{align} \label{flatnabla}
		^{ K_2 } \nabla C = 0.
	\end{align}
The lift of $K_2$ to the universal cover, $\wt{K}_2$, written with
respect to the global coordinates $\wt{z}$ on $\C$, is a constant
coefficient metric.  Hence (\ref{flatnabla}) means $C \equiv v$
for some constant vector $v \in \C$, which, as desired, is conformal Killing.  The first statement follows exactly as in the genus $\Sigma > 1$ case.


\end{proof}

\subsection{Proof of Proposition \ref{mainlemma}}

Recall from Lemma \ref{cokerminuse} that the Fredholm map
 \begin{equation*}
		\Ll : r^{1 - \e}\Xx^{2, \g}_b \lra r^{1 - \e - 2\Aa}\Xx^{0, \g}_b
\end{equation*}
satisfies
	\begin{align} \label{pr1}
		\Ll(r^{1 - \e}\Xx^{2, \g}_b) \oplus \Kk = r^{1 - \e - 2\Aa}\Xx^{0, \g}_b,
	\end{align}
and the sum is $L^2$ orthogonal.  This obviously implies
	\begin{align}
		\lp \Ll(r^{1 - \e}\Xx^{2, \g}_b) \oplus \Kk  \rp \cap r^{1 + \e - 2\Aa}\Xx^{0, \g}_b = r^{1 + \e - 2\Aa}\Xx^{0, \g}_b,
	\end{align}	
Since $\Ll : r^{1 + \e}\Xx^{2, \g}_b \lra r^{1 + \e - 2\Aa}\Xx^{0,
  \g}_b$ is also Fredholm for $\e$ small, $\Ll\lp r^{1 + \e}\Xx^{2,
  \g}_b \rp \subset  \Ll(r^{1 - \e}\Xx^{2, \g}_b) \cap r^{1 + \e -
  2\Aa}\Xx^{0, \g}_b $ is a finite index inclusion, and we can find 
	\begin{align} \label{W}
		W \subset  \Ll(r^{1 - \e}\Xx^{2, \g}_b) \cap r^{1 + \e - 2\Aa}\Xx^{0, \g}_b
	\end{align}		
so that
	\begin{align} \label{Wdecomp}
		\lp \Ll\lp r^{1 + \e}\Xx^{2, \g}_b \rp+ W \rp \oplus \mc{K} =  r^{1 + \e - 2\Aa} \Xx^{2, \g}_b.
	\end{align}
We will now use the final paragraph in section \ref{bcalcsect}, which says that vector fields like those in $W$ which are sent via $\Ll$ to vector fields with higher vanishing rates than generically expected admit partial expansions.  Precisely, any vector field $\psi \in W$ satisfies
 \begin{equation} \label{W2}
\psi \in W \implies 
\left\{\begin{array}{cc}
		\psi \in r^{1 + \e - 2\Aa}\Xx^{0, \g}_b \\
		\psi= \Ll\psi' \mbox{ for } \psi' \in r^{1 - \e}\Xx^{0, \g}_b,
              \end{array} \right.
\end{equation}
so from (\ref{expdecomp1}) and (\ref{expdecomp2}) we have $\psi' =
{\psi}_1 + {\psi}_2$, where
	\begin{align*}
	{\psi}_1 &= \sum_{(s, p) \in \Lambda \cap [1 - \e, 1 + \e]} r^s \log^p (r)  a_{s, p}(\theta),
	\end{align*}
and ${\psi}_2 \in r^{1 + \e}\Xx^{2, \g}_b.$
Looking at the eigenvectors in section \ref{indroots} shows that
${\psi}_1(z) = \lambda z$ for some $ \lambda \in \C$.  Thus	$\psi'
\in \mbox{T}_{id} \Vv + r^{1 + \e} \Xx^{2, \g}_b$, which of course
implies
 \begin{equation}
		\Ll(  r^{1 + \e} \Xx^{2, \g}_b ) + W = \Ll(  r^{1 + \e} \Xx^{2, \g}_b + \mbox{T}_{id}\Vv ).
 \end{equation}
Now by (\ref{Wdecomp})
 \begin{equation} \label{cokD}
		 \Ll(  r^{1 + \e} \Xx^{2, \g}_b + \mbox{T}_{id}\Vv ) \oplus \Kk = r^{1 + \e - 2\Aa}\Xx^{0, \g}_b,
 \end{equation}
where $\Kk \perp  \Ll(  r^{1 + \e} \Xx^{2, \g}_b + \mbox{T}_{id}\Vv )$.  Let
 \begin{equation*}
		\pi_{\Kk} = \mbox{projection onto $\Kk$ in (\ref{cokD})}.
\end{equation*}
	
	Now we add $\mbox{T}_{id}\mc{T}_{> \pi}$ to the domain of $\Ll$.  Let $\psi \in \mbox{T}_{id}\mc{T}_{> \pi}$ corresponding to $w \in \C^n$, so near $p_i \in \pg$ we have $\psi \equiv w_i$
Since near $p_i$ we have $\Ll w_i = 0$, we know that
	\begin{align} \label{Lpsi80}
		\Ll \psi \in  r^{1 + \e - 2\Aa}\Xx^{0, \g}_b.
	\end{align}
Using the basis for $\mc{K}$ in (\ref{Ci})--(\ref{Jai}),
\eqref{Lpsi80} implies that we can write
 \begin{equation*}
			\pi_{\Kk} \Ll \psi = \sum_{j = 1}^N \la \Ll \psi, J_{a^i}\ra_{L^{2}} J_{a^i} + \sum_{k = 1}^M \la \Ll \psi, C_k\ra_{L^{2}} C_k,
\end{equation*}
for the $L^{2}$ inner product in \eqref{pairingdefb}. A simple integration by parts using shows that for $\wt{\psi} \in \Kk$ we have
	\begin{align*}
		\la \Ll \psi, \wt{\psi} \ra_{L^2} &= -4\pi \Re
                \sum_{p_i \in \pg} w_i \Res \rvert_{p_i}
                \Phi(\wt{\psi}).
               	\end{align*}
This immediately implies that
	\begin{align*}
		\pi_{\mc{K}} \Ll ( r^{1 + \e} \Xx^{2, \g}_b \oplus \mbox{T}_{id}\Vv \oplus \mbox{T}_{id}\mc{T}_{> \pi}) = \mbox{span}\la J_{a^i} \ra,
	\end{align*}
and thus we have shown that
	\begin{align*}
		\Ll \lp r^{1 + \e} \Xx^{2, \g}_b \oplus \mbox{T}_{id}\Vv \oplus \mbox{T}_{id}\mc{T}_{> \pi} \rp \oplus \CKK =  r^{1 + \e - 2\Aa}\Xx^{0, \g}_b,
	\end{align*}
which is what we wanted.




\subsection{Proof of Proposition \ref{h2prop}}
For the proof of Proposition \ref{h2prop} we will need a formula for the
second variation of energy in the direction of an arbitrary $J_w \in
\mc{J}_{\qq}$.  First, we compute the first variation near a solution \eqref{hme2}.

\begin{proposition} \label{scndvarformula}
Let $(u_0, g_0, \Gfixed)$ solve (\ref{hme2}) and satisfy Assumption
\ref{assump}.  With notation as in the previous section, for any $w
\in \C^{\absv{\qq}}$ we have
 \begin{equation} \label{frstvar}
				\eval{ \frac{d}{dt}}_{ t = 0} E(u_{tw}, g_0, \Gfixed) = \Re \lp 2 \pi i \sum_{p_i \in \pl} \Res \rvert_{p_i} \lp \iota_{J_w} \Phi(u_0) \rp \rp,
 \end{equation}			
and if $\Phi(u_0) = \phi_{u_0} dz^2$
 \begin{equation} \label{frstvar2}
				\Res \rvert_{p_i} \lp \iota_{J_w} \Phi \rp  = w_i \Res \rvert_{p_i} \phi_{u_0}
 \end{equation}
\end{proposition}

	\begin{corollary} \label{hessharm}
		Let $(u_0, g_0, \Gfixed)$ be a solution to
                (\ref{hme2}).  Then $u_0$ minimizes $E(\cdot, g_0,
                \Gfixed)$ in its free homotopy class if and only if it
                solves (\ref{hme2}) with $\qq = \varnothing$, i.e.\ if
                and only if $Phi(u_0)$ is holomorphic on $(\Sigma, g_{0})$.
	\end{corollary}
	\begin{proof}
		That $\Res \Phi(u_0) = 0$ implies the minimizing property is the content of Proposition \ref{umin}.
	
For the other direction, if $(u_0, g_0, \Gfixed)$ solves (\ref{hme2}) and satisfies Assumption \ref{assump}, then
			\begin{align*}
				\Res \Phi(u) \neq 0 \implies u_0 \mbox{ is not energy minimizing}
			\end{align*}
since by (\ref{frstvar}), if $w \in \C^{\absv{\qq}}$ has $\Re 2\pi i \sum_{p_i \in \pl} w_i \Res \rvert_{p_i} \phi_{u_0} \neq 0$ (which is easy to arrange), then 
 \begin{equation*}
				\eval{ \frac{d}{dt}}_{ t = 0} E(u_{tw}, g_0, \Gfixed) \neq 0,
\end{equation*}
contradicting minimality.
	\end{proof}

	\begin{remark} The one form $\iota_{J_w} \Phi(u_0)$ is not
          holomorphic (since $J_w$ is not), but we will show that it still has a residue, meaning that the limit
 \begin{equation*}
		\Res \iota_{J_w} \Phi := \lim_{\e \to 0}  \frac{1}{2 \pi i} \int_{\absv{z} = e}  \iota_{J_w} \Phi (z)
\end{equation*}
exists.  
	\end{remark}

By our assumption that $u_0 = id$ and decomposition (\ref{idecomp}), in conformal coordinates the metric $\Gfixed$ is expressed by
 \begin{equation} \label{decomp2}
		\Gfixed = \sigma \absv{dz}^2 + 2 \Re \phi(z) dz^2.
 \end{equation}
Below, these coordinates are used on both the domain and the target.

	\begin{proof}[Proof of Proposition \ref{scndvarformula}]
		
The proof is similar to the proof of Lemma \ref{cokhess}.  As always assume $u_0 =
id$.  We can write  $u_{tw} = \wt{u}_{tw} \circ T_{tw}$
where $\wt{u}_{tw} \in \Bb^{ 1 + \e}(u_0) \circ \Vv \circ \Tt_{>\pi}$
and $T_{tw} \in \mc{T}_{\qq}$, and $\ddt u_{tw} = J_{w}$ (see\eqref{eq:jw}).
Then	
 \begin{equation*}
				\eval{ \frac{d}{dt}}_{ t = 0} E\lp u_{tw}, \Sigma \rp = 	\underbrace{\eval{ \frac{d}{dt}}_{ t = 0} E\lp u_{tw}, \Sigma  - \bigcup_{i = 1}^{\absv{p}} D_{tw}(\delta)  \rp}_{ : = A(\delta)} + \underbrace{\eval{ \frac{d}{dt}}_{ t = 0} E\lp u_{tw}, \bigcup_{i = 1}^{\absv{p}} D_{tw}(\delta) \rp}_{ : = B(\delta)}.
\end{equation*}
As in Lemma \ref{cokhess} $B(\delta) \to 0$ as $\delta \to 0$.
Thus we arrive at
 \begin{equation*}
		\eval{ \frac{d}{dt}}_{ t = 0} E\lp u_{tw}, \Sigma \rp  = \lim_{\delta \to 0} A(\delta).
\end{equation*}
By the chain rule,
	\begin{align} \label{A}
		A(\delta) = \underbrace{\eval{ \frac{d}{dt}}_{ t = 0} E\lp u_{tw}, \Sigma  - \bigcup_{\pp} D_{0}(\delta)  \rp}_{:= A_1(\delta)} + \underbrace{\eval{ \frac{d}{dt}}_{ t = 0} E\lp u_{0}, \Sigma  - \bigcup_{\pp} D_{tw}(\delta)  \rp}_{:= A_2(\delta)}
	\end{align}
As we will see shortly, the term $A_2(\delta)$ is not in general
bounded as $\delta \to 0$, but we will show that $A_1(\delta)$
decomposes into a sum of two terms, $A_1(\delta) = A_1^1(\delta) +
A_1^2(\delta)$ where $A_1^1(\delta) \sim - A_2(\delta)$ (i.e.\ it
cancels the singularity), and $A_1^2(\delta)$ converges to the
expression in (\ref{frstvar}). $A_1$ is an integral over a smooth
manifold with boundary so by the first variation formula
\eqref{frstvar0}, and in the last line using decomposition (\ref{decomp2}), if $\Sigma(\delta)$ is as in (\ref{sigmaeps}), we have
			\begin{align*}
			A_{1} &=  \int_{\p \Sigma (\delta)} \la J_w , \p_\nu \ra_{\Gfixed} ds  \\
				 &= - \underbrace{\sum_{p_i \in \pl}  \int_0^{2\pi} g\lp J_w , \p_r \rp \delta d\theta}_{ : = A_1^1}  - \underbrace{\sum_{p_i \in \pl}  \int_0^{2\pi} \Re \Phi\lp J_w , \p_r \rp \delta d\theta}_{A_1^2}.
			\end{align*}
It thus remains only to show: 1)  $\lim_{\delta \to 0} A_1^2= \Re \lp
2 \pi i \Res \lp \iota_{J_w} \Phi \rp \rp$, 2) (\ref{frstvar2}) holds,
and 3)  $\lim_{\delta \to 0} \absv{ A_1^1 + A_2 } = 0$.

We prove numbers 1 and 2 together:
 \begin{equation*}
		 \int_{\p D_0(\delta)} \Re \Phi\lp J_w , \p_r \rp \delta d\theta =  \int_{\p D_0(\delta)} \Re \phi(z) dz^2\lp J_w , \p_r \rp \delta d\theta.
\end{equation*}
Since
	\begin{align*}
		dz^2\lp J_w , \p_r \rp =   J_w dz\lp  \p_r \rp =  J_w \frac{z}{\absv{z}},
	\end{align*}
we have
 \begin{equation} \label{eeee}
		\int_0^{2\pi} \Re \Phi\lp J_w , \p_r \rp \delta d\theta = \Re \int_0^{2\pi} J_w \phi(z) dz,
 \end{equation}
and by definition
 \begin{equation} \label{eee}
		\int_0^{2\pi} \Re \Phi\lp J_w , \p_r \rp \delta d\theta = \Re \int_0^{2\pi} \iota_{ J_w}\Phi(z).
 \end{equation}
So 
	\begin{align} \nonumber
		\sum_{p_i \in \pl}\lim_{\delta \to 0}  \int_0^{2\pi} J_w \phi(z) dz &=  \sum_{p_i \in \pl}\lim_{\delta \to 0}  \int_0^{2\pi} \lp-w_i + \mc{O}(\absv{z}) \rp \phi(z) dz \\ \label{eeeee}
		&= - 2 \pi i \sum_{p_i \in \pl} w_i  \Res \rvert_{p_i} \phi_{u_0}.
	\end{align}
Putting (\ref{eeeee}) together with (\ref{eee}) gives us what we wanted.
	
For number 3, note that from the expression $g = e^{2 \mu} \absv{z}^{2(\al - 1)}$ we have
	\begin{align*}
		A_1^1(\delta) &= \int_0^{2\pi} \la -w_i + \mc{O}(\absv{z}), \p_r \ra_{g} r d\theta 
		 = \int_0^{2\pi} \la -w_i , \p_r \ra_{g} r d\theta  + \mc{O}(\delta^{2\al})
	\end{align*}
Using \eqref{jesus} and that fact that near $p \in \pp$, $T_{tw}^{-1} (z) = z + tw$ parametrizes the boundary of $\p D_{tw}(\delta)$ we have 
	\begin{align*}
          A_{2}  &= \sum_{i = 1}^{\absv{ p }}   \int_0^{2\pi} \la
          \dot{T}_0^{-1} , \p_r \ra_g r d\theta
          = \sum_{i = 1}^{\absv{ p }}   \int_0^{2\pi} \la w_i, \p_r \ra_g ds + \mc{O}(\delta^{2\al}).
	\end{align*}
Thus $\absv{ A_1^1 + A_2 } = \mc{O}(\delta^{2 \al})$, and the proof is
finished.
	\end{proof}

Now suppose that $u_{0}$ solves \eqref{hme2}.  Thus $u_{0}$ is an absolute minimum of
$E(\cdot, g, \Gfixed)$.  By differentiating again, we get the following
as a corollary to Proposition \ref{scndvarformula}.
	\begin{corollary} \label{secondvar}
Suppose that $(u_{0}, g, \Gfixed)$ solves \eqref{hme2} and satisfies
Assumption \ref{assump}.  Then
			\begin{equation} \begin{split} \label{scndvar}
				\eval{ \frac{d^2}{dt^2}}_{ t = 0}
                                \wt{E}(u_{tw}) =  \Re \lp 2 \pi i
                                \sum_{p_i \in \pl} \Res \rvert_{p_i}
                                \lp \iota_{J_w} \Phi(J_w) \rp \rp =
                                \Re \lp 2 \pi i \sum_{p_i \in \pl} w_i
                                \Res \rvert_{p_i}   \phi(J_w) \rp
			\end{split} \end{equation}
where $\Phi(J_w) dz^2 = \eval{\frac{d}{dt}}_{t = 0} \phi_{u_{tw}} dz^2$ 
	\end{corollary}

We now conclude the proof Proposition \ref{h2prop} using Proposition
\ref{scndvarformula}

\begin{proof}[Proof of Proposition \ref{h2prop}]
We proceed by interpreting the formula for the Hessian in Corollary
\ref{secondvar} in light of the characterization of the Hessian of the
energy functional at a solution to (\ref{hme2}) in Corollary
\ref{positiveindex}.  By Corollary \ref{positiveindex}, the Hessian of
the energy functional is positive definite on any compliment of the
conformal Killing fields, while by Corollary \ref{hessharm}, if $\Res
J_w = 0$, $J_w$ is a zero of the Hessian.  If genus $\Sigma > 1$,
there are no conformal Killing fields, so the map (\ref{h2map}) is an
injective linear map between vector spaces of the same dimension, thus
an isomorphism.  If genus $\Sigma \leq 1$, any space $\mc{H}arm_{\qq}'
\subset \mc{H}arm_{\qq}$  with tangent space complimentary to
$\mc{C}onf_0$ will satisfy (\ref{smallinj}).

\end{proof}


\section[Closedness]{Closedness:  limits of harmonic diffeomorphisms}
\label{closedsec}

In this section we prove that $\mc{H}(\qq)$ (see
Proposition \ref{thm:clopenprop}) is closed.  In fact we prove the
stronger result

\begin{theorem} \label{closedthm}
	Let $(u_k, g_{k}, G_k)$ be a sequence of solutions to
        (\ref{hme2}) where $g_{k} \to g_{0}$ with $g_k \in \Mm^*_{2, \g,
          \nu}(g_{0}, \pp, \Aa)$, $G_k \to \horrible$ with $G_k \in \Mm^*_{2, \g,
          \nu}(\horrible, \pp, \Aa)$, and all the $u_{k}$ are in Form \ref{uform},.  Assume that $\kappa_{G_k} \leq 0$ and that each $u_k$ has non-vanishing Jacobian away from $u_k^{-1}(\pp)$.  Then the $u_k$ converge to a map $u_0$ so that $(u_0, g_0, \horrible)$ solves (\ref{hme2}) in Form \ref{uform}.  The convergence is $C^{2, \g}_{loc}$ away from $u_0^{-1}(\pp)$. For a precise description of the convergence near the cone points, see Corollary \ref{wkconverge} below.
\end{theorem}

To be precise about the convergence of the $G_k$, in conformal coordinates near $p \in \pp$ we have
 \begin{equation*}
		G_k = c_k e^{2 \mu_k} \absv{w}^{2 (\al - 1)} \absv{dw}^2,
\end{equation*}
and
 \begin{equation*}
		\horrible = c_0 e^{2 \mu_0} \absv{w}^{2 (\al - 1)} \absv{dw}^2.
\end{equation*}
By $G_k \to \horrible$, we mean that for some $\sigma > 0$,
	\begin{equation} \begin{split} \label{eq:metnearp}
		\mu_k &\to \mu \mbox{ in } r^{\nu} C^{2, \g}_b(D(\sigma))   \\
		c_k &\to c_0
	\end{split}\end{equation}
(The $b$-H\"older spaces are defined in
(\ref{c2gb0})--(\ref{c2gb2})).  Away from the cone points $G_k \to G$
in $C^{2, \g}_{loc}$.  We can (and do) reduce to the case $c_k = c_0 =
1$ by replacing $G_k$ by $G_k / c_k$ and $\horrible$ by $\horrible / c_0$.  Note
that \eqref{eq:metnearp} easily implies that near each cone point the scalar curvature
functions $\kappa_{G_{k}}$ converge to $\kappa_{\horrible}$ in
$C^{0, \g}_{b}$ (see \eqref{eq:scalar}); in particular are they uniformly
bounded in absolute value.

For the $g_{k}$, we make the stronger assumption that near
$u^{-1}(\pp)$ the metrics look like the standard round conic metric
$g_{\alpha}$ (see \eqref{roundmetric}).  To do this uniformly, we need the uniform bound on the
modulus of continuity obtained in the next section; the precise
statement of this assumption is in section \ref{edbouc}.  In the end the theorem is
true as stated (i.e.\ without this stronger assumption), since we will
change the domain metric in a bounded way and in its conformal class.

We refer the reader to the introduction for an outline of the
subsequent arguments.  Before we prove Theorem \ref{closedthm}, we use
it to prove
\begin{proposition}
  $\mc{H}(\qq)$ is closed.
\end{proposition}
\begin{proof}
  Let $t_{k} \in \mc{H}(\qq)$ be a sequence such that $t_{k} \to
  t_{0}$, and let $\cc_{k}$ be the corresponding conformal structures,
  so $\cc_{k} \to \cc_{0}$.  We again uniformize locally, i.e.\ we
  choose diffeomorphisms $v_{k}$ so that $id: (\Sigma, \cc_{0}) \lra
  (\Sigma, v_{k}^{*}\cc_{k})$ is conformal near $\pp$ in such a way
  that $v_{k} \to id$ in $C^{\infty}$.  By assumption,
  there is a rel$.$ $\qq$ minimizer $u_{k} : (\Sigma, \cc_{k}) \lra
  (\Sigma, G)$. Let $g_{k}$ be metrics in $\cc_{k}$ and $g_{0}$ be a
  metric in $\cc_{0}$ that are conic
  near near $u_{k}^{-1}(\pp)$ with cone angles $\Aa$.  Then
  $u_{k}\circ v_{k} : (\Sigma, v_{k}^{*}g_{k}) \lra (\Sigma, G)$ are
  also rel$.$ $\qq$ minimizers, but now $v_{k}^{*}g_{k} \in \Mm^*_{2, \g,
          \nu}(g_{0}, \pp, \Aa)$, and thus Theorem \ref{closedthm}
        applies.  The limiting map $u_{0}$ is the minimizer we desire,
        to $\mc{H}(\qq)$ so closed.
\end{proof}

\subsection{Energy bounds and uniform continuity} \label{ebauc}

The $u_k$ in Theorem \ref{closedthm} are in fact a rel$.$ $\qq$ energy minimizing sequence for the metric $\horrible$, meaning that 
	\begin{align} \label{minseq}
		\limsup_{k \to \infty} E(u_k, g_{k}, \horrible) = \inf_{u \sim_{rel. \qq} id} E(u, g_{k}, \horrible).
	\end{align}
This follows from Proposition \ref{umin}, since for any $u \sim_{rel. \qq} id$
	\begin{align*}
		\limsup_{k \to \infty} E(u_k, g, \horrible) &= \limsup_{k \to \infty} E(u_k, g_{k}, G_k) \\
		&\leq \lim_k E(u, g_{k}, G_k) \\
		&= E(u, g_{0}, \horrible).
	\end{align*}
Assume that genus $\Sigma > 0$.  It is standard that the $u_{k}$ are an equicontinuous
sequence.  Thus they subconverge.   Let
 \begin{equation*}
		R := u_0^{-1}(\Sigma_\pp).
\end{equation*}
On this regular set, it again standard that the $u_{k}$ converge
uniformly in $C^{2, \gamma}_{loc}$ on compact sets of $R$. \cite{t}, \cite{j}.

For $\Sigma = S^{2}$, lift as in \eqref{eq:lift} to a branched cover
and applying the above arguments gives the same results. To summarize, we have
	\begin{lemma} \label{easyreg}
          The $u_k$ converge in $C^0(\Sigma) \cap C^{2,
                  \g}_{loc}(R)$ to a map $u_{0}$.  On each compact set of $R$ the
                Jacobian of $u_0$ is bounded below by a positive constant.
	\end{lemma}

\subsection{Upper bound for the energy density and lower bound on the Jacobian.} \label{edbouc}

Fix $p \in \pp$, and let
 \begin{equation*}
		q_k = u_k^{-1}(p),
\end{equation*}
we can pass to a subsequence so that
 \begin{equation*}
		q_k \to q_0,
\end{equation*}
for some $q_0$.  Let $S \subset \Sigma$ be any set containing $u_0^{-1}(p)$ so that $S \cap u_0^{-1}(\pp - \set{p}) = 0$ and $S$ is diffeomorphic to a disc, and choose conformal maps
	\begin{align*}
		F_k : D &\lra S\\
		 0 &\longmapsto q_k
	\end{align*}
so that $F_k \to F_0$ in $C^\infty$.  Finally, define
 \begin{equation*}
		w_{k} := u_k \circ F_k.
\end{equation*}
Thus we have a sequence of harmonic maps $w_k :  D \lra (\Sigma, G_k)$
with $w_k(0) = p$.  By uniform continuity, we may choose a single
conformal coordinate chart containing $w_k(D) = u_k \circ F_{k}(D)$.
By abuse of notation, we denote these coordinates by $w$.  Our goal is
to prove uniform estimates for the $w_k$.  Specifically, we wish to
control their energy densities and Jacobians.  The remarkable
fact here is that, although being harmonic is conformally invariant, if the domain is given a conic metric with cone
points at the inverse images of the cone points in the target, then uniform control, from above and below, in terms of
the energy is available exactly when the metric on the domain has cone
angles equal to that of the corresponding cone points in the target
metric.  Consider the maps
 \begin{equation} \label{wk}
		w_k :  (D, g_\al) \lra (\Sigma, G_k),
 \end{equation}
where $2 \pi \al$ is the cone angle at $p$ and again $g_\al =
\absv{z}^{2 ( \al - 1)} \absv{dz}^2$.  

\begin{proposition} \label{loclip}
	The maps (\ref{wk}) have uniformly bounded energy density (see (\ref{confenergy})), i.e.\ for some $C > 0$
 \begin{equation} \label{edb}
		e_k(z) = e(w_k, g_\al, G_k)(z) < C \mbox{ for } \absv{z} \leq 1/2.
 \end{equation}
At $z = 0$ we have the lower bound
	\begin{align} \label{harnbound}
		0 < c \leq \lim_{z \to 0} e_k(z, g_\al, G_k)
	\end{align}
for some uniform $c$.
\end{proposition}

Before we begin the proof, we remark that by Form \ref{uform},
 \begin{equation} \label{wuform}\begin{split}
		w_k(z) &= \lambda_k z + v_k(z) \\
		v_k &\in r^{1 + \e}C^{2, \g}_b.
\end{split}\end{equation}
As we will show below
 \begin{equation}\label{eq:limitlambda}
		\lambda_k^{2\al} = \lim_{z \to 0} e_k(z, g_\al, G_k),
\end{equation}
so showing (\ref{harnbound}) reduces to showing that $0 < c \leq \lambda_k$
for some uniform $c$.

\begin{proof}[Proof of Proposition \ref{loclip}]
For any harmonic map $w : (D, \sigma \absv{dz}^2) \lra (D, \rho \absv{dw}^2)$, as in \cite{w}, let
 \begin{equation} \label{hell}\begin{split}
		h(z) = \frac{\rho(w(z))}{\sigma(z)} \absv{\p_z w}^2,
                \qquad  \qquad
		\ell(z) = \frac{\rho(w(z))}{\sigma(z)} \absv{\p_{\ov{z}} w}^2,
\end{split}\end{equation}
so the energy density and the Jacobian satisfy, respectively
 \begin{equation} \label{energyhl}\begin{split}
		e = h + \ell \qquad \& \qquad
		J = h - \ell.
 \end{split}\end{equation}
The lemma follows from analysis of the following inequality and identity, which are standard and can be found for example in \cite{sy}, where they appear as equations (1.19) and (1.17), respectively.
\begin{equation} \label{eq:eandhweitz}
	\begin{split}
		\Delta e(u) &\geq - 2 \kappa_\rho J  + 2 \kappa_\sigma e(u) \\ 
		\Delta \log h &= - 2 \kappa_\rho \lp h - \ell \rp + \kappa_\sigma.
	\end{split}
      \end{equation}

Here $\kappa_\rho$ and $\kappa_\sigma$ are the scalar curvature
functions for the range and domain, respectively, and $\Delta$ is the
Laplacian for $\sigma \absv{dz}^2$.  The second equation holds only
when $h(z) \neq 0$, and of course both equations make sense only when
$\sigma$ and $\rho$ are sufficiently regular.  For the $w_k$ in
(\ref{wk}), the equations simplify as follows: 1) $\sigma(z)
\absv{dz}^2 = g_\al$ has $\kappa_\sigma \equiv 0$  away from $ z = 0
$, 2) $\kappa_{\rho_k} \leq 0, J \geq 0$ by the assumptions of Theorem
(\ref{closedthm}), and 3) $\kappa_{\rho_k} > - C$ by the remark about the
convergence of scalar curvature at the beginning of the section.  Therefore, we restrict our attention to the inequalities
	\begin{align} \label{subsoln}
		\Delta e_k &\geq 0 \\ \label{supsoln}
		\Delta \log h_k &\leq  C h_k.
	\end{align}

To prove (\ref{edb}), we use (\ref{subsoln}) as follows.  Since
$\Delta = \absv{z}^{2(1 - \al)} \Delta_0$ where $\Delta_0 = 4 \p_z \p_{\ov{z}}$ is the euclidean Laplacian, we also have 
 \begin{equation}\label{eq:delta0}
		\Delta_0 e_k \geq 0,
\end{equation}
away from $z = 0$.  We claim that in fact each $e_k$ is a subsolution
to \eqref{eq:delta0} on all of $D$.
To see that this is true, write $e_k = h_k + \ell_k$ as in (\ref{energyhl}).  Using (\ref{wuform}), we have
	\begin{align*}
		h_k(z) &= \frac{{\rho_k}(w_k(z))}{\absv{z}^{2(\al - 1)}} \absv{ \p_z w_k(z) }^2 \\
		&= \lambda_k^{2(\al - 1)} \absv{1 + \absv{z}^{2(1 - \al)}v_k(z)}^{2(\al - 1)} e^{2{\mu_k}(w_k(z))} \absv{  \lambda_k +  \p_z v_k(z) }^2 \\
		&= \lambda_k^{2\al} \absv{1 + \absv{z}^{2(1 - \al)}v_k(z)}^{2(\al - 1)} e^{2{\mu_k}(w_k(z))} \absv{ 1 +  \p_z v_k(z) /  \lambda_k  }^2,
	\end{align*}
so since $v_k \in r^{1 + \e} C^{2, \g}_b$ and $\mu_k \in r^\nu C^{2, \g}_b$, we have
 \begin{equation} \label{hreg}
		h_k - \lambda_k^{2\al}  \in r^\e C^{1,\g}_b 
 \end{equation}
for some $\e > 0$.  Similarly, using $\p_{\ov{z}} z = 0$,
 \begin{equation} \label{lreg}
		\ell_k \in r^\e C^{1,\g}_b.
 \end{equation}
Therefore
 \begin{equation} \label{wkedensity}
		e_k = \lambda_k^{2\al} + f_k(z),
 \end{equation}
where $f_k \in r^{\e} C^{1, \g}_b(D)$.  In particular, $r \p_r e_k \to 0$ as $r \to 0$.  Therefore, for any non-negative function $\zeta \in C_c^\infty(D)$.
	\begin{equation}\label{eq:esubsoln}
          \begin{split}
            \int_D \nabla e_k \cdot \nabla \zeta dxdy
            &= -  \lim_{\e \to 0} \int_{D - D(\e)} ( \Delta e_k ) \zeta dxdy  - \int_{r = \e} (\p_r e_k )\zeta r d\theta \\
            &\leq 0,
          \end{split}
	\end{equation}
so the $e_{k}$ are indeed subsolutions.  By (\ref{wkedensity}), each $e_k$ is a bounded function, so the standard theory of subsolutions to elliptic linear equations (see section 5 of \cite{mo}) now implies that for some uniform constant $C > 0$,
 \begin{equation*}
		\sup_{z \in D(1/2)} e_k(z) \leq C \int_{D} e_k dxdy.
\end{equation*}
The right hand side is controlled by the energy,
	\begin{align*}
		C \int_{D} e_k dxdy &= C \int_{D} e_k \frac{\absv{z}^{2(\al - 1)}}{\absv{z}^{2(\al - 1)}}dxdy \\
		&\leq C \int_{D} e_k \absv{z}^{2(\al - 1)}dxdy \\ 
		&= C E(w_k, D, G_{k}) \\
		&\leq C',
	\end{align*}
and this establishes (\ref{edb}).  

It remains to prove \eqref{harnbound}.  From (\ref{hreg}) and (\ref{lreg}), we now see that
 \begin{equation} \label{dnmbound}
		\lim_{z \to 0} e_k(z) = \lim_{z \to 0} h_k(z) = \lambda_k^{2\al},
 \end{equation}
Thus \eqref{eq:limitlambda} is correct, and to prove (\ref{harnbound}), it is equivalent to prove that $c < \lambda_k$ for some uniform $c$.

To do so, we use (\ref{supsoln}).  Dropping the $k$'s for the moment, by (\ref{hreg}) and the fact that the logarithm is smooth and vanishes simply at $1$, we have
 \begin{equation} \label{loghreg}
		\log h - \log \absv{\lambda}^{2\al} \in r^{\e} C^{2, \g}_b.
 \end{equation}
We will now apply the assumption from Theorem \ref{closedthm} regarding the Jacobian, specifically that $J = h - \ell > 0$ on compact sets away from $0$, and thus by continuity and (\ref{dnmbound}) we may choose a $\delta$ satisfying 
	\begin{align*}
		0 < \delta \leq \frac{1}{2} \inf_{z \in D} h(z).
	\end{align*}
Thus $h / \delta$ is bounded from below by $2$.  We also need control
from above.  We already know by \eqref{edb} that $ h + \ell = e < c$
for some $c > 0$, so $\sup_{z \in D} h(z) < c$ for some constant
depending only on the energy.  Using this and the obvious bound $h
\leq  \frac{c}{\log (c' / \delta) } \log \lp \frac{h}{\delta} \rp$ we
conclude from (\ref{supsoln}) that
 \begin{equation} \label{supsoln2}\begin{split}
		\Delta \log \lp h / \delta \rp
		&\leq \frac{c}{  \log (c' / \delta)}  \log \lp \frac{h}{\delta} \rp.
\end{split}\end{equation}
The upshot is that $\log h / \delta \geq \ln 2 > 0$, so this
inequality looks promising for an application of the Harnack
inequality.  We use the following explicit inequality, inspired by
Lemma 6 of \cite{he}.  
\begin{remark}
  The lemma we are about to prove has slightly more general
  assumptions than seem necessary, but they are necessary for the
  analogous statement in section \ref{coneanglepisect} where we deal
  with the case $\pe \neq \varnothing$.
\end{remark}
	\begin{lemma}[Harnack Inequality] \label{genharn}
		Let $\Delta$ denote the Laplacian on the standard cone $(D, g_\al)$.  Let $f : D \lra \R$, $f \in C^2(\ov{D} - \set{0})$, $f > 0$, and assume that for some $\sigma > 0$, 
			\begin{align*}
			\lp \Delta - \sigma \rp f &\leq 0 \mbox{ on } D - \set{0}.
			\end{align*}
Assume furthermore that
 \begin{equation} \label{hypos}\begin{split}
				f  &= a + b(\theta) + v(r, \theta) \\
				a &\in \C \\
				b &\in C^{\infty}(S^1) \\
				v &\in r^{\e}C^{2, \g}_b(D)
\end{split}\end{equation}
Then if $\sigma < \e^2$,
 \begin{equation} \label{basicharn}
				\liminf_{z \to 0} f = a + \inf b \geq e^{- c\sigma}\frac{1}{2\pi}\int_0^{2\pi} f(e^{i\theta}) d\theta
 \end{equation}
for some $c > 0$.
	\end{lemma}
Before proving the lemma, we conclude the proof of (\ref{harnbound})
(and thus of Proposition \ref{loclip}) by applying the lemma to
(\ref{supsoln2}) as follows.  By (\ref{loghreg}), the lemma applies to
\eqref{supsoln2} with $f = \log h/\delta$ and $b \equiv 0$.  We will choose $\eta > 0$ small so that if 
	\begin{align} \label{deltak}
		\delta_k = \min \set{ \eta, \frac{1}{2} \inf_{z \in D - 0} h_k(z) }
	\end{align}
then the hypotheses of the lemma are satisfied.  

Thus by the lemma
	\begin{align*}
		\lim_{z \to 0}\log (h_k(z) / \delta_k) &\geq e^{ \frac{c}{\log(c' / \delta_k)}} \frac{1}{2\pi} \int_0^{2\pi} \log(h_k / \delta_k) d\theta \\
		\lim_{z \to 0} \log h_k(z) &\geq e^{ \frac{c}{\log(c' / \delta_k)}} \lp \int_0^{2\pi} \log{h_k} \  d\theta \rp - \lp  e^{ \frac{c}{\log(c' / \delta_k)}}  - 1 \rp \log \delta_k.
	\end{align*}
It is trivial to check that $\lp  e^{ \frac{c}{\log(c' / \delta_k)}}
- 1 \rp \log \delta_k > c''$ for some $c''$ independent of $\delta_k$.  Thus, since $e^{ \frac{c}{\log(c' / \delta_k)}} $ is bounded above, for some $c$, we have
 \begin{equation} \label{hbound}
		\lim_{z \to 0} \log h_k(z) = \absv{\lambda_{k}}^{2 \alpha} \geq -c  \absv{ \int_0^{2\pi} \log{h_k}(e^{i \theta}) \  d\theta } + c
 \end{equation}
By Lemma \ref{easyreg}, the Jacobians, $J_k$, and thus the $h_k$, are uniformly bounded below on $\p D$.  Thus (\ref{hbound}) gives a uniform lower bound for $h_k(0)$ from below, which finishes the proof of Proposition \ref{loclip} modulo the proof of Lemma \ref{genharn}.

\end{proof}

	\begin{proof}[Proof of Lemma \ref{genharn}]
		To prove the lemma we will work in normal polar coordinates.  Let $\phi = \theta, \rho = r^{\al}/\al.$  In these coordinates $g_\al = d \rho^2 + \al^2 \rho^2 d\phi^2$ and
$\Delta = \p_\rho^2 + \frac{1}{\rho}\p_\rho + \frac{1}{\al^2 \rho^2} \p_\phi^2$.
The fact that $f - \lp a + b(\theta) \rp \in r^{\e}C^{2, \g}_b$ implies that $f - \lp a + b(\theta) \rp \in \rho^{\e/\al}C^{2, \g}_b$.  The exact vanishing rate is irrelevant, and we refer to it henceforth as $\e$.
		
		The proof proceeds by comparing $f$ to the solution $\wt{f}$ to the equation
 \begin{equation} \label{ftilde}\begin{split}
				(\Delta - \sigma) \wt{f} &= 0 \\
				\wt{f} \rvert_{\p D} &= f \rvert_{\p D} 
\end{split}\end{equation}
That such a solution exists follows, for example, from separation of
variables from \cite{ta}, in particular there is a solution that is
continuous on $\ov{D}$ with the asymptotic behavior 
 \begin{equation} \label{solnasym}\begin{split}
		\wt{f}(z) - \wt{f}(0) &\in r^{a}C^{2, \g}_b \\
		a &= \min \set{ 2 \al, 1} .
\end{split}\end{equation}
It follows that $F = f - \wt{f}$ satisfies 
 \begin{equation} \label{subsoln2}
   \begin{split}
		\lp \Delta - \sigma \rp F &\leq 0 \mbox{ away from } z
                = 0 \\
	F \rvert_{\p D} &= 0, \mbox{ where } \p D = \set{ \rho = 1/\al }.
      \end{split}
	\end{equation}

The main technical challenge is to show that $F(\rho, \phi) \geq 0$. Assuming this for the moment, we have in particular that
 \begin{equation} \label{zerobound}
		 a + b(\theta) \geq \wt{f}(0).
 \end{equation}
Let
	\begin{align*}
		\gamma(\rho) := \frac{1}{2\pi} \int_0^{2\pi} \wt{f}(\rho, \phi) d\phi.
	\end{align*}
Then $\lp \p_\rho^2 + \frac{1}{\rho} \p_\rho - \sigma \rp \gamma = 0$,
so since $\gamma(0) = \wt{f}(0)$, 
 \begin{equation*}
		\gamma(\rho) = \wt{f}(0)\sum \frac{\sigma^m \rho^{2m}}{2^{2m}(m!)^2 }.
\end{equation*}
But then by (\ref{zerobound}) and $F\rvert_{\p D} = 0$, we have
	\begin{align*}
		\frac{1}{2\pi} \int_0^{2\pi} f(1/\al, \phi) d\phi &= 		\int_0^{2\pi} \wt{f}(1/\al, \phi) d\phi \\
		&=\wt{f}(0)\sum \frac{\sigma^m (1/\al)^{2m} }{2^{2m}(m!)^2 } \\
		&\leq e^{\sigma/ 4 \al^2} \lp a + b(\theta) \rp,
	\end{align*}
which is equivalent to (\ref{basicharn}).

It remains to show that $F \geq 0$.  Switch back to conformal
coordinates, $z = r e^{i \theta}$.  By assumption (\ref{hypos}), $F(z)
= F(r, \theta)$ is continuous on $[0, 1]_r \times S^1_\theta$, thus
attains a minimum.  If that minimum is on $r = 1$, we are done by
(\ref{subsoln2}).  If it is in the interior and away from $z = 0$, say
at $z_0$, then (\ref{subsoln2}) implies that $0 \leq \Delta F \leq \sigma F$,
so since $\sigma \geq 0$, $F(z_0) \geq 0$.  

Finally, assume that $F$ attains its minimum on $r = 0$, and suppose that
\begin{equation*}
I := \inf_{r = 0} F(r, \theta) < 0
\end{equation*}
We will perturb $F$ to a function $F_\mu$ solving $\lp \Delta - \sigma
\rp F_\mu \leq 0$ with a negative interior minimum. This is a
contradiction by the previous paragraph, so $I \geq 0$ and thus $F
\geq 0$.

Consider the function
 \begin{equation} \label{bbig}\begin{split}
		F_\mu(z) = F(z) - \mu \absv{z}^\nu,
\end{split}\end{equation}
where $\mu > 0,  \nu \geq \sqrt{\sigma}$.  Since $(\Delta - \sigma) \absv{z}^\nu = \absv{z}^\nu \lp \absv{z}^{-
  2\al} \nu^2 - \sigma \rp \geq 0$, $F_\mu$ also satisfies $\lp \Delta - \sigma \rp F_\mu(z) \leq 0$.
Since $F_\mu \rvert_{\p D} \equiv- \mu$ and $\liminf_{z \to 0}
F_\mu(z) = I$, if $I \leq -\mu$
then $F_\mu = F_\mu(r, \theta)$ has a (negative) minimum away from $\p D$.  On the other hand, given any $(r, \theta)$ with $r > 0$,
	\begin{equation*}
		F_{\mu}(0, \theta) \geq F_\mu(r, \theta) \iff
		\frac{F(0, \theta) - F(r, \theta)}{{r^\nu}} \geq - \mu.
\end{equation*}
Thus \emph{if} we can find $\mu > 0$ so that for some $r > 0, \theta, \nu$
	\begin{align} \label{donecondition}
		I \leq - \mu \leq \frac{F(0, \theta) - F(r, \theta)}{{r^\nu}},
	\end{align}
then $F_\mu$ will have a negative interior minimum.  By the
assumptions of the lemma and (\ref{solnasym}) we have that $F(0,
\theta) - F(r, \theta) \in r^{\e}C^{2, \g}_b$, and in particular $\frac{\absv{F(0, \theta) - F(r, \theta)}}{{r^\e}} < C$.
Thus, as long as $\nu  < \epsilon$,
we have that
	\begin{align*}
		\frac{\absv{F(0, \theta) - F(r, \theta)}}{r^\nu} \to 0 \mbox{ as } z \to 0.
	\end{align*}
Thus, taking the second line in (\ref{bbig}) into account,
\eqref{donecondition} can be attained if
	\begin{align} \label{sigmacondition}
		\sqrt{\sigma} < \e.
	\end{align}
This finishes the proof.
	\end{proof}

\subsection{Harmonic maps of $C_\al$}

Now we prove a classification lemma for harmonic maps of the standard cone $C_\al$ defined in section \ref{examplesection}.
\begin{lemma} \label{globalclass}
Let $u: C_\al \lra C_\al$ be a smooth harmonic map (i.e.\ a solution to (\ref{hme2})) in Form \ref{uform},  with uniformly bounded energy density, that is the uniform limit of a sequence of orientation-preserving homeomorphisms fixing the cone point.  Then $u$ is a dilation composed with a rotation.
\end{lemma}
\begin{proof}
Let $e$ be the energy density function of $u$, and let $h, \ell$, and $J$ be the functions defined in (\ref{hell}).  We will use a differential inequality for $\ell$, similar to those used in the proof of Proposition \ref{loclip}.  Namely, (1.18) from \cite{sy} gives
 \begin{equation*}
		\Delta \ell \geq 2 \kappa_\rho J \ell + 2 \kappa_\sigma \ell,
\end{equation*}
where $\kappa_\rho$ and $\kappa_\sigma$ are the scalar curvatures on
the domain and target, respectively, and the inequality holds only
away from $z = 0$.  Since $C_\al$ is flat away from the cone point,
this gives $\Delta \ell(z) \geq 0 \mbox{ when } z \neq 0$.
The assumption that the energy density of $u$ is uniformly bounded
implies $\ell$ is uniformly bounded ($\ell < e$), and on the other hand, as in
\eqref{eq:esubsoln}, $\ell$ is a subsolution on all of $\C$.  Thus
$\ell$ is identically constant.  (There are no non-constant, bounded
entire subsolutions.)  Since $\ell \in r^{\epsilon}C^{2, \gamma}_{b}$,
$\lim_{z \to 0} \ell = 0$, and thus $\ell \equiv 0$, i.e.\ $\p_{\ov z}
u = 0 \mbox{ when } z \neq 0$ Since $u$ bounded near $0$, it is in
fact entire.  The fact that $u$
is the uniform limit of homeomorphisms and takes Form \ref{uform}
implies that it is $1-1$ on the inverse image of an open ball around
$0$.  Since the only entire holomorphic function with this behavior
that fixes $0$ is
$u(z) = az$ for $a \in \C^*$, the proof is complete.
\end{proof}


\subsection{Uniform $C^{2, \gamma}_{b}$ bounds near $\pp$ and the
  preservation of Form \ref{uform}} \label{uniformbounds}

We will now make precise the sense in which the $w_k$ in (\ref{wk})
are uniformly bounded near $\pp$.  Let $D(R) = \set{ z \in C : \absv{z} < R}$.
We will need the well-known Rellich lemma for $b$-H\"older spaces; given $0 < \g' < \g < 1$, $c' < c$, and non-negative integers $k' \leq k$ the containment
	\begin{align} \label{rellich}
 		r^{c}C^{k, \g}_b(D(R)) \subset 		r^{c'}C^{k', \g'}_b(D(R))
	\end{align}
is compact.
Given a smooth function $f: \C \lra \R$, define
	\begin{equation} \label{funnorm}
		\| f \|_{c, k, \g, R} := \| f \|_{r^c C^{k,\g}_b(D(R))},
	\end{equation}
And for any map $w: D(R) \lra D(R)$ with
 \begin{equation} \label{w0}
		w(z)  =  \lambda z + v(z) 
 \end{equation}
and $v \in r^{c}C^{2, \g}_b(D(R))$, let
 	\begin{align} \label{norm}
				 \displaystyle{
								\fbox{
				$ \braces[c, k, \g, R]{w}  :=  \norm[ c, k, \g, R]{v}
				$}
				}
	\end{align}
Having established the uniform energy density bound, we now set out to prove
	\begin{proposition} \label{closed}
		For the $w_k$ in (\ref{wk}), there exist uniform constants $\sigma, C, \e > 0$ such that 
			\begin{align} \label{bound2}
				\braces[1 + \e, 2, \g, \sigma]{ w_k }< C.
			\end{align}
	\end{proposition}
From this, Proposition \ref{loclip}, and the fact that the containment (\ref{rellich}) is compact, we immediately deduce
	\begin{corollary} \label{wkconverge}
		The $ w_k = \lambda_k z + v_k$ converge to a map $w_0 = \lambda_0 z + v_0$, in the sense that
			\begin{align*}
				\lambda_k &\to \lambda_0 \\
				v_k &\to v_0 \mbox{ in } r^{1 + \e}C^{2, \g}_b,
			\end{align*}
for some $\e, \g > 0$.  Since the $F_k$ in $u_k \circ F_k = w_k $ converge in $C^\infty$ to some univalent conformal map $F_0$, the limit in $u_k \to u_0$ is of Form \ref{uform}.
	\end{corollary}

%

%


%

Before we prove Proposition \ref{closed}, we discuss scaling properties of the norm in (\ref{funnorm}).  Let $R > 0$, $\sigma > 0$, and let $f$ be a function defined on $D(R)$.  Writing  
	\begin{align*}
		f_\sigma(z) = f(\sigma z),
	\end{align*}
it is trivial to verify from (\ref{c2gb0})--(\ref{c2gb2}) that
 \begin{equation}  \label{scaling} 
		\| f_{\sigma} \|_{c, k, \g, R/\sigma} = \sigma^c \| f \|_{c, k, \g, R}
 \end{equation}
For $w: D \lra D$ as in (\ref{w0}), we may define, locally
	\begin{align} \label{uscaled2}
		\frac{1}{\tau} w_\sigma (z) &=  
			\frac{ \lambda \sigma}{\tau} z+ \frac{v_\sigma (z)}{\tau}.
	\end{align}
By (\ref{scaling}) we have
	\begin{lemma} \label{realscalednorm}
If $w(z)  =  \lambda z + v(z)$, then
		\begin{align*}
			\braces[c,  k, \g, R/\tau]{ \frac{1}{\sigma} w_\tau } &= 
			\frac{\tau^c}{\sigma} \braces[c,  k, \g, R]{ w}
		\end{align*}
	\end{lemma}

We will apply this lemma directly to the $w_k$ in (\ref{wk}).  For these maps $G_k =  e^{2 \mu_k} \absv{ w }^{2(\al - 1)} \absv{ dw }^2$, and the map $ \frac{1}{\sigma} (w_k)_\tau (z)$ is an expression in normalized conformal coordinates (see (\ref{conemet})) of the map
 \begin{equation} \label{scaledmapp}
		w_k :  (D, g_\al/ \tau^{2 \al}) \lra (\Sigma, G_k/ \sigma^{2 \al}),
 \end{equation}
meaning simply that since
	\begin{align*}
		g_\al/ \tau^{2 \al} = \frac{\absv{z}^{2(\al - 1)}}{\tau^{2\al}} \absv{dz}^2  \ \ \& \ \ G_k/ \sigma^{2 \al} =  e^{2 \mu_k(w)}\frac{\absv{w}^{2(\al - 1)}}{\sigma^{2\al} } \absv{dw}^2,
	\end{align*}
doing $z \mapsto \tau z$ and $w \mapsto \sigma w$ gives
	\begin{align*}
		g_\al/ \tau^{2 \al} = \absv{z}^{2(\al - 1)}\absv{dz}^2  \ \ \& \ \ G_k/ \sigma^{2 \al} = e^{2 \mu_k(\sigma w)} \absv{w}^{2(\al - 1)} \absv{dw}^2.
	\end{align*}
Note that it follows immediately from (\ref{hmec}) that
(\ref{scaledmapp}) is harmonic.  Before we begin the proof of
Proposition \ref{closed}, we recall that by (\ref{harnbound}) we know
that $\lambda_k \to \lambda_0$ for some $\lambda_0 > 0$, and, by
setting $G_{k} = G_{k} / \lambda_{k}^{2 \alpha}$, we may assume
without loss of generality that $$\lambda_k \equiv 1$$ so that the normalized conformal coordinate expression is $w_{k} = z + v_{k}(z)$.

\begin{proof}[Proof of Proposition \ref{closed}]

We proceed by contradiction.  Supposing Proposition \ref{closed} false, we will produce a sequence $\sigma_k \to 0$ so that the scaled maps
	\begin{align*}
		 \frac{ 1}{\sigma_k } w_{k, \sigma_k },
	\end{align*}
converge to a harmonic map $w_\infty : C_\al \lra C_\al$ satisfying the assumptions but not the conclusion of Lemma \ref{globalclass}.  This contradiction proves the lemma.

If (\ref{bound2}) does not hold then for all $\sigma, C > 0$ there is a $k$ such that
	\begin{align} \label{tocontradict}
		\braces[1 + \e, 2, \g, \sigma]{  w_k }> C .
	\end{align}
Thus, for every $l \in \N$ there is a $k_l$ such that $\braces[1 + \e,
2, \g, 1 / l]{ w_{k_l} } > l^\e$, and passing to a subsequence, we
assume that $\braces[1 + \e, 2, \g, 1 / k]{ w_{k} } > k^\e$.
Since the L.H.S. is monotone decreasing in the radius, for each $k$ there is a number $ \sigma_k < 1/k$ such that
$\braces[1 + \e, 2, \g, \sigma_k ]{ w_{k} }  =  \sigma_k^{-\e}$.
or, by the scaling properties in Lemma \ref{realscalednorm}
	\begin{align} \label{unibound}
		\fbox{ $
		\displaystyle{
		\braces[{1 + \e}, 2, \g, 1 ]{ \frac{1}{\sigma_k} w_{k, \sigma_k} } = 1
		}
		$ }
	\end{align}
By the remarks immediately preceding the proof, this map, which is the normalized coordinate expression of the map
	\begin{align} \label{scaledmap}
		w_k : (D, g / \sigma_k^{2\al}) \lra (\Sigma , G_k/ \sigma_k^{2\al} )
	\end{align}
is harmonic and in normalized conformal coordinates satisfies $\frac{1}{\sigma_{k} }w_{k, \sigma_{k}}(z) = z+  \frac{v_{k, \sigma_k}(z)}{\sigma_k}$.
Set
\begin{equation}
  \label{eq:1}
  \wt{v}_{k} := \frac{v_{k, \sigma_k}(z)}{\sigma_k}
\end{equation}
By the obvious identity for the energy density, $e(w_k, g / \sigma_k^2, G_k/\sigma_k^2) =  e(w_k, g, G_k)$,
the uniform bound (\ref{edb}) holds for the maps in (\ref{scaledmap}), and thus they
converge on compact subsets of $(D, g_\al / \sigma_k^{2\al})=
(D(1/\sigma_k), g_\al)$ to a map $w_\infty : C_\al \lra C_\al$. The following claim will finish the proof of Proposition \ref{closed} since it contradicts Lemma \ref{globalclass}.
	\begin{claim} \label{toconeClaim}
		The map $w_{\infty}$ is in Form \ref{uform}, with
 \begin{equation} \label{coordexpresh0}
		w_{\infty}(z) = z + v_{\infty}(z)
 \end{equation}
and $\wt{v}_k \to v_{\infty} \in  r^{1 + \e} C^{2, \g}_b(D) $, for $\e' , \g'  > 0$ sufficiently small.  Furthermore 
		\begin{align} \label{nonzero}
			v_{\infty}(z)
			&\not \equiv
			0
		\end{align}
	\end{claim}
To prove the claim, set
 \begin{equation*}
		\wt{G}_k = G_k / \sigma_k^{2\al}
\end{equation*}
Using the elliptic theory of $b$-differential operators from section
\ref{bcalcsect}, we will prove that, for $ \e' < \e$ as above and $0< \g' < \g <1 $ we have the inequality
 \begin{equation} \label{ellestimate}
		\norm[{1 + \e} , 2, \g, 1]{ \wt{v}_{k} }  \leq C \norm[{1 + \e'} , 2, \g', 2]{ \wt{v}_{k} }.
 \end{equation}
Note the shift in regularity and the fact that the norm on the right is on a ball of larger radius than the norm on the left.  The L.H.S. is bounded from below by (\ref{unibound}), so the right hand side is also bounded from below.  By the compact containment (\ref{rellich}), the $v_{k , \sigma_k}$ converge strongly in the norm on the left, thus they converge to a non-zero function, i.e.\ (\ref{nonzero}) holds.

Thus it remains to prove the estimate (\ref{ellestimate}).  To do so, we apply Taylor's theorem to the Harmonic map operator $\tau$ around the map $id: (D, g_\al) \lra (D, \wt{G}_k)$.  Let $r^c \dot{C}^{k, \g}_b(C_\al)$ denote the set of maps $v \in r^cC^{k, \g}_b(D(2))$ which vanish on $\p D(2)$.  By section \ref{sec:linearization} we have
	\begin{align*}
		\tau(w_k, g_\al, \wt{G}_k) &= \tau(id, g_\al, \wt{G}_k) + \Ll_k \wt{v}_k + Q_k( \wt{v}_k) \\
		\Ll_k \wt{v}_k &= - Q_k( \wt{v}_k) 
	\end{align*}
All the materlial in section \ref{bcalcsect} applies to
	\begin{align}\label{eq:lkdirichlet}
		\Ll_k : r^{1 + \e}\dot{C}^{j, \g}_b(D(2)) \lra r^{1 +  \e - 2\Aa } C_b^{j - 2, \g}(D(2)).
	\end{align}
In particular, it is Fredholm for any $\g$, $j$, and $\e$ small, and
it has an generalized inverse $\mathcal{G}_{k}$ which satisfies the
mapping properties analogous to \eqref{eq:Gcpolyhomo}.  By Lemma
\ref{pairLemma}, \eqref{eq:lkdirichlet} is injective, so $\Gg_k \Ll_k = I$.
From (\ref{qbound}), for $\e' < \e$ as above, the inequality $\norm[1
+ 2\e' - 2\al, j, \g, 2]{Q_k(\wt{v}_k)} \leq C\norm[1 + \e' , j + 1,
\g, 2]{ \wt{v}_k }$.  Let $\chi(r)$ be a cutoff
function that is $1$ on the $D$ and supported in $D(2)$.  Then we have
	\begin{align*}
		\Gg_k \chi Q_k(\wt{v}_k) &= -  \Gg_k \chi \Ll_k(\wt{v}_k)  = - \Gg_k[\Ll_k, \chi]\wt{v}_k  - \chi \wt{v}_k,
	\end{align*}
so	
 \begin{equation*}
		\chi \wt{v}_k = \Gg_k \chi Q_k(\wt{v}_k) - \Gg_k[\Ll_k, \chi]\wt{v}_k.
\end{equation*}
Since $[\Ll_k, \chi]$ is the zero operator near the cone point, so satisfies 
 \begin{equation*}
	[\Ll_k, \chi]: r^{ 1 + \e } C_b^{j , \g}(D(2)) \lra r^{N}C^{j-1, \g}_b(D(2))
\end{equation*}
for any $N > 0$.
Tracing through all of the boundedness properties above, we get
	\begin{align*}
		\norm[1 + 2 \e', 2, \g, 1]{\wt{v}_k} &\leq \norm[1 + 2 \e', 2, \g, 2]{ \chi \wt{v}_k} \\
		&= \norm[1 + 2 \e', 2, \g, 2]{\Gg_k \chi Q_k(\wt{v}_k)} - \norm[ 1 + 2 \e', 2 , \g, 2]{\Gg_k[\Ll_k, \chi]\wt{v}_k} \\
		&\leq 
			\norm[1 + \e', 1, \g, 2]{\wt{v}_k} \\
	\end{align*}
But $\e'$ is an arbitrary positive number less than $\e$, so
regardless of the $\e$ present in (\ref{unibound}), we can choose $\e
> \e' > \e / 2$ and \eqref{ellestimate} is proven.

This completes the proof of Claim \ref{toconeClaim} and thus the proof of Proposition \ref{closed}.
\end{proof}



%
%
%
%
%
%
%
%
%
%
%
%
%
%

\section{Cone angle $\pi$} \label{coneanglepisect}

We now discuss the case
 \begin{equation*}
		\pe \neq \varnothing,			
\end{equation*}
which involves only minor modifications of the above arguments.

Let $id = u_0: (\Sp, g) \lra (\Sp, G)$ be energy minimizing and fix $p
\in \pe$.  Let $\phi_1: D \lra (\Sp, g)$ and $\phi_2: D \lra (\Sp,
G)$, so that the $\phi$ are conformal and $\phi_i(0) = p$.  As usual,
pick conformal coordinates $z$ and $w$, resp$.$.  The double cover $f: D
\lra D$ with $f(z) = z^2$ can be used to pull back $G$ to a metric
$\ov{G} : =f^* \phi_2^* G$ on $D$ with (not necessarily smooth) cone
angle $2 \pi$.  The map $w = \phi_2 \circ u_0 \circ \phi_1^{-1}$ lifts
to a harmonic map
	\begin{align} \label{w}
	\xymatrix{
		   D  \ar[d]^f \ar[r]^{\wt{w}} & (D, \ov{G}) \ar[d]^f \\
		   D   \ar[r]^{w} & (D, \phi_2^* G)  
	}
	\end{align}
Since $\ov{G}$ has cone angle $2\pi$, by section \ref{metricssect}, in
conformal coordinates $v$ we can write
 \begin{equation*}
		\ov{G} = e^{2 \mu} \absv{d\wt{w}}^2.
\end{equation*}
Generically, the metric $\ov{G}$ is not smooth near $0$.

The reason we treat this case separately is that the form of these harmonic maps near $\pe$ is different than Form \ref{uform}.  We have
\begin{form}[$\pe \neq \varnothing$] \label{uformpi}
We say that $u: (\Sp, g) \lra (\Sp, G)$ is in Form \ref{uformpi} (with respect to $g$ and $G$) if 
\begin{enumerate}
	\item $ u \in \Dsp$
	
	\item $ u \sim id$

	\item  Near $p \in \pe$, if $w$ is defined as in (\ref{w}), then
 \begin{equation*}
			\wt{w}({\wt{z}})  =  a z  + b \ov{\wt{z}} + v({\wt{z}})
\end{equation*}
where $a, b \in \C$, $a \neq 0$, and 
 \begin{equation*}
			v \in r^{1 + \e} C^{2,\g}_b(D(R)),
\end{equation*}
for some sufficiently small $\e  > 0$.  
\end{enumerate}
\end{form}
It is easy to check that Lemma \ref{hopfpoles} holds, i.e.\ that harmonic maps in Form \ref{uformpi} have Hopf differentials that are holomorphic with at worst simple poles at $p$.  In fact, since $\Phi(\wt{w})$ is invariant under the deck transformation, we can just compute $\Phi(\wt{w})$ and use $\Phi(w) = f_* \Phi(\wt{w})$.
	\begin{align*}
		\Phi(\wt{w}) &= e^{2\mu} v_{\wt{z}} \ov{v}_{\wt{z}} \\
		&= (a \ov{b} + \mc{O}(\absv{\wt{z}})) d\wt{z}^2.
	\end{align*}
By holomorphicity and the invariance of $\Phi$ under the deck transformation, 
 \begin{equation*}
		\Phi(\wt{w}) = (a \ov{b} + g(\wt{z}^2)) d\wt{z}^2,
\end{equation*}
where $g$ is a holomorphic function with $g(0) = 0$.  Thus
	\begin{align*}
		\Phi(w) &= f_* \Phi(\wt{w}) \\
		&= \frac{1}{4}\lp \frac{ a \ov{b}}{z} + g(z) \rp dz^2 .
	\end{align*}
Thus, we have shown more than Lemma \ref{hopfpoles}, namely,
\begin{lemma} \label{formreduce}
Suppose $u : (\Sp, g) \lra (\Sp, G)$ is harmonic and is in Form \ref{uformpi} with $a \neq 0$, then $u$ is actually in Form \ref{uform} if and only if $\Phi(u)$ extends holomorphically over $\pe$.
\end{lemma}
\noindent it follows trivially that \emph{all of the results of section
  \ref{uniquesect} hold}, since the only way Form \ref{uform} came
into the picture was in proving Lemma \ref{hopfpoles}.

Given a harmonic $u_0$ in Form \ref{uformpi}, the space $\Bb^{1 +
  \e}(u_0)$ is defined so that $u \in \Bb^{1 + \e}(u_0)$ if and only
if near $q \in \pp$, $u - u_0 \in r^{1 + \e}C^{2, \g}_b$.  For $p \in
\pe$, writing the lift of $w_{0}$ as
 \begin{equation*}
		\wt{w}_{0}(\wt{z})  =  a_0 \wt{z}  + b_0 \ov{\wt{z}} + v_0(\wt{z})
\end{equation*}
we see that
 \begin{equation*}
		u \in \Bb^{1 + \e}(u_0) \implies \wt{w}(z) = a_0 \wt{z}  + b_0 \ov{\wt{z}} + v(\wt{z}) \mbox{ for some } v \in r^{1 + \e} C^{2,\g}_b
\end{equation*}
where $w$ is the localized lift of $u$ from \eqref{w}.  As above, we
allow the tension field operator $\tau$ to act on a space of geometric
perturbations.  Near $p \in \pe$ they can be described simply; given $\lambda = (\lambda_1, \lambda_2) \in \C^2$, consider the (locally defined) map
 \begin{equation} \label{dnearpe}\begin{split}
		\wt{w}(\wt{z}) &= (a_0 + \lambda_1) \wt{z}  + (b_0 + \lambda_2) \ov{\wt{z}} + v(\wt{z}).
\end{split}\end{equation}
Following the above arguments we let $\Vv'$ be a space of automorphisms which look like \eqref{eq:cplxmult} near $q \in \pp - \pe$ and which lift to look like (\ref{dnearpe}) near each $p \in \pe$.

\subsection{$\mc{H}(\qq)$ is open ($\pe \neq \varnothing$)}

We state and sketch the proof of the main lemma
	\begin{lemma}
The tension field operator $\tau$ acting on $\Bb^{1 + \e}(u_0) \circ \Vv' \circ \mc{T}_{> \pi}$ is $C^1$.  Its linearization $L = D_u \tau$ is bounded as a map
 \begin{equation*}
		L : r^{1 + \e}\Xx^{2, \g}_b \oplus T_{id}\Vv' \oplus T_{id} \mc{T}_{> \pi} \lra r^{1 + \e - 2\al}\Xx^{0, \g}_b,
\end{equation*}
and is transverse to $(\CKK)^\perp$.
	\end{lemma}
By Remark \ref{rm:coneanglepi}, the lemma implies the openness
statement exactly as it did in the case $\pe = \varnothing$.

The map 	
	\begin{align} \label{mappi}
		L : r^{1 + \e}\Xx^{2, \g}_b \lra r^{1 + \e - 2\al}\Xx^{0, \g}_b
	\end{align}	
is Fredholm for small $\e$, as is $L : r^{1 - \e}\Xx^{2, \g}_b \lra r^{1 - \e  - 2 \al}\Xx^{2, \g}_b $.  The latter map has cokernel $\mc{K} = \krn{ 1 + \e - 2\Aa}$.  The cokernel of (\ref{mappi}) can again be written as $\wt{W} \oplus \mc{K}$ where $\wt{W}$ consists of vectors $\psi \in r^{1 - \e}\Xx$ with $L\psi \in r^{1 + \e - 2\al}\Xx$, and again such vectors have expansions determined by the indicial roots.  Near $p \in \pe$, using the lifts at the beginning of this section makes the asymptotics easy to calculate.  Note that
 \begin{equation*}
		H(u, g, G) = (\phi_1)_* f_* H(\wt{w}, D, \ov{G}),
\end{equation*}
where $\wt{z}$ is defined by (\ref{w}), which immediately implies
(initially on $\Xx^\infty$), that near $p \in \pe$,
 \begin{equation*}
		\Ll_{u_0, g, G} = (\phi_1)_* f_* D_{\wt{w}} H \rvert_{\wt{w}_0, D, \ov{G}}
\end{equation*}
If $\ov{L} := D_{\wt{z}} H \rvert_{\wt{w}_0, D, \ov{G}}$, then setting
$\wt{\ov{L}} = (\absv{\wt{z}}^2 / 4) \ov{L}$ we have
		\begin{align*}
		\wt{\ov{L}} &=\lp \wt{z}\p_{\wt{z}} \rp \lp  \overline{\wt{z}} \p_{\overline{\wt{z}}} \rp +
		E(\wt{z}).
	\end{align*}
Any $\psi \in r^{1 + \e - 2\al}\Xx$ is in $r^\e C^{2, \g}_b$ near $p \in \pe$ since $\al_p = 1/2$.  It is now easy to see that a solution $L \psi = 0$ has lift
 \begin{equation}\label{eq:pisjfs}
   \begin{split}
     \ov{\psi}(\wt{z}) &= \lambda_1 \wt{z} + \lambda_2
     \ov{\wt{z}} + \ov{\psi}' \\
     \ov{\psi}' &\in r^{1 + \epsilon} C^{2, \gamma}_{b}
   \end{split}
\end{equation}
near such a $p$, and this shows that
 \begin{equation*}
		L \lp  r^{1 + \e}\Xx^{2, \g}_b \oplus \Vv' \rp \oplus \mc{K} =  r^{1 + \e - 2\Aa} \Xx^{0, \g}_b
\end{equation*}
In analogy with Lemma \ref{stfsjfs}, we have that, near $p \in \pe$, elements of $\mc{K}$
look like (\ref{eq:pisjfs}).  All of the material in section \ref{pertsection1} follows after
replacing $\Vv$ by $\Vv'$.

\subsection{$\mc{H}(\qq)$ is closed ($\pe \neq \varnothing$).}

Since by Lemma \ref{formreduce} solutions to (\ref{hme2}) are in Form \ref{uform}, the proof the $\mc{H}_2$ is closed is identical to the proof in the $\pe = \varnothing$ case.

For (\ref{hme2}), the proof again requires only minor modifications.
The main difference is the following; consider a sequence of maps
$u_k$ and converging metrics metrics $G_k \to \horrible$ and $g_{k}
\to g_{0}$ as in the statement of Theorem \ref{closedthm}.  In the
same way as in the $\pe = \varnothing$ case, we reduce to the  local
analysis of the $u_{k}$ near a cone point $p \in \pe$.  Here we assume
that the maps are in Form \ref{uformpi} with \begin{align} \label{abigb} \absv{a_k} > \absv{b_k}. \end{align}
By analogy with the treatment of the $\lambda_k$ in the $\pe =
\varnothing$ case, we want want to show that the condition
\eqref{abigb} persists in the limit, and for this we need some
uniform control of the $a_k, b_k$.  In fact, we claim that
	\begin{align*}
		\absv{a_k} \leq c \mbox{ and } \absv{a_k} - \absv{b_k} \geq c > 0,
	\end{align*}
for some uniform constant $c$.  As in the previous case, the $a_k$ and
$b_k$ are relate to the energy density $e_k$ and the function $h_k$
defined in (\ref{hell}).  We have the following
 \begin{equation} \label{piclosedineq}\begin{split}
		h_k(z) &\sim  \absv{a_k}^{2\al} \absv{1 + \frac{b_k}{a_k} e^{- 2 i \theta}}^{2 (\al - 1)} + f_2 \\	
		e_k(z) &\sim \absv{a_k}^{2 (\al - 1)}  \absv{1 + \frac{b_k}{a_k} e^{- 2 i \theta}}^{2 (\al - 1)} (\absv{a_k}^2 + \absv{b_k}^2 ) + f_1\\
		f_i &\in r^{\e}C^{2, \g}_b
\end{split}\end{equation}
If follows that the $e_k$ are uniformly bounded, since they are still subsolutions.  In the second line, choosing $\theta$ so that $\frac{b_k}{a_k} e^{- 2 i \theta} = \absv{\frac{b_k}{a_k}} < 1$ and using the uniform bound $e_k < c$ gives $\absv{a_k} < c$.  We apply Lemma \ref{genharn} to the $\log h_k / \delta_k$ for $\delta_k$ defined as in (\ref{deltak}), noting that the hypotheses are satisfied by (\ref{piclosedineq}).  As above, this leads to the lower bound
	\begin{align*}
		\inf_{z \in D - 0} h_k(z) \geq c > 0,
	\end{align*}
so choosing $\theta$ such that 
	\begin{align*}
		1 + \frac{b_k}{a_k} e^{- 2 i \theta} = 1  - \absv{ \frac{b_k}{a_k} }
	\end{align*}
we get that
	\begin{align*}
		\absv{a_k} - \absv{b_k} \geq c > 0
	\end{align*}

The rest of the argument proceeds as in the $\pe = \varnothing$ case,
with $a_{k}$ playing the role of $\lambda_{k}$.  Assuming the same
type of blow-up near $p \in \pe$, and rescaling in the exact same way,
on the local double cover a harmonic map of $w_{\infty} : \C \lra \C$
results with $w_{\infty} = a_{\infty}z + b_{\infty}\ov{z} +
v_{\infty}$ with $v_{\infty} \in r^{1 + \epsilon}C^{2, \gamma}_{b}$
not identically zero.  This is a contradiction by the following the
argument from \cite{d}, which we outline briefly; an orientation
preserving harmonic mapping map of $\C$ can be written, globally, as a
sum $f + \ov{g}$ where $f$ and $g$ are holomorphic.  The ratio
$\p_{z}g / \p_{z}f$ is bounded by the orientation preserving property and
is clearly holomorphic, hence constant.  Integrating proves the statement.



%

%

\section{$H$ is continuously differentiable} \label{seccont}

We now discuss in detail the map (\ref{mapquest}).  To study its
properties we trivialize the bundle $\EB \lra \Bb^{1 + \e}_{2,\g}(u_0)
\circ \mc{C} \times  \Mm^*_{2,\g, \nu}(g_{0}, \pp, \Aa) \times
\Mm^*_{2,\g, \nu}(\horrible, \pp, \Aa)$, and use the trivializing map
to define the topology of $\EB$.  We define a map
        \begin{equation} \label{themap}
		\Xi : \EB \lra \Xx^{1 + \e - 2\Aa}_{0, \g}(u_0)
	\end{equation}
as follows.  Let $((u \circ C, g, G), \psi) \in \EB$ where $\psi \in
\Xx^{1 + \e - 2\Aa}_{0, \g}(u \circ C)$.  By the definition of $u \in
\Bb^{1 + \e}(u_0)$, there is a unique $\wt{\psi} \in \Xx^{1 +
  \e}(u_0)$ so that
	\begin{equation*}
		u = \mbox{exp}_{u_0}(\wt{\psi})
	\end{equation*}
Assuming for the moment that $C = id$, let $\Xi((u, g, G), \psi) = \Xi_u(\psi) $ where
	\ben \label{xiu}
		 \Xi_u(\psi) &:=& \begin{array}{c}
		 \mbox{ parallel translation of $\psi$ along }  \\
		 \gamma_t :=  \exp_{u_0}(t\wt{\psi}) \mbox{ from $t=1$ to $t = 0$.}
		 \end{array}
	\een
In general, motivated by pointwise conformal invariance of $\tau$ (
see \eqref{tauconfinv}), define $\Xi((u \circ C, g, G), \psi) = \Xi_u(
\psi \circ C^{-1})$.  Obviously, $\Xi$ is an isomorphism on each
fiber, and we endow $\EB$ with the pullback topology induced by
$\Sigma$.  Thus a section $\sigma$ of $\EB$ is $C^1$ if and only if
$\Xi \circ \sigma$ is $C^1$.  The purpose of this section is to prove
Proposition \ref{Hcontinuity}, which states that if $g_{0}$ and
$\horrible$ satisfy Assumption \ref{assump}, then the map \ref{themap}
is $C^1$.

We reduce the proposition to a computation in local coordinates.  Near
$p \in \pp - \pe$, we heave
	\begin{align*}
		H(u \circ C, g, G) &= \Xi \lp \tau(u , C^* g, G) \rp \\
		&= \Xi_u \lp \tau^i (u, C^* g, G) \p_i \rp \\
		&= \tau^i (u, C^* g, G) {\lp \Xi_u \rp_i}^j \p_j,
	\end{align*}
where ${\lp \Xi_u \rp_i}^j$ is the local coordinate expression for parallel translation of $\p_\al$ along the path in (\ref{xiu}).

Since $(u_0, g_0, \horrible)$ solves (\ref{hme2}) we have $\tau(u_0, g_0, \horrible) = 0$, and since $G$ and $\horrible$ have the same the same conformal coordinates near $p \in \pp$, we have
	\begin{equation} \label{metricscont}
	\begin{split}
		\horrible &=  \rho_0(u) \absv{du}^2 = c_0 e^{2 \mu_0} \absv{u}^{2(\al - 1)} \absv{du}^2 \\
		G &= \rho(u) \absv{du}^2 = c e^{2 \mu} \absv{u}^{2(\al - 1)} \absv{du}^2
	\end{split}
	\end{equation}
and
\begin{align*}
			g_{0} &=  \sigma_0(z) \absv{dz}^2 = c_0 e^{2 \lambda_0} \absv{z}^{2(\al - 1)} \absv{dz}^2 \\
		g &= \sigma(z) \absv{dz}^2 = c e^{2 \lambda} \absv{z}^{2(\al - 1)} \absv{dz}^2.
              \end{align*}
Recall that, in local coordinates
	\begin{equation*}
		\tau(u,g,G)(z) =  \tau(u, g, G) = \frac{4}{ \sigma} \lp u_{z \overline{z}} + \frac{\p \log \rho}{\p u} u_z u_{\overline{z}}  \rp
	\end{equation*}
By (\ref{metricscont}),
	\begin{equation*}
		\frac{\p \log \rho(u_0)}{\p u} = 2 \frac{ \p \mu}{ \p u} + \frac{ \al - 1}{u_0} .
	\end{equation*}
Thus we have
\begin{equation}\label{eq:11}
		 \tau^i (u, C^* g, G) \in \frac{4}{ \sigma} r^{ \e - 1 } C^{0, \g}_b = r^{1 + \e - 2\Aa}C^{0, g}_b
               \end{equation}

Near $p \in \pe$, recall that $C = D'_{\lambda_1, \lambda_2} \circ T_{w}$ where
$T_{w}(z) = z - w$ and $D'_{\lambda_1, \lambda_2}(\wt{z}) = \lambda_{1}\wt{z} + \lambda_{2}\ov{\wt{z}}$, for
$\wt{z}^{2} = z$ and $\wt{u}^{2} = u$ coordinates on the local double cover.  Again we
have, locally
	\begin{align*}
		H(u \circ D_{\lambda_1, \lambda_2} \circ T_{w}, g, G) &= \Xi \lp
                \tau(u \circ D_{\lambda_1, \lambda_2} , T_{w}^* g, G) \rp \\
		&= \Xi_u \lp \tau^i (u \circ D_{\lambda_1, \lambda_2}, T_{w}^* g, G) \p_i \rp \\
		&= \tau^i (u \circ D_{\lambda_1, \lambda_2}, T_{w}^* g, G) {\lp \Xi_u \rp_i}^j \p_j,
	\end{align*}
The local computation of $\tau$ can now be done in the lifted
coordinates $\wt{z}$, where $\wt{u} = a \wt{z} + b \ov{\wt{z}} + v$ for $v \in r^{1 +
  \epsilon}C^{2, \gamma}_{b}$
coordinates.  The pulled back tension field is
	\begin{equation*}
                \frac{4}{
                  \sigma} \lp \wt{u}_{\wt{z} \overline{\wt{z}}} + \frac{\p
                  \log \rho}{\p \wt{u}} ((a + \lambda_{1} ) +
                \wt{u}_{\wt{z}}) ((b + \lambda_{2}) + \wt{u}_{\overline{\wt{z}}})
                  \rp
	\end{equation*}
So by $\frac{\p \log \rho(\wt{u}_0)}{\p \wt{u}} = 2 \frac{ \p
                  \wt{\mu}}{ \p \wt{u}}$, we have $\tau^{i} \in
                r^{\epsilon}C^{0, \gamma}_{b}$, which is \eqref{eq:11}
                in this context.

As for the expression ${\lp \Xi_u \rp_i}^j \p_j$, a simple exercise in
ODEs shows that (if we assume $u - z \in r^{1 + \e}C^{2,
\gamma}_{b}$), then ${\lp \Xi_u \rp_i}^j \p_j - \p_{i} \in r^{1 + \e}C^{0,
\gamma}_{b}$.  This immediately implies

Thus we have established

	\begin{lemma} \label{c1}
		If $(u_0, g_0, \horrible)$ satisfy (\ref{hme2}) and $u_0$ is in Form \ref{uform}, then
			\be
				\tau: \Bb^{1 + \e}_{2,\g}(u_0) \circ \Vv \times \Mm^*_{2,\g, \nu}(\horrible, \pp, \Aa) &\lra& \EB \\
				(u \circ C, g, G) &\lra&   \tau(u
                                \circ C, g, G)
			\ee
is $C^1$ near $u_0$.
	\end{lemma}



\newpage

\printnomenclature

\end{document}